\documentclass[12pt]{article}
\usepackage[]{amsmath,amssymb}
\usepackage{amscd}
\usepackage{latexsym}
\usepackage{cite}
\usepackage{amsthm}

\newtheorem{definition}{Definition}[section]
\newtheorem{theorem}[definition]{Theorem}
\newtheorem{lemma}[definition]{Lemma}
\newtheorem{corollary}[definition]{Corollary}
\newtheorem{remark}[definition]{Remark}
\newtheorem{example}[definition]{Example}

\newtheorem{proposition}[definition]{Proposition}

\begin{document}
\title{\bf 
The alternating central extension \\
 of the $q$-Onsager algebra}
\author{
Paul Terwilliger 
}
\date{}

\maketitle
\begin{abstract}
The $q$-Onsager algebra $O_q$ is presented by two generators $W_0$, $W_1$ and two relations, called the $q$-Dolan/Grady relations.
Recently Baseilhac and Koizumi introduced a current algebra $\mathcal A_q$ for $O_q$. Soon afterwards, Baseilhac and Shigechi gave a presentation of
$\mathcal A_q$ by generators and relations. We show that these generators give a PBW basis for $\mathcal A_q$. Using this PBW basis, we show that the algebra $\mathcal A_q$ is isomorphic to 
$O_q \otimes \mathbb F \lbrack z_1, z_2, \ldots \rbrack$, where $\mathbb F$ is the ground field and $\lbrace z_n \rbrace_{n=1}^\infty $ are mutually
commuting indeterminates. Recall the positive part $U^+_q$ of the quantized enveloping algebra $U_q(\widehat{\mathfrak{sl}}_2)$.
Our results show that $O_q$ is related to $\mathcal A_q$  in the same way that $U^+_q$ is related to the alternating central extension of $U^+_q$.
For this reason, we propose to call $\mathcal A_q$ the alternating central extension of $O_q$.

\bigskip


 \end{abstract}
 
 \section{Introduction}
 Our topic involves  the $q$-Onsager algebra $O_q$ \cite{bas1,qSerre}. This infinite-dimensional associative algebra is presented by 
  two generators $W_0$, $W_1$ and two relations called the
$q$-Dolan/Grady relations; see Definition \ref{def:U}
below. 
\medskip

\noindent
The algebra $O_q$ first appeared in algebraic combinatorics, in the study of 
$Q$-polynomial distance-regular graphs;
to our knowledge the $q$-Dolan/Grady relations first appeared in \cite[Lemma~5.4]{tersub3}.
The algebra $O_q$ belongs to a family of algebras called  tridiagonal algebras 
\cite[Definition~3.9]{qSerre}. One can think of $O_q$ as the ``most general''  tridiagonal algebra \cite[Section~1.2]{ItoTerAug}. 
There is an object in linear algebra called a tridiagonal pair
\cite[Definition~1.1]{TD00}.
A finite-dimensional irreducible $O_q$-module is essentially the same thing as a tridiagonal pair of $q$-Racah type \cite[Theorem~3.10]{qSerre}. 
These tridiagonal pairs are classified up to isomorphism in \cite[Theorem~3.3]{ItoTer}. See \cite{ito} for a survey about $O_q$ and tridiagonal pairs.
\medskip

\noindent  Over the past two decades the algebra $O_q$ has appeared in a wide variety of contexts.
 For instance, $O_q$ is used 
 to study boundary 
integrable systems 
\cite{
bas2,
bas1,
bas8,
basXXZ,
basBel,
BK05,
bas4,
basKoi,
basnc}.
The algebra $O_q$  can be realized as a left or right coideal subalgebra of the quantized enveloping algebra
$U_q(\widehat{\mathfrak{sl}}_2)$; see \cite{bas8, basXXZ, kolb}. The algebra $O_q$ is the simplest example of a quantum symmetric pair coideal subalgebra 
of affine type \cite[Example~7.6]{kolb}. 
A Drinfeld type presentation of $O_q$ is obtained in \cite{LuWang}, and this  is used in \cite{LRW} to realize $O_q$ as an $\iota$Hall algebra of the projective line.
There is an algebra  homomorphism from  $O_q$ into the universal Askey-Wilson algebra 
\cite[Sections~9,10]{uaw}. See  \cite{BK, lusztigaut,  pbw,  z2z2z2, conj} for more information about $O_q$.

\medskip

 \noindent
In \cite{BK05} Baseilhac and Koizumi introduce a current algebra $\mathcal A_q$ for $O_q$, in order to solve boundary integrable systems with hidden symmetries.
In \cite[Definition~3.1]{basnc} Baseilhac and Shigechi give a presentation of $\mathcal A_q$ by generators and relations. 
The generators are denoted
\begin{align}
\label{eq:AltGen}
 \lbrace \mathcal W_{-k}\rbrace_{k\in \mathbb N}, \qquad \lbrace \mathcal W_{k+1}\rbrace_{k\in \mathbb N}, \qquad  
 \lbrace \mathcal G_{k+1} \rbrace_{k\in \mathbb N},
\qquad
\lbrace  \mathcal {\tilde G}_{k+1} \rbrace_{k \in \mathbb N}
\end{align}
and the relations are given in 
\eqref{eq:3p1}--\eqref{eq:3p11} below.
The relationship between $\mathcal A_q$ and $O_q$ has been the subject of intense study.  There are some partial results and conjectures in \cite{basBel}, and some conjectures in \cite{z2z2z2}.
The results of \cite{basBel} are summarized as follows.
By \cite[line (3.7)]{basBel} the elements $\mathcal W_0$, $\mathcal W_1$  satisfy the $q$-Dolan/Grady relations.
It is shown in \cite[Section~3]{basBel} that $\mathcal A_q$ is generated by $\mathcal W_0$, $\mathcal W_1$ together with the
central elements $\lbrace \Delta_n \rbrace_{n=1}^\infty$ defined in
 \cite[Lemma~2.1]{basBel}.
In \cite[Section~3]{basBel} a quotient algebra of $\mathcal A_q$ is obtained
by sending $\Delta_n$ to a scalar for all $n\geq 1$. The construction yields an algebra homomorphism
 $\Psi$ from $O_q$ onto this quotient.
According to \cite[Conjecture~2]{basBel} the map $\Psi$ is an isomorphism. In \cite[Conjecture~4.6]{z2z2z2} the above material is considered from a different point of view.
We will state the conjecture after defining a few terms. Let $\langle \mathcal W_0, \mathcal W_1 \rangle$ denote
the subalgebra of $\mathcal A_q$ generated by $\mathcal W_0, \mathcal W_1$.
Let $\mathcal Z$ denote the center of $\mathcal A_q$. Let $\lbrace z_n \rbrace_{n=1}^\infty$ denote mutually commuting indeterminates.
Consider the polynomial algebra $\mathbb F \lbrack z_1, z_2,\ldots \rbrack$ where $\mathbb F$ is the ground field. 
According to \cite[Conjecture~4.6]{z2z2z2}, (i)
there exists an algebra isomorphism
 $\mathbb F\lbrack z_1, z_2,\ldots \rbrack \to \mathcal Z$
 that sends $z_n \mapsto \Delta_n$ for $n \geq 1$;
(ii) there exists an algebra isomorphism
$O_q \to \langle \mathcal W_0, \mathcal W_1\rangle$ 
that sends $W_0\mapsto \mathcal W_0$ and
$W_1\mapsto \mathcal W_1$;
(iii) the multiplication map 
\begin{align*}
\langle \mathcal W_0, \mathcal W_1\rangle 
\otimes
\mathcal Z
 & \to   \mathcal A_q
\\
 w \otimes z  &\mapsto  wz
 \end{align*}
 is an isomorphism of algebras.
As noted below \cite[Conjecture~4.6]{z2z2z2}, a proof of (i)--(iii) yields
an algebra
isomorphism
$ O_q \otimes \mathbb F \lbrack z_1,
z_2,\ldots \rbrack
\to
\mathcal A_q 
$.
\medskip

 \noindent The results of the present paper imply that \cite[Conjecture~2]{basBel} and \cite[Conjecture~4.6]{z2z2z2} are correct, provided 
  that ${\rm Char}(\mathbb F) \not=2$. The difficulty with ${\rm Char}(\mathbb F) =2$ is not serious; we overcome it  by replacing the central elements $\lbrace \Delta_n\rbrace_{n=1}^\infty$
  with slightly different central  elements $\lbrace \mathcal Z_n\rbrace_{n=1}^\infty$. Below we list our main results, which hold without restriction on $\mathbb F$.
  \begin{enumerate}
  \item[$\bullet$] a PBW basis for $\mathcal A_q$ is obtained by the generators \eqref{eq:AltGen} in any linear order $<$ such that
\begin{align*}
\mathcal G_{i+1} < \mathcal W_{-j} < \mathcal W_{k+1} < \mathcal {\tilde G}_{\ell+1}\qquad \qquad (i,j,k, \ell \in \mathbb N).
\end{align*}
 \item[$\bullet$] 
There exists an algebra isomorphism
 $\mathbb F\lbrack z_1, z_2,\ldots \rbrack \to \mathcal Z$
 that sends $z_n \mapsto \mathcal Z_n$ for $n \geq 1$.
 \item[$\bullet$] 
There exists an algebra isomorphism
$O_q \to \langle \mathcal W_0, \mathcal W_1\rangle$ 
that sends $W_0\mapsto \mathcal W_0$ and
$W_1\mapsto \mathcal W_1$.
\item[$\bullet$] 
The multiplication map 
\begin{align*}
\langle \mathcal W_0, \mathcal W_1\rangle 
\otimes
\mathcal Z
 & \to   \mathcal A_q
\\
 w \otimes z  &\mapsto  wz
 \end{align*}
 is an isomorphism of algebras.
 \item[$\bullet$] 
There exists an algebra isomorphism $ O_q \otimes \mathbb F \lbrack z_1, z_2,\ldots \rbrack \to \mathcal A_q$ that sends
\begin{align*}
W_0 \otimes 1 \mapsto \mathcal W_0, \qquad \quad
W_1 \otimes 1 \mapsto \mathcal W_1, \qquad \quad 
1 \otimes z_n \mapsto \mathcal Z_n, \qquad n\geq 1.
\end{align*}
  \end{enumerate}
  \noindent We will prove the above results in the order of bullets 1,5,2,3,4; see Theorems \ref{thm:rr}, \ref{thm:2}, \ref{lem:c}, \ref{lem:ww},     \ref{lem:UZ}
  in the main body of the paper.
  \medskip

\noindent Near the beginning of the paper, we mentioned   that $O_q$ is a tridiagonal algebra. There is another tridiagonal algebra of interest, denoted $U^+_q$ and called the positive part of $U_q(\widehat{\mathfrak{sl}}_2)$ \cite{pospart}.
The  results listed above show that $O_q$ is related to $\mathcal A_q$ in the same way that $U^+_q$ is related to the alternating central extension of $U^+_q$ 
\cite{alternating},
\cite{altCE}.
For this reason, we propose to call $\mathcal A_q$ the alternating central extension of $O_q$. 
\medskip

\noindent We believe that $\mathcal A_q$ is a profound discovery and worthy of much further study. In the future we hope to investigate whether the generators \eqref{eq:AltGen} have a combinatorial interpretation for
$Q$-polynomial distance-regular graphs.
\medskip

\noindent The paper is organized as follows. Section 2 contains some preliminaries and the definition of $\mathcal A_q$. In Section 3 the algebra $\mathcal A_q$ is described using generating functions. 
The results of Section 3 are used in Sections 4, 5  to obtain some equations called reduction rules.
In Section 6 we use the reduction rules and the Bergman diamond lemma to show that the generators \eqref{eq:AltGen} give a PBW basis for $\mathcal A_q$.
In Section 7 we obtain a filtration of $\mathcal A_q$ that will help us understand how $\mathcal A_q$ is related to $O_q$.
In Section 8 we obtain some central elements $\lbrace \mathcal Z_n \rbrace_{n=1}^\infty$ in $\mathcal A_q$. The results of Sections 6, 7, 8 are used in Sections 9, 10
to describe how $\mathcal A_q$ is related to $O_q$.
In Appendix A we list many identities that hold in $\mathcal A_q$. 
Appendix B contains some data and examples that are meant to clarify Section 8.
Appendix C contains some details from our application of the Bergman diamond lemma.

\section{The algebra $\mathcal A_q$}

\noindent We now begin our formal argument. Throughout the paper, the following notational conventions are in effect.
Recall the natural numbers $\mathbb N= \lbrace 0,1,2,\ldots \rbrace$ and integers $\mathbb Z=\lbrace 0,\pm 1, \pm 2,\ldots \rbrace$. Let $\mathbb F$ denote a field.
Every vector space and tensor product discussed in this paper is over $\mathbb F$.
Every algebra discussed in this paper is associative, over $\mathbb F$, and has a multiplicative identity.  Let $\mathcal A$ denote an algebra. By an {\it automorphism} of $\mathcal A$
we mean an algebra isomorphism $\mathcal A\rightarrow \mathcal A$. The algebra $\mathcal A^{\rm opp}$ consists of the vector space $\mathcal A$ and the multiplication map $\mathcal A \times \mathcal A \rightarrow \mathcal A$, $(a,b)\to ba$.
By an {\it antiautomorphism} of $\mathcal A$ we mean an algebra isomorphism $\mathcal A \rightarrow \mathcal A^{\rm opp}$.
\medskip


 \begin{definition}\label{def:pbw}
 \rm 
(See \cite[p.~299]{damiani}.)
Let $ \mathcal A$ denote an algebra. A {\it Poincar\'e-Birkhoff-Witt} (or {\it PBW}) basis for $\mathcal A$
consists of a subset $\Omega \subseteq \mathcal A$ and a linear order $<$ on $\Omega$
such that the following is a basis for the vector space $\mathcal A$:
\begin{align*}
a_1 a_2 \cdots a_n \qquad n \in \mathbb N, \qquad a_1, a_2, \ldots, a_n \in \Omega, \qquad
a_1 \leq a_2 \leq \cdots \leq a_n.
\end{align*}
We interpret the empty product as the multiplicative identity in $\mathcal A$.
\end{definition}
\begin{definition}\label{def:gr}\rm 
 A  {\it grading} of an algebra $\mathcal A$ is a sequence  $\lbrace \mathcal A_n \rbrace_{n \in \mathbb N}$ of subspaces of $\mathcal A$ such that
(i) $1 \in \mathcal A_0$; (ii) the sum $\mathcal A = \sum_{n \in \mathbb N} \mathcal A_n$ is direct; (iii) $\mathcal A_r \mathcal A_s \subseteq \mathcal A_{r+s} $ for $r,s\in \mathbb N$.
\end{definition}
\begin{definition} \label{def:filtration} \rm (See \cite[p.~202]{carter}.)
A {\it filtration} of an algebra $\mathcal A$ is a sequence  $\lbrace \mathcal A_n \rbrace_{n \in \mathbb N}$ of subspaces of $\mathcal A$ such that
(i) $1 \in \mathcal A_0$; (ii) $\mathcal A_{n-1} \subseteq \mathcal A_n$ for $n\geq 1$; (iii) $\mathcal A = \cup_{n \in \mathbb N} \mathcal A_n$;
(iv) $\mathcal A_r \mathcal A_s \subseteq \mathcal A_{r+s} $ for $r,s\in \mathbb N$.
\end{definition}

\noindent
Fix a nonzero $q \in \mathbb F$ that is not a root of unity. Recall the notation
\begin{align*}
\lbrack n \rbrack_q = \frac{q^n-q^{-n}}{q-q^{-1}} \qquad \qquad n \in \mathbb Z.
\end{align*}
\noindent
For elements $X$ and $Y$ in any algebra, their commutator and $q$-commutator are given by
\begin{align*}
 \lbrack X,Y\rbrack=XY-YX, \qquad \qquad \lbrack X,Y\rbrack_q=
qXY-q^{-1}YX.
\end{align*}

\begin{definition}\rm
\label{def:Aq}
(See 
\cite{BK05}, \cite[Definition~3.1]{basnc}.)
Define the algebra $\mathcal A_q$
by generators
\begin{align}
\label{eq:4gens}
\lbrace \mathcal W_{-k}\rbrace_{n\in \mathbb N}, \qquad  \lbrace \mathcal  W_{k+1}\rbrace_{n\in \mathbb N},\qquad  
 \lbrace \mathcal G_{k+1}\rbrace_{n\in \mathbb N},
\qquad
\lbrace \mathcal {\tilde G}_{k+1}\rbrace_{n\in \mathbb N}
\end{align}
 and the following relations. For $k, \ell \in \mathbb N$,
\begin{align}
&
 \lbrack \mathcal W_0, \mathcal W_{k+1}\rbrack= 
\lbrack \mathcal W_{-k}, \mathcal W_{1}\rbrack=
({\mathcal{\tilde G}}_{k+1} - \mathcal G_{k+1})/(q+q^{-1}),
\label{eq:3p1}
\\
&
\lbrack \mathcal W_0, \mathcal G_{k+1}\rbrack_q= 
\lbrack {\mathcal{\tilde G}}_{k+1}, \mathcal W_{0}\rbrack_q= 
\rho  \mathcal W_{-k-1}-\rho 
 \mathcal W_{k+1},
\label{eq:3p2}
\\
&
\lbrack \mathcal G_{k+1}, \mathcal W_{1}\rbrack_q= 
\lbrack \mathcal W_{1}, {\mathcal {\tilde G}}_{k+1}\rbrack_q= 
\rho  \mathcal W_{k+2}-\rho 
 \mathcal W_{-k},
\label{eq:3p3}
\\
&
\lbrack \mathcal W_{-k}, \mathcal W_{-\ell}\rbrack=0,  \qquad 
\lbrack \mathcal W_{k+1}, \mathcal W_{\ell+1}\rbrack= 0,
\label{eq:3p4}
\\
&
\lbrack \mathcal W_{-k}, \mathcal W_{\ell+1}\rbrack+
\lbrack \mathcal W_{k+1}, \mathcal W_{-\ell}\rbrack= 0,
\label{eq:3p5}
\\
&
\lbrack \mathcal W_{-k}, \mathcal G_{\ell+1}\rbrack+
\lbrack \mathcal G_{k+1}, \mathcal W_{-\ell}\rbrack= 0,
\label{eq:3p6}
\\
&
\lbrack \mathcal W_{-k}, {\mathcal {\tilde G}}_{\ell+1}\rbrack+
\lbrack {\mathcal {\tilde G}}_{k+1}, \mathcal W_{-\ell}\rbrack= 0,
\label{eq:3p7}
\\
&
\lbrack \mathcal W_{k+1}, \mathcal G_{\ell+1}\rbrack+
\lbrack \mathcal  G_{k+1}, \mathcal W_{\ell+1}\rbrack= 0,
\label{eq:3p8}
\\
&
\lbrack \mathcal W_{k+1}, {\mathcal {\tilde G}}_{\ell+1}\rbrack+
\lbrack {\mathcal {\tilde G}}_{k+1}, \mathcal W_{\ell+1}\rbrack= 0,
\label{eq:3p9}
\\
&
\lbrack \mathcal G_{k+1}, \mathcal G_{\ell+1}\rbrack=0,
\qquad 
\lbrack {\mathcal {\tilde G}}_{k+1}, {\mathcal {\tilde G}}_{\ell+1}\rbrack= 0,
\label{eq:3p10}
\\
&
\lbrack {\mathcal {\tilde G}}_{k+1}, \mathcal G_{\ell+1}\rbrack+
\lbrack \mathcal G_{k+1}, {\mathcal {\tilde G}}_{\ell+1}\rbrack= 0.
\label{eq:3p11}
\end{align}
In the above equations $\rho = -(q^2-q^{-2})^2$. The generators 
\eqref{eq:4gens} are called {\it alternating}.
\noindent For notational convenience define
\begin{align}
{\mathcal G}_0 = -(q-q^{-1})\lbrack 2 \rbrack^2_q, \qquad \qquad 
{\mathcal {\tilde G}}_0 = -(q-q^{-1}) \lbrack 2 \rbrack^2_q.
\label{eq:GG0}
\end{align}
\end{definition}

\noindent Our first main goal in this paper is to show that a PBW basis for $\mathcal A_q$ is obtained by its alternating generators in any linear order $<$ such that
\begin{align}
\mathcal G_{i+1} < \mathcal W_{-j} < \mathcal W_{k+1} < \mathcal {\tilde G}_{\ell+1}\qquad \qquad (i,j,k, \ell \in \mathbb N).
\label{eq:order}
\end{align}
\noindent The above result is a variation on  \cite[Conjecture~4.5]{z2z2z2}. Both the above result and  \cite[Conjecture~4.5]{z2z2z2} are inspired by
\cite[Conjecture~1]{basBel}.

\section{Generating functions}

\noindent As we investigate the algebra $\mathcal A_q$, our first task is to express the defining relations \eqref{eq:3p1}--\eqref{eq:3p11} in terms of generating functions. 
Later in the paper, we will put these relations in a more useful form. As we go along, the following symmetries will be used to simplify some of our proofs.


\begin{lemma}
\label{lem:aut} {\rm (See \cite[Remark~1]{basBel}.)} There exists an automorphism $\sigma$ of $\mathcal A_q$ that sends
\begin{align*}
\mathcal W_{-k} \mapsto \mathcal W_{k+1}, \qquad
\mathcal W_{k+1} \mapsto \mathcal W_{-k}, \qquad
\mathcal G_{k+1} \mapsto \mathcal {\tilde G}_{k+1}, \qquad
\mathcal {\tilde G}_{k+1} \mapsto \mathcal G_{k+1}
\end{align*}
 for $k \in \mathbb N$. Moreover $\sigma^2 = 1$.
\end{lemma}

\begin{lemma}\label{lem:antiaut} {\rm (See \cite[Lemma~3.7]{z2z2z2}.)} There exists an antiautomorphism $\dagger$ of $\mathcal A_q$ that sends
\begin{align*}
\mathcal W_{-k} \mapsto \mathcal W_{-k}, \qquad
\mathcal W_{k+1} \mapsto \mathcal W_{k+1}, \qquad
\mathcal G_{k+1} \mapsto \mathcal {\tilde G}_{k+1}, \qquad
\mathcal {\tilde G}_{k+1} \mapsto \mathcal G_{k+1}
\end{align*}
for $k \in \mathbb N$. Moreover $\dagger^2=1$.
\end{lemma}

\begin{lemma} \label{lem:sdcom} The maps $\sigma$, $\dagger $ commute.
\end{lemma}
\begin{proof} Routine.
\end{proof}

\begin{definition}
\label{def:gf4}
\rm
We define some generating functions in an indeterminate $t$:
\begin{align}
&\mathcal W^-(t) = \sum_{n \in \mathbb N} \mathcal W_{-n} t^n,
\qquad \qquad  
\mathcal W^+(t) = \sum_{n \in \mathbb N} \mathcal W_{n+1} t^n,
\label{eq:gf12}
\\
\label{eq:gf34}
&\mathcal G(t) = \sum_{n \in \mathbb N} \mathcal G_n t^n,
\qquad \qquad \qquad 
\mathcal {\tilde G}(t) = \sum_{n \in \mathbb N} \mathcal {\tilde G}_n t^n.
\end{align}
\end{definition}
\noindent We emphasize that throughout this paper all generating functions are formal;  issues of convergence are not considered.
\medskip

\noindent By \eqref{eq:3p4} we have
\begin{align*}
\lbrack \mathcal W_0, \mathcal W^-(t)\rbrack = 0, \qquad \qquad 
\lbrack \mathcal W_1, \mathcal W^+(t)\rbrack=0.
\end{align*}

\begin{lemma}\label{lem:aaut}
The automorphism $\sigma$ sends 
\begin{align*}
\mathcal W^-(t) \mapsto \mathcal W^+(t), \qquad
\mathcal W^+(t) \mapsto \mathcal W^-(t), \qquad
\mathcal G(t) \mapsto \mathcal {\tilde G}(t), \qquad
\mathcal {\tilde G}(t)\mapsto \mathcal G(t).
\end{align*}
The antiautomorphism $\dagger$ sends
\begin{align*}
\mathcal W^-(t) \mapsto \mathcal W^-(t), \qquad
\mathcal W^+(t) \mapsto \mathcal W^+(t), \qquad
\mathcal G(t) \mapsto \mathcal {\tilde G}(t), \qquad
\mathcal {\tilde G}(t)\mapsto \mathcal G(t).
\end{align*}
\end{lemma}
\begin{proof} By Lemmas \ref{lem:aut}, \ref{lem:antiaut}.
\end{proof}

\noindent  We now give the relations \eqref{eq:3p1}--\eqref{eq:3p11}  in terms of the generating functions \eqref{eq:gf12}, \eqref{eq:gf34}.
Let $s$ denote an indeterminate that commutes with $t$. The following result is a variation on \cite[Definition~2.2]{basnc}.

\begin{lemma} \label{lem:ad} For the algebra $\mathcal A_q$ we have
\begin{align}
& \label{eq:3pp1}
\lbrack \mathcal W_0, \mathcal W^+(t) \rbrack = \lbrack \mathcal W^-(t), \mathcal W_1 \rbrack = t^{-1}(\mathcal {\tilde G}(t)-\mathcal G(t))/(q+q^{-1}),
\\
& \label{eq:3pp2}
\lbrack \mathcal W_0, \mathcal G(t) \rbrack_q = \lbrack \mathcal {\tilde G}(t), \mathcal W_0 \rbrack_q = \rho \mathcal W^-(t)-\rho t \mathcal W^+(t),
\\
&\label{eq:3pp3}
\lbrack \mathcal G(t), \mathcal W_1 \rbrack_q = \lbrack \mathcal W_1, \mathcal {\tilde G}(t) \rbrack_q = \rho \mathcal W^+(t) -\rho t \mathcal W^-(t),
\\
&\label{eq:3pp4}
\lbrack  \mathcal W^-(s), \mathcal W^-(t) \rbrack = 0, 
\qquad 
\lbrack \mathcal W^+(s),  \mathcal W^+(t) \rbrack = 0,
\\ \label{eq:3pp5}
&\lbrack  \mathcal W^-(s), \mathcal W^+(t) \rbrack 
+
\lbrack \mathcal W^+(s), \mathcal W^-(t) \rbrack = 0,
\\ \label{eq:3pp6}
&s \lbrack \mathcal W^-(s), \mathcal G(t) \rbrack 
+
t \lbrack  \mathcal G(s),  \mathcal W^-(t) \rbrack = 0,
\\ \label{eq:3pp7}
&s \lbrack  \mathcal W^-(s), \mathcal {\tilde G}(t) \rbrack 
+
t \lbrack  \mathcal {\tilde G}(s), \mathcal W^-(t) \rbrack = 0,
\\ \label{eq:3pp8}
&s \lbrack   \mathcal W^+(s),  \mathcal G(t) \rbrack
+
t \lbrack   \mathcal G(s), \mathcal W^+(t) \rbrack = 0,
\\ \label{eq:3pp9}
&s \lbrack   \mathcal W^+(s), \mathcal {\tilde G}(t) \rbrack
+
t \lbrack \mathcal {\tilde G}(s), \mathcal W^+(t) \rbrack = 0,
\\ \label{eq:3pp10}
&\lbrack   \mathcal G(s), \mathcal G(t) \rbrack = 0, 
\qquad 
\lbrack  \mathcal {\tilde G}(s),  \mathcal {\tilde G}(t) \rbrack = 0,
\\ \label{eq:3pp11}
&\lbrack  \mathcal {\tilde G}(s), G(t) \rbrack +
\lbrack   \mathcal G(s), \mathcal {\tilde G}(t) \rbrack = 0.
\end{align}
\end{lemma}
\begin{proof} Use \eqref{eq:3p1}--\eqref{eq:3p11} and Definition
\ref{def:gf4}.
\end{proof}

\noindent We just displayed some relations involving the generating functions \eqref{eq:gf12}, \eqref{eq:gf34}. There are more relations of interest. We will show that for  $\mathcal A_q$,
\begin{align}
&\lbrack \mathcal W^-(s), \mathcal G(t) \rbrack_q -
\lbrack \mathcal W^-(t), \mathcal G(s) \rbrack_q =
s\lbrack \mathcal W^+(s), \mathcal G(t) \rbrack_q -
t\lbrack \mathcal W^+(t), \mathcal G(s) \rbrack_q,
\label{eq:3p12}
\\
&\lbrack \mathcal G(s), \mathcal W^+(t) \rbrack_q -
\lbrack \mathcal G(t), \mathcal W^+(s) \rbrack_q =
t\lbrack \mathcal G(s), \mathcal W^-(t) \rbrack_q -
s\lbrack \mathcal G(t), \mathcal W^-(s) \rbrack_q,
\label{eq:3p13}
\\
&\lbrack \mathcal {\tilde G}(s), \mathcal W^-(t) \rbrack_q -
\lbrack \mathcal {\tilde G}(t), \mathcal W^-(s) \rbrack_q =
t\lbrack \mathcal {\tilde G}(s), \mathcal W^+(t) \rbrack_q -
s\lbrack \mathcal {\tilde G}(t), \mathcal W^+(s) \rbrack_q,
\label{eq:3p14}
\\
&\lbrack \mathcal W^+(s), \mathcal {\tilde G}(t) \rbrack_q -
\lbrack \mathcal W^+(t), \mathcal {\tilde G}(s) \rbrack_q =
s\lbrack \mathcal W^-(s), \mathcal {\tilde G}(t) \rbrack_q -
t\lbrack \mathcal W^-(t), \mathcal {\tilde G}(s) \rbrack_q,
\label{eq:3p15}
\\
&\frac{t^{-1} \lbrack \mathcal G(s), \mathcal {\tilde G}(t) \rbrack - s^{-1} \lbrack \mathcal G(t),\mathcal {\tilde G}(s) \rbrack}{\rho (q+q^{-1})}
= \lbrack \mathcal W^-(t), \mathcal W^+(s)\rbrack_q -
 \lbrack \mathcal W^-(s), \mathcal W^+(t)\rbrack_q \nonumber
 \\
&+st \lbrack \mathcal W^+(t), \mathcal W^-(s)\rbrack_q
-st \lbrack \mathcal W^+(s), \mathcal W^-(t)\rbrack_q + (q-q^{-1})(s-t) \mathcal W^+(s) \mathcal W^+(t) \nonumber
\\ &
-(q-q^{-1})(s-t) \mathcal W^-(s) \mathcal W^-(t),
\label{eq:3p16}
\\
&\frac{t^{-1} \lbrack \mathcal {\tilde G}(s), \mathcal  G(t) \rbrack - s^{-1} \lbrack \mathcal {\tilde G}(t),\mathcal G(s) \rbrack}{\rho (q+q^{-1})}
= \lbrack \mathcal W^+(t), \mathcal W^-(s)\rbrack_q -
 \lbrack \mathcal W^+(s), \mathcal W^-(t)\rbrack_q \nonumber
 \\
&+st \lbrack \mathcal W^-(t), \mathcal W^+(s)\rbrack_q
-st \lbrack \mathcal W^-(s), \mathcal W^+(t)\rbrack_q + (q-q^{-1})(s-t) \mathcal W^-(s) \mathcal W^-(t) \nonumber
\\ &
-(q-q^{-1})(s-t) \mathcal W^+(s) \mathcal W^+(t),
\label{eq:3p17}
\\
&\lbrack \mathcal G(s), \mathcal {\tilde G}(t)\rbrack_q -
\lbrack \mathcal G(t), \mathcal {\tilde G}(s)\rbrack_q = (q+q^{-1})\rho t \lbrack \mathcal W^-(t), \mathcal W^+(s)\rbrack \nonumber
\\
& \qquad \qquad \qquad \qquad \qquad \qquad \qquad \qquad \quad 
-(q+q^{-1}) \rho s \lbrack \mathcal W^-(s), \mathcal W^+(t)\rbrack,
\label{eq:3p18}
\\
&
\lbrack \mathcal {\tilde G}(s), \mathcal  G(t)\rbrack_q -
\lbrack \mathcal {\tilde G}(t), \mathcal G(s)\rbrack_q = (q+q^{-1})\rho t \lbrack \mathcal W^+(t), \mathcal W^-(s)\rbrack \nonumber
\\
& \qquad \qquad \qquad \qquad \qquad \qquad \qquad \qquad \quad 
-(q+q^{-1}) \rho s \lbrack \mathcal W^+(s), \mathcal W^-(t)\rbrack.
\label{eq:3p19}
\end{align}

\noindent In Appendix A we display many equations  involving the above expressions. We will invoke these equations as we go along.
  \medskip
  
  \noindent The following result is a variation on \cite[Proposition~3.1]{basnc}.

  \begin{proposition} \label{eq:p12top19}
In the algebra $\mathcal A_q$ the relations {\rm \eqref{eq:3p12}--\eqref{eq:3p19}} hold.
\end{proposition}
\begin{proof} We refer to the generating functions $A(s,t), B(s,t),\ldots ,S(s,t)$ from Appendix A.
By Lemma \ref{lem:ad} the functions
$A(s,t), B(s,t),\ldots, J(s,t)$ are all zero in $\mathcal A_q$. The present lemma asserts that $K(s,t), L(s,t),\ldots, S(s,t)$ are all zero in $\mathcal A_q$.
 To verify this assertion, we refer to the commutator relations in Appendix A.
 The commutator relation for $\lbrack \mathcal W_0, H(s,t)\rbrack_{q^2}$ implies $K(s,t)=0$ in $\mathcal A_q$.
 Then the commutator relation for $\lbrack H(s,t), \mathcal W_1\rbrack_{q^2}$ implies $L(s,t)=0$ in $\mathcal A_q$.
 Then the commutator relation for $\lbrack I(s,t), \mathcal W_0\rbrack_{q^2}$ implies $M(s,t)=0$ in $\mathcal A_q$.
 Then the commutator relation for $\lbrack \mathcal W_1, I(s,t) \rbrack_{q^2}$ implies $N(s,t)=0$ in $\mathcal A_q$.
 Then the commutator relation for $\lbrack \mathcal W_0, L(s,t)\rbrack_{q}$ implies $P(s,t)=0$ in $\mathcal A_q$.
 Then the commutator relation for $\lbrack N(s,t), \mathcal W_0\rbrack_{q}$ implies $Q(s,t)=0$ in $\mathcal A_q$.
 Then the commutator relation for $\lbrack D(s,t), \mathcal W_1\rbrack_{q}$ implies $R(s,t)=0$ in $\mathcal A_q$.
 Then the commutator relation for $\lbrack \mathcal W_0, F(s,t)\rbrack_{q}$ implies $S(s,t)=0$ in $\mathcal A_q$.
 We have shown that  $K(s,t), L(s,t),\ldots, S(s,t)$ are all zero in $\mathcal A_q$.
\end{proof}
\medskip

\section{Another presentation for $\mathcal A_q$}

\noindent In Definition \ref{def:Aq} we gave a presentation of $\mathcal A_q$  by generators and relations. 
We now give a second presentation of $\mathcal A_q$ by generators and relations. The relations will be expressed using the
generating functions \eqref{eq:gf12}, \eqref{eq:gf34}.
\begin{proposition}\label{lem:rr}
The algebra $\mathcal A_q$ is presented by its alternating generators and the following relations:
\begin{enumerate}
\item[\rm (i)] we have
\begin{align*}
&\lbrack \mathcal G(s), \mathcal G(t)\rbrack=0, \qquad \qquad \lbrack \mathcal W^-(s), \mathcal W^-(t)\rbrack=0,  
\\
&\lbrack \mathcal W^+(s), \mathcal W^+(t)\rbrack=0, \qquad \qquad 
\lbrack \mathcal {\tilde G}(s), \mathcal {\tilde G}(t)\rbrack=0;
\end{align*}
\item[\rm (ii)]  $\mathcal W^+(s) \mathcal W^{-}(t)$ is a weighted sum with the following terms and coefficients:
\bigskip

\centerline{
\begin{tabular}[t]{c|cc}
  {\rm term} & {\rm coefficient} & {\rm coeff. name}
   \\
\hline
$\mathcal W^{-} (t) \mathcal W^{+} (s)$ & $1$ & $1$
\\
$\mathcal G(t) \mathcal {\tilde G}(s)$ & $\frac{1}{(q^2-q^{-2})(q+q^{-1})^2 (s-t)} $ & $e_{s,t}$
\\
$\mathcal G(s) \mathcal {\tilde G}(t)$ & $\frac{-1}{(q^2-q^{-2})(q+q^{-1})^2 (s-t)} $ & $-e_{s,t}$
	       \end{tabular}}
	       
\item[\rm (iii)]  $\mathcal {\tilde G}(s) \mathcal G(t)$ is a weighted sum with the following terms and coefficients:
\bigskip

\centerline{
\begin{tabular}[t]{c|cc}
  {\rm term} & {\rm coefficient} & {\rm coeff. name}
   \\
\hline
$\mathcal G(t) \mathcal {\tilde G} (s)$ & $1$ & $1$
\\
$\mathcal W^{-}(s) \mathcal W^{+}(t)$ & $\frac{(q^2-q^{-2})^3 s t (st -1)}{s-t}$ & $f_{s,t}$
\\
$\mathcal W^{-}(t) \mathcal W^{+}(s)$ & $-\frac{(q^2-q^{-2})^3 s t (st -1)}{s-t}$ & $-f_{s,t}$
\\
$\mathcal W^{+}(s)  \mathcal W^{+}(t)$ & $(q^2-q^{-2})^3 s t $ & $F_{s,t}$
\\
$\mathcal W^{-}(s)  \mathcal W^{-}(t)$ & $-(q^2-q^{-2})^3 s t $ & $-F_{s,t}$
\end{tabular}}

\item[\rm (iv)]  $\mathcal W^+(t) \mathcal G(s)$ is a weighted sum with the following terms and coefficients:
\bigskip

\centerline{
\begin{tabular}[t]{c|cc}
  {\rm term} & {\rm coefficient} & {\rm coeff. name}
   \\
\hline
$\mathcal G(t) \mathcal W^{-}(s)$ & $ \frac{q s^2(q-q^{-1})}{s-t} $ & $a'_{s,t}$
\\
$\mathcal G(s) \mathcal W^{-}(t)$ & $-\frac{q st (q-q^{-1})}{s-t} $ & $A'_{s,t}$
\\
$\mathcal G(t) \mathcal W^{+}(s)$ & $-\frac{q s(q-q^{-1})}{s-t} $ & $A_{s,t}$
\\
$\mathcal G(s) \mathcal W^{+}(t)$ & $\frac{q(qs-q^{-1}t)}{s-t} $ & $a_{s,t}$
	       \end{tabular}}
	       
\item[\rm (v)]  $\mathcal W^-(t) \mathcal G(s)$ is a weighted sum with the following terms and coefficients:
\bigskip

\centerline{
\begin{tabular}[t]{c|cc}
  {\rm term} & {\rm coefficient} & {\rm coeff. name}
   \\
\hline
$\mathcal G(t) \mathcal W^{-}(s)$ & $ \frac{q^{-1}s(q-q^{-1})}{s-t} $ & $B_{s,t}$
\\
$\mathcal G(s) \mathcal W^{-}(t)$ & $\frac{q^{-1}(q^{-1}s-qt)}{s-t} $ & $b_{s,t}$
\\
$\mathcal G(t) \mathcal W^{+}(s)$ & $-\frac{q^{-1} s^2(q-q^{-1})}{s-t} $ & $b'_{s,t}$
\\
$\mathcal G(s) \mathcal W^{+}(t)$ & $\frac{q^{-1}st (q-q^{-1})}{s-t} $ & $B'_{s,t}$
	       \end{tabular}}

\item[\rm (vi)]  $\mathcal {\tilde G}(s)\mathcal W^+(t)$ is a weighted sum with the following terms and coefficients:
\bigskip

\centerline{
\begin{tabular}[t]{c|cc}
  {\rm term} & {\rm coefficient} & {\rm coeff. name}
   \\
\hline
$\mathcal W^-(s) \mathcal {\tilde G}(t) $ & $ \frac{q s^2(q-q^{-1})}{s-t} $ & $a'_{s,t}$
\\
$\mathcal W^-(t) \mathcal {\tilde G}(s) $ & $-\frac{q st (q-q^{-1})}{s-t} $ & $A'_{s,t}$
\\
$\mathcal W^+(s) \mathcal {\tilde G}(t) $ & $-\frac{q s(q-q^{-1})}{s-t} $ & $A_{s,t}$
\\
$\mathcal W^+(t) \mathcal {\tilde G}(s)$ & $\frac{q(qs-q^{-1}t)}{s-t} $ & $a_{s,t}$
	       \end{tabular}}
	       
\item[\rm (vii)]  $ \mathcal {\tilde G}(s)\mathcal W^-(t)$ is a weighted sum with the following terms and coefficients:
\bigskip

\centerline{
\begin{tabular}[t]{c|cc}
  {\rm term} & {\rm coefficient} & {\rm coeff. name}
   \\
\hline
$\mathcal W^-(s)\mathcal {\tilde G}(t) $ & $ \frac{q^{-1}s(q-q^{-1})}{s-t} $ & $B_{s,t}$
\\
$\mathcal W^-(t) \mathcal {\tilde G}(s) $ & $\frac{q^{-1}(q^{-1}s-qt)}{s-t} $ & $b_{s,t}$
\\
$\mathcal W^+(s) \mathcal {\tilde G}(t) $ & $-\frac{q^{-1} s^2(q-q^{-1})}{s-t} $ & $b'_{s,t}$
\\
$\mathcal W^+(t) \mathcal {\tilde G}(s) $ & $\frac{q^{-1}st (q-q^{-1})}{s-t} $ & $B'_{s,t}$
	       \end{tabular}}

	       \end{enumerate}
\end{proposition}
\begin{proof} We first show that the above relations (i)--(vii) hold in $\mathcal A_q$. The relations (i)
are from \eqref{eq:3pp4}, \eqref{eq:3pp10}.
 As we discuss the relations (ii)--(vii), we will 
 use the fact that the generating functions $A(s,t), B(s,t), \ldots, S(s,t)$ from Appendix A  are all zero in $\mathcal A_q$.
\\
\noindent (ii) $\mathcal W^+(s) \mathcal W^-(t)$ minus the given weighted sum is equal to
\begin{align*}
\frac{(q+q^{-1}) \rho s C(s,t)+ q^{-1} J(s,t)-R(s,t)}{(q+q^{-1})\rho (s-t)}.
\end{align*}
\noindent 
 (iii) $\mathcal {\tilde G}(s) \mathcal G(t)$ minus the given weighted sum is equal to
\begin{align*}
&\frac{(q+q^{-1})\rho st(qst+q^{-1})C(s,t)+ s J(s,t)-(q+q^{-1})\rho stP(s,t)}{s-t}.
\end{align*}
(iv) $\mathcal W^+(t) \mathcal G(s)$ minus the given weighted sum is equal to
\begin{align*}
\frac{F(s,t)-sD(s,t)-qsL(s,t)}{s-t}.
\end{align*}
(v) $\mathcal W^-(t) \mathcal G(s)$ minus the given weighted sum is equal to
\begin{align*}
\frac{D(s,t)-sF(s,t)-q^{-1}sK(s,t)}{s-t}.
\end{align*}
(vi) $\mathcal {\tilde G}(s)\mathcal W^+(t)$ minus the given weighted sum is equal to
\begin{align*}
\frac{s E(s,t)-G(s,t)+qsN(s,t)}{s-t}.
\end{align*}
(vii) $ \mathcal {\tilde G}(s) \mathcal W^-(t)$ minus the given weighted sum is equal to
\begin{align*}
\frac{sG(s,t)-E(s,t)+q^{-1}sM(s,t)}{s-t}.
\end{align*}
\noindent We have shown that the relations (i)--(vii) hold in $\mathcal A_q$. Next we show that the relations (i)--(vii) imply the defining relations \eqref{eq:3p1}--\eqref{eq:3p11} for $\mathcal A_q$.
It suffices to show that the relations (i)--(vii) imply the relations \eqref{eq:3pp1}--\eqref{eq:3pp11}. The relations  \eqref{eq:3pp1} are obtained from the relation (ii) by setting $s=0$ or $t=0$.
The relations \eqref{eq:3pp2} are obtained from the relations (v), (vii) by setting $t=0$. 
The relations  \eqref{eq:3pp3} are obtained from the relations (iv), (vi) by setting $t=0$. 
The relations  \eqref{eq:3pp4}, \eqref{eq:3pp10} are the relations (i).
Concerning the remaining relations  in Lemma \ref{lem:ad}, pick any one of them and let $\Delta$ denote the left-hand side minus the right-hand side.
To obtain $\Delta=0$ from  the relations (i)--(vii), use these relations to put the terms of $\Delta$ in the order \eqref{eq:order}. This brief calculation is omitted.
\end{proof}

\section{The reduction rules for $\mathcal A_q$}
\noindent In Proposition \ref{lem:rr} we gave a presentation of $\mathcal A_q$ by generators and relations. In this presentation we expressed the relations using generating functions.
Next we show how the presentation looks if generating functions are not used.
\begin{proposition}\label{lem:rra}
The algebra $\mathcal A_q$ is presented by its alternating generators and the following relations:
\begin{enumerate}
\item[\rm (i)] for $i,j\in \mathbb N$ such that $i>j$,
\begin{align*}
&\mathcal G_{i+1} \mathcal G_{j+1} = \mathcal G_{j+1} \mathcal G_{i+1}, \qquad \qquad
\mathcal W_{-i} \mathcal W_{-j} = \mathcal W_{-j} \mathcal W_{-i}, 
\\
& \mathcal W_{i+1} \mathcal W_{j+1} = \mathcal W_{j+1} \mathcal W_{i+1}, 
\qquad \qquad 
\mathcal {\tilde G}_{i+1} \mathcal {\tilde G}_{j+1} = \mathcal {\tilde G}_{j+1} \mathcal {\tilde G}_{i+1};
\end{align*}
\item[\rm (ii)] for $i,j\in \mathbb N$,
\begin{align}
  \mathcal W_{i+1}\mathcal W_{-j} &= \mathcal W_{-j} \mathcal W_{i+1}
  + \sum_{\ell =0}^{{\rm min}(i,j)} \frac{\mathcal G_\ell \mathcal {\tilde G}_{i+j+1-\ell} - \mathcal G_{i+j+1-\ell} \mathcal {\tilde G}_{\ell}}{(q^2-q^{-2})(q+q^{-1})^2};
  \label{eq:rr2}
  \end{align}
\item[\rm (iii)] for $i,j\in \mathbb N$,
\begin{align*}
\mathcal {\tilde G}_{i+1} \mathcal G_{j+1} &= \mathcal G_{j+1} \mathcal {\tilde G}_{i+1} - (q^2-q^{-2})^3 \mathcal W_{-i} \mathcal W_{-j} +(q^2-q^{-2})^3 \mathcal W_{i+1} \mathcal W_{j+1}
\\
&\quad+ (q^2-q^{-2})^3 \sum_{\ell =0}^{{\rm min}(i,j)} \bigl(
\mathcal W_{-\ell} \mathcal W_{i+j+2-\ell} -\mathcal W_{\ell-1-i-j} \mathcal W_{\ell+1} \bigr)
\\
&\quad -(q^2-q^{-2})^3 \sum_{\ell =1}^{{\rm min}(i,j)} \bigl(
\mathcal W_{1-\ell} \mathcal W_{i+j+1-\ell} -\mathcal W_{\ell-i-j} \mathcal W_{\ell} \bigr);
\end{align*}
\item[\rm (iv)] for $i,j\in \mathbb N$,
\begin{align*}
\mathcal W_{i+1} \mathcal G_{j+1} &= \mathcal G_{j+1} \mathcal W_{i+1} +q(q-q^{-1}) \sum_{\ell=0}^{{\rm min}(i,j)}  \mathcal G_{\ell} \mathcal W_{\ell-i-j}
\\
&\quad+ q(q-q^{-1}) \sum_{\ell =0}^{{\rm min}(i,j)} \bigl(
\mathcal G_{i+j+1-\ell} \mathcal W_{\ell+1} -\mathcal G_{\ell} \mathcal W_{i+j+2-\ell} \bigr)
\\
&\quad -q(q-q^{-1}) \sum_{\ell =1}^{{\rm min}(i,j)} 
\mathcal G_{i+j+1-\ell} \mathcal W_{1-\ell} ;
\end{align*}
\item[\rm (v)]  for $i,j\in \mathbb N$,
\begin{align*}
\mathcal W_{-i} \mathcal G_{j+1} &= \mathcal G_{j+1} \mathcal W_{-i} -q^{-1}(q-q^{-1}) \sum_{\ell=0}^{{\rm min}(i,j)}  \mathcal G_{\ell} \mathcal W_{i+j+1-\ell}
\\
&\quad+ q^{-1}(q-q^{-1}) \sum_{\ell =0}^{{\rm min}(i,j)} \bigl(
\mathcal G_{\ell} \mathcal W_{\ell-1-i-j} -\mathcal G_{i+j+1-\ell} \mathcal W_{-\ell} \bigr)
\\
&\quad +q^{-1}(q-q^{-1}) \sum_{\ell =1}^{{\rm min}(i,j)} 
\mathcal G_{i+j+1-\ell} \mathcal W_{\ell} ;
\end{align*}

\item[\rm (vi)] for $i,j\in \mathbb N$,
\begin{align*}
\mathcal {\tilde G}_{i+1} \mathcal W_{j+1} &= \mathcal W_{j+1} \mathcal {\tilde G}_{i+1} +q(q-q^{-1}) \sum_{\ell=0}^{{\rm min}(i,j)}  \mathcal W_{\ell-i-j} \mathcal {\tilde G}_{\ell}
\\
&\quad+ q(q-q^{-1}) \sum_{\ell =0}^{{\rm min}(i,j)} \bigl(
\mathcal W_{\ell+1} \mathcal {\tilde G}_{i+j+1-\ell} -\mathcal W_{i+j+2-\ell} \mathcal {\tilde G}_{\ell} \bigr)
\\
&\quad -q(q-q^{-1}) \sum_{\ell =1}^{{\rm min}(i,j)} 
\mathcal W_{1-\ell} \mathcal {\tilde G}_{i+j+1-\ell} ;
\end{align*}
\item[\rm (vii)]  	  for $i,j\in \mathbb N$,
\begin{align*}
\mathcal {\tilde G}_{i+1} \mathcal W_{-j} &= \mathcal W_{-j} \mathcal {\tilde G}_{i+1} -q^{-1}(q-q^{-1}) \sum_{\ell=0}^{{\rm min}(i,j)}  \mathcal W_{i+j+1-\ell} \mathcal {\tilde G}_{\ell}
\\
&\quad+ q^{-1}(q-q^{-1}) \sum_{\ell =0}^{{\rm min}(i,j)} \bigl(
\mathcal W_{\ell-1-i-j} \mathcal {\tilde G}_{\ell} -\mathcal W_{-\ell} \mathcal {\tilde G}_{i+j+1-\ell} \bigr)
\\
&\quad +q^{-1}(q-q^{-1}) \sum_{\ell =1}^{{\rm min}(i,j)} 
\mathcal W_{\ell} \mathcal {\tilde G}_{i+j+1-\ell}.
\end{align*}
	       \end{enumerate}
\end{proposition}
\begin{proof} Pick any one of the relations (i)--(vii) in Proposition \ref{lem:rr}. Expand each side as a power series in $s,t$ and then equate coefficients
to obtain the corresponding relations in the present proposition.
We illustrate using relation (ii) in Proposition \ref{lem:rr}. This relation asserts that
\begin{align}
\label{eq:repeat}
\mathcal W^+(s) \mathcal W^-(t) = \mathcal W^-(t) \mathcal W^+(s) + \frac{\mathcal G (t) \mathcal {\tilde G}(s) - \mathcal G(s) \mathcal {\tilde G}(t)}{(q^2-q^{-2})(q+q^{-1})^2(s-t)}.
\end{align}
\noindent By \eqref{eq:gf12},
\begin{align}
&\mathcal W^+(s) \mathcal W^-(t) = \sum_{i,j\in \mathbb N} \mathcal W_{i+1} \mathcal W_{-j} s^i t^j, 
\qquad \quad 
\mathcal W^-(t) \mathcal W^+(s) = \sum_{i,j\in \mathbb N} \mathcal W_{-j} \mathcal W_{i+1} s^i t^j.
\label{eq:look1}
\end{align}
\noindent Using \eqref{eq:gf34},
\begin{align}
 \frac{\mathcal G (t) \mathcal {\tilde G}(s) - \mathcal G(s) \mathcal {\tilde G}(t)}{s-t} &= \sum_{n,m\in \mathbb N} \mathcal G_n \mathcal {\tilde G}_m \frac{s^m t^n-s^n t^m}{s-t} \nonumber
\\
& = \sum_{n<m} \bigl(\mathcal G_n \mathcal {\tilde G}_m -\mathcal G_m \mathcal {\tilde G}_n\bigr) \frac{s^m t^n-s^n t^m}{s-t} \nonumber
\\
& = \sum_{n<m} \bigl(\mathcal G_n \mathcal {\tilde G}_m -\mathcal G_m \mathcal {\tilde G}_n\bigr)\bigl(s^{m-1} t^n + s^{m-2} t^{n+1} + \cdots + s^n t^{m-1}\bigr) \nonumber
\\
&= \sum_{i,j\in \mathbb N} 
 \sum_{\ell=0}^{{\rm min}(i,j)} \bigl(\mathcal G_\ell \mathcal {\tilde G}_{i+j+1-\ell} - \mathcal G_{i+j+1-\ell} \mathcal {\tilde G}_\ell \bigr)s^i t^j.
 \label{eq:look2}
\end{align}
Equation \eqref{eq:rr2} is obtained from \eqref{eq:repeat} by examining the coefficient of $s^it^j$ in \eqref{eq:look1}, \eqref{eq:look2}.
\end{proof}

\begin{definition} \label{def:redrule}
\rm The relations (i)--(vii) in  Proposition \ref{lem:rra} are called the {\it reduction rules} for $\mathcal A_q$.
\end{definition}
\noindent 
The relations (i)--(vii) in Proposition \ref{lem:rr} express the reduction rules for $\mathcal A_q$
in terms of generating functions. Going forward, we will refer to these relations as follows.
\begin{definition} \label{def:GFredrule} \rm
The relations (i)--(vii) in Proposition \ref{lem:rr} are called the {\it GF reduction rules} for $\mathcal A_q$.
\end{definition}

\section{The alternating generators form a PBW basis for $\mathcal A_q$}

\noindent  We now state our first main result, which is a variation on  \cite[Conjecture~4.5]{z2z2z2}. Both the result and  \cite[Conjecture~4.5]{z2z2z2} are inspired by
\cite[Conjecture~1]{basBel}.
\begin{theorem}
\label{thm:rr}
A PBW basis for $\mathcal A_q$ is obtained by its alternating generators in any linear order $<$ such that
\begin{align*}
\mathcal G_{i+1} < \mathcal W_{-j} < \mathcal W_{k+1} < \mathcal {\tilde G}_{\ell+1}\qquad \qquad (i,j,k, \ell \in \mathbb N).
\end{align*}
\end{theorem}
\noindent We will prove Theorem \ref{thm:rr} after a brief discussion.
\medskip

\noindent We have some
 comments about the linear order $<$ in the theorem statement.
By Proposition \ref{lem:rra}(i), we may assume without loss that the linear order $<$ satisfies
\begin{align*}
\mathcal G_{i+1} < \mathcal G_{j+1}, \qquad
\mathcal W_{-i} < \mathcal W_{-j}, \qquad
\mathcal W_{i+1} < \mathcal W_{j+1}, \qquad
\mathcal {\tilde G}_{i+1} < \mathcal {\tilde G}_{j+1}
\end{align*}
for all $i,j \in \mathbb N$ such that $i<j$. This assumption is in force for the rest of the paper.
\medskip

\noindent We discuss the notion of an $\mathcal A_q$-word.
By an {\it $\mathcal A_q$-letter} we mean an alternating generator for $\mathcal A_q$.
For $n \in \mathbb N$,  by an {\it $\mathcal A_q$-word of length $n$}
we mean a sequence of $\mathcal A_q$-letters
 $a_1a_2\cdots a_n$.
We interpret the 
$\mathcal A_q$-word
of length zero to be the multiplicative
identity $1 \in \mathcal A_q$. An
 $\mathcal A_q$-word
$a_1a_2\cdots a_n$
 is called {\it reducible} whenever there exists an integer $i$ $(2 \leq i \leq n)$ such that $a_{i-1}>a_i$.
An $\mathcal A_q$-word is called  {\it irreducible} whenever it is not reducible. An $\mathcal A_q$-word
$a_1 a_2 \cdots a_n$ is irreducible if and only if $a_1 \leq a_2 \leq  \cdots \leq  a_n$.
Theorem \ref{thm:rr} asserts  that the irreducible $\mathcal A_q$-words give a basis for the vector space $\mathcal A_q$.
\medskip

\noindent Let $\xi$ denote an indeterminate.
To each $\mathcal A_q$-letter we assign a weight, as follows. 
For $k\in \mathbb N$,
\medskip
 
\centerline{
\begin{tabular}[t]{c|cccc}
  $u$& $\mathcal G_{k+1}$ & $\mathcal W_{-k}$ &$\mathcal W_{k+1}$ & $\mathcal {\tilde G}_{k+1}$ 
   \\
\hline
${\rm wt}(u)$ & $k+1$ &$\xi^2+k$ &$\xi^2+\xi+k$ & $2\xi^2+k+1 $
\end{tabular}}
\bigskip

\noindent To each $\mathcal A_q$-word $w$ we assign a weight, as follows. Writing $w=a_1a_2\cdots a_n$,
\begin{align*}
{\rm wt}(w) = \sum_{i=1}^n (n-i+1) {\rm wt}(a_i).
\end{align*}
For example, $\mathcal W_{2} \mathcal W_{-1}  \mathcal G_3\mathcal W_4$ has weight $8\xi^2+5\xi+16$
and 
 $\mathcal W_{3} \mathcal {\tilde G}_4 \mathcal W_{-1}  \mathcal W_{-2}$ has weight $13\xi^2+4\xi+24$.
The weight of each $\mathcal A_q$-word has the form $a\xi^2+b\xi+c $ with $a,b,c \in \mathbb N$. 
\medskip

\noindent
The abelian group $\mathbb Z \xi^2+ \mathbb Z\xi+\mathbb Z$ consists of the polynomials $a\xi^2  +b\xi+c $ with $a,b,c\in \mathbb Z$.
We endow $\mathbb Z \xi^2+ \mathbb Z\xi+\mathbb Z$ with a linear order $<$ as follows. For $a,b,c\in \mathbb Z$ and $a',b',c'\in \mathbb Z$ we have $a\xi^2 + b\xi+c < a' \xi^2+ b'\xi+c'$ whenever 
 (i) $a<a'$; or (ii) $a=a'$ and $b<b'$; or (iii) $a=a'$ and $b=b'$ and $c<c'$. The linear order $<$ on $\mathbb Z \xi^2+ \mathbb Z\xi + \mathbb Z$
has the following
property: for $u,v,u',v' \in\mathbb Z \xi^2+ \mathbb Z\xi+ \mathbb Z$ such that
$u\leq u' $ and $v \leq v'$, we have $u+v\leq u'+v'$.
\medskip

\noindent
  In Definition \ref{def:redrule} we introduced the reduction rules for $\mathcal A_q$. These reduction rules are in bijection with
  the reducible $\mathcal A_q$-words of length 2.
Let $w$ denote a reducible $\mathcal A_q$-word of length 2, and consider the corresponding reduction rule.
In this reduction rule, the left-hand side is $w$ and the right-hand side is a linear combination of 
$\mathcal A_q$-words that have length 1 or 2. 
An $\mathcal A_q$-word that appears on the right-hand side of this reduction rule is called a {\it descendent} of $w$.
\medskip

\noindent 
Next we define a partial
order $\preceq $ on the set of $\mathcal A_q$-words.
Let $w, w'$ denote $\mathcal A_q$-words, and write $w=a_1 a_2\cdots a_n$. We say that {\it $w$ dominates $w'$} whenever  there exists an integer
$i$ $(2 \leq i \leq n)$ such that $a_{i-1}>a_i$, and $w'$ is obtained from $w$ by replacing $a_{i-1}a_i$ by one of its descendents.
In this case, either (i) $w'$ has length  $n-1$; or (ii) $w'$ has length $n$ and ${\rm wt}(w')< {\rm wt}(w)$. By these comments the transitive
closure of the domination relation is a partial order on the set of $\mathcal A_q$-words; we denote this partial order by  $\preceq$.
\medskip

\noindent {\it Proof of Theorem \ref{thm:rr}}. We will use the Bergman diamond lemma \cite[Theorem~1.2]{berg}. 
We refer to the above partial order $\preceq $ on the set of $\mathcal A_q$-words. By construction $\preceq$ is
a semigroup partial order
\cite[p.~181]{berg} and satisfies the
descending chain condition
\cite[p.~179]{berg}. 
We show that the reduction rules are compatible with $\preceq$ in the sense of
Bergman \cite[p.~181]{berg}.
Let  $w=a_1a_2\cdots a_n$ denote a reducible $\mathcal A_q$-word.
There exists an integer $i$ $(2 \leq i \leq n)$ such that $a_{i-1}>a_i$.
By construction $a_{i-1}a_i$ is a reducible $\mathcal A_q$-word of length 2. Using the corresponding reduction rule
 we eliminate $a_{i-1}a_i$ from $w$,
and thereby express $w$ as a linear combination of
$\mathcal A_q$-words, each  $\prec w$.
Therefore,
the reduction rules are compatible
with $\preceq $ in the sense of Bergman \cite[p.~181]{berg}. 
\medskip

\noindent
In order to employ the diamond lemma,
we must show that 
the ambiguities
are resolvable in the sense of Bergman
\cite[p.~181]{berg}.
There are potentially two kinds of ambiguities; inclusion
ambiguities and overlap ambiguities
\cite[p.~181]{berg}.
For the present example there are no inclusion ambiguities.
There are four types of overlap ambiguities, and these are displayed below:
\begin{enumerate}
\item[\rm (i)] For $i,j,k\in \mathbb N$,
\begin{align*}
\mathcal W_{i+1} \mathcal W_{-j} \mathcal G_{k+1}, \qquad \quad
\mathcal {\tilde G}_{i+1} \mathcal W_{-j} \mathcal G_{k+1}, \qquad \quad
\mathcal {\tilde G}_{i+1} \mathcal W_{j+1} \mathcal G_{k+1}, \qquad \quad
\mathcal {\tilde G}_{i+1} \mathcal W_{j+1} \mathcal W_{-k}.
\end{align*}
\item[\rm (ii)] For $i,j,k \in \mathbb N$ such that $i>j$,
\begin{align*}
&\mathcal W_{-i} \mathcal W_{-j} \mathcal G_{k+1}, \qquad \quad
\mathcal W_{i+1} \mathcal W_{j+1}  \mathcal G_{k+1}, \qquad \quad
\mathcal {\tilde G}_{i+1} \mathcal {\tilde G}_{j+1} \mathcal G_{k+1},
\\
&\mathcal W_{i+1} \mathcal W_{j+1} \mathcal W_{-k}, \qquad \quad
\mathcal {\tilde G}_{i+1} \mathcal {\tilde G}_{j+1} \mathcal W_{-k}, \qquad \quad
\mathcal {\tilde G}_{i+1} \mathcal {\tilde G}_{j+1} \mathcal W_{k+1}.
\end{align*}
\item[\rm (iii)] For $i,j,k \in \mathbb N$ such that $j>k$,
\begin{align*}
&\mathcal W_{-i} \mathcal G_{j+1} \mathcal G_{k+1}, \qquad \quad
\mathcal W_{i+1} \mathcal G_{j+1} \mathcal G_{k+1}, \qquad \quad
\mathcal {\tilde G}_{i+1} \mathcal G_{j+1} \mathcal G_{k+1},
\\
&\mathcal W_{i+1} \mathcal W_{-j} \mathcal W_{-k}, \qquad \quad
\mathcal {\tilde G}_{i+1} \mathcal W_{-j} \mathcal W_{-k}, \qquad \quad
\mathcal {\tilde G}_{i+1} \mathcal W_{j+1} \mathcal W_{k+1}.
\end{align*}
\item[\rm (iv)] For $i,j,k\in \mathbb N$ such that $i>j>k$,
\begin{align*}
\mathcal G_{i+1} \mathcal G_{j+1} \mathcal G_{k+1}, \qquad \quad
\mathcal W_{-i} \mathcal W_{-j} \mathcal W_{-k}, \qquad \quad
\mathcal W_{i+1} \mathcal W_{j+1} \mathcal W_{k+1}, \qquad \quad
\mathcal {\tilde G}_{i+1} \mathcal {\tilde G}_{j+1} \mathcal {\tilde G}_{k+1}.
\end{align*}
\end{enumerate}
Take for instance the overlap ambiguity $\mathcal W_{i+1} \mathcal W_{-j} \mathcal G_{k+1}$. We use the reduction rules to express
$\mathcal W_{i+1} \mathcal W_{-j} \mathcal G_{k+1}$ as a  linear combination of irreducible $\mathcal A_q$-words.
There are two ways to begin; we could eliminate
$\mathcal W_{i+1} \mathcal W_{-j}$ first, or we could eliminate 
$\mathcal W_{-j} \mathcal G_{k+1}$ first. It turns out that the two ways yield the same result, which means that 
the 
overlap ambiguity  $\mathcal W_{i+1} \mathcal W_{-j} \mathcal G_{k+1}$ is resolvable \cite[p.~181]{berg}. In Appendix C we use generating functions to show that every overlap ambiguity is resolvable.
Since every  ambiguity is
resolvable, 
by the diamond lemma 
\cite[Theorem~1.2]{berg}
the irreducible $\mathcal A_q$-words give a basis for
the vector space $\mathcal A_q$.
\hfill $\Box$
\bigskip

\noindent The proof of Theorem \ref{thm:rr} illustrates how to write a reducible $\mathcal A_q$-word
as a linear combination of irreducible $\mathcal A_q$-words.
Let $w=a_1a_2\cdots a_n$ denote a reducible $\mathcal A_q$-word. There exists an integer $i$ $(2 \leq i \leq n)$ such that $a_{i-1}>a_i$. By construction $a_{i-1}a_i$ is a reducible $\mathcal A_q$-word of length 2.
Using the corresponding reduction rule
 we eliminate $a_{i-1}a_i$ from $w$,
and thereby express $w$ as a linear combination of
$\mathcal A_q$-words, each $\prec w$. We iterate this procedure and after a finite number of steps, $w$ is expressed
 as a linear combination of irreducible $\mathcal A_q$-words, each $\prec w$.

\section{A filtration of $\mathcal A_q$}

\noindent Recall the filtration concept from Definition \ref{def:filtration}.
In this section we obtain a filtration of $\mathcal A_q$ that is related to the PBW basis from Theorem \ref{thm:rr}.

\begin{definition}\label{def:degree}
\rm
To each alternating generator for $\mathcal A_q$ we assign a degree, as follows. For $k \in \mathbb N$,
\bigskip
 
\centerline{
\begin{tabular}[t]{c|cccc}
  $u$ & $\mathcal G_{k+1}$ & $\mathcal W_{-k}$ & $\mathcal W_{k+1}$ & $\mathcal {\tilde G}_{k+1}$ 
   \\
\hline
${\rm deg}(u)$ & $2k+2$ & $2k+1$ & $2k+1$ & $2k+2$
\end{tabular}}
\bigskip

\end{definition}

\begin{definition}\label{def:worddeg}
\rm
To each $\mathcal A_q$-word $w$ we assign a degree, as follows. Writing $w=a_1a_2\cdots a_n$,
\begin{align*}
{\rm deg}(w) = \sum_{i=1}^n {\rm deg}(a_i).
\end{align*}
\end{definition}

\begin{definition}\label{def:An} \rm
For $d \in \mathbb N$ let $A_d$ denote the subspace of $\mathcal A_q$ spanned by the irreducible $\mathcal A_q$-words of degree $d$.
\end{definition}

\begin{lemma} \label{lem:ds} We have $A_0=\mathbb F1$. Moreover the sum $\mathcal A_q = \sum_{d \in \mathbb N} A_d$ is direct.
\end{lemma}
\begin{proof} We have $A_0=\mathbb F1$ since $1$ is the unique irreducible $\mathcal A_q$-word that has degree 0. The sum $\mathcal A_q=\sum_{d\in \mathbb N} A_d$ is direct by
 Definition \ref{def:An} and since the irreducible $\mathcal A_q$-words form a basis for the vector space $\mathcal A_q$.
\end{proof}

\noindent Referring to Definition \ref{def:An}, we next consider the dimension of $A_d$.

\begin{definition}\label{def:genf}
\rm
 We define a generating function in an indeterminate $x$:
\begin{align*}
\mathcal H(x) = \sum_{d \in \mathbb N} {\rm dim}(A_d) x^d.
\end{align*}
\end{definition}

\begin{lemma}\label{lem:genf} We have
\begin{align}
\label{eq:prod2}
\mathcal H(x) = \prod_{n=1}^\infty \frac{1}{(1-x^n)^2}.
\end{align}
\end{lemma}
\begin{proof} Each alternating generator $g$ contributes a factor $(1-x^d)^{-1} = 1 + x^d + x^{2d}+ \cdots $ to $\mathcal H(x)$, where $d$ is the degree of $g$.
By this and Definition \ref{def:degree},
\begin{align*}
\mathcal H(x) = \prod_{k\in \mathbb N} \frac{1}{(1-x^{2k+1})^2} \frac{1}{(1-x^{2k+2})^{2}} =\prod_{n=1}^\infty \frac{1}{(1-x^n)^2}.
\end{align*}
\end{proof}

\begin{example} \rm In the table below we display the dimension of $A_d$ for $0 \leq d \leq 8$.
\bigskip
 
\centerline{
\begin{tabular}[t]{c|ccccccccc}
  $d$ & $0$& $1$ &$2$& $3$& $4$& $5$& $6$ &$7$&$8$
   \\
\hline
${\rm dim}(A_d)$ & $1$&$ 2$&$ 5$&$ 10$&$ 20$&$ 36$&$ 65 $&$110$&$185$
\end{tabular}}
\bigskip
\end{example}

\begin{definition}\label{def:filt}\rm
For $d \in \mathbb N$ define $B_d = A_0+A_1+\cdots + A_d$. For notational convenience define $B_{-1}=0$.
\end{definition}

\begin{lemma} \label{lem:AB} The following hold for $d\in \mathbb N$:
\begin{enumerate}
\item[\rm (i)] the sum $B_d = B_{d-1}+A_d$ is direct;
\item[\rm (ii)]  the quotient vector space $B_d/B_{d-1}$ is isomorphic to $A_d$.
\end{enumerate}
\end{lemma}
\begin{proof} (i) By Lemma \ref{lem:ds} and Definition \ref{def:filt}.
\\
\noindent (ii) By (i) above.
\end{proof}

\begin{lemma} \label{lem:123}
 The following {\rm (i)--(iii)} hold:
\begin{enumerate}
\item[\rm (i)] $B_0 = \mathbb F1$;
\item[\rm (ii)] $B_{d-1} \subseteq B_d$ for $d\geq 1$;
\item[\rm (iii)] $\mathcal A_q = \cup_{d \in \mathbb N} B_d$;
\end{enumerate}
\end{lemma}
\begin{proof} By Lemma \ref{lem:ds} and Definition \ref{def:filt}.
\end{proof}
\noindent Shortly we will show that  $B_r B_s \subseteq B_{r+s}$ for $r,s \in \mathbb N$. To prepare for this, we consider how the degree of an $\mathcal A_q$-word
is affected by the reduction rules.

\begin{lemma}\label{lem:desless}
 Let $w$ denote a reducible $\mathcal A_q$-word of length 2, and let $w'$ denote a descendent of $w$. Then ${\rm deg}(w')\leq {\rm deg}(w)$.
\end{lemma}
\begin{proof} There exists a reduction rule with $w$ on the left  and $w'$ appearing on the right.  The reduction rules are displayed in (i)--(vii) of Proposition
\ref{lem:rra}. For each reduction rule, we examine the term on the left and all the terms on the right  using Definitions \ref{def:degree}, \ref{def:worddeg}. The examination shows that ${\rm deg}(w')\leq {\rm deg}(w)$.
\end{proof}

\noindent In the next result we refer to the partial order $\preceq$ on $\mathcal A_q$-words defined above the proof of Theorem \ref{thm:rr}.
\begin{lemma} \label{lem:pos}
Let $w$ and $w'$ denote $\mathcal A_q$-words with $w' \preceq w$. Then  ${\rm deg}(w')\leq {\rm deg}(w)$.
\end{lemma} 
\begin{proof} The relation $\preceq$ is the transitive closure of the domination relation. So without loss of generality, we may assume that $w$ dominates $w'$. Write $w=a_1a_2\cdots a_n$. There exists an integer $i$ $(2 \leq i \leq n)$ such that
$a_{i-1}>a_i$, and $w'$ is obtained from $w$ by replacing $a_{i-1}a_i$ by a descendent. We have  ${\rm deg}(w')\leq {\rm deg}(w)$ in view of 
Definition \ref{def:worddeg} and Lemma \ref{lem:desless}.
\end{proof}

\begin{lemma}\label{lem:span}
For $d \in \mathbb N$ the vector space $B_d$ is spanned by the $\mathcal A_q$-words that have degree at most $d$.
\end{lemma}
\begin{proof} By Definition \ref{def:filt}, $B_d$ is spanned by the irreducible $\mathcal A_q$-words that have degree at most $d$. Let $w$ denote a reducible $\mathcal A_q$-word that has degree at most $d$. We show that $w \in B_d$.
By the discussion below Theorem \ref{thm:rr}, $w$ is a linear combination of irreducible $\mathcal A_q$-words, each  $\prec w$.
For each irreducible $\mathcal A_q$-word $w'$ in this linear combination, 
we have ${\rm deg}(w') \leq {\rm deg}(w)\leq d$ in view of Lemma \ref{lem:pos},  so $w' \in B_d$. The vector $w$ is a linear combination of elements in $B_d$, so $w \in B_d$. The result follows.
\end{proof}

\begin{proposition}
\label{prop:filt} 
We have  $B_r B_s \subseteq B_{r+s}$ for $r,s \in \mathbb N$.
\end{proposition}
\begin{proof} By Lemma \ref{lem:span} the subspace $B_r$ (resp. $B_s$) is spanned by the $\mathcal A_q$-words that have degree at most $r$ (resp. $s$).
 Let $u$ (resp. $v$)  denote an $\mathcal A_q$-word that has degree at most $r$ (resp. $s$). The $\mathcal A_q$-word $uv$ has degree ${\rm deg}(u)+ {\rm deg}(v) \leq r+s$. Therefore $uv \in B_{r+s}$ in view of Lemma \ref{lem:span}.
The result follows.
\end{proof}

\noindent 
By Lemma \ref{lem:123} and Proposition \ref{prop:filt} the sequence $\lbrace B_d\rbrace_{d \in \mathbb N}$ is a filtration of $\mathcal A_q$. By Definition \ref{def:genf} and Lemma \ref{lem:AB}(ii)
we obtain
\begin{align}
\mathcal H(x) = \sum_{d \in \mathbb N} {\rm dim}(B_d /B_{d-1}) x^d.
\label{eq:grcharB}
\end{align}

\section{Some central elements $\lbrace \mathcal Z_n\rbrace_{n\in \mathbb N}$ for $\mathcal A_q$}

\noindent In \cite{basBel} Baseilhac and Belliard introduce some central elements $\lbrace \Delta_n\rbrace_{n=1}^\infty$  in $\mathcal A_q$ that together with $\mathcal W_0, \mathcal W_1$ generate $\mathcal A_q$.
Their analysis requires the assumption that ${\rm Char}(\mathbb F)\not=2$; see Remark \ref{rem:charF} below. In this section we modify their approach in order to remove the assumption on $\mathbb F$.
\medskip

\noindent 
 Recall the indeterminate $t$. For notational convenience define
\begin{align}
S=\frac{q+q^{-1}}{q^{-1} t+ q t^{-1}}, \qquad \qquad 
T=\frac{q+q^{-1}}{qt + q^{-1} t^{-1}}.
\label{eq:ST}
\end{align}
\noindent We view $S$ and $T$ as power series
\begin{align*}
S = (q+q^{-1}) \sum_{\ell \in \mathbb N} (-1)^\ell q^{-2\ell-1} t^{2\ell+1} \qquad \qquad T = (q+q^{-1}) \sum_{\ell \in \mathbb N} (-1)^\ell q^{2\ell+1} t^{2\ell+1}.
\end{align*}
\noindent We will consider $\mathcal W^{\pm}(S)$, $\mathcal G(S)$, $\mathcal {\tilde G}(S)$ and also $\mathcal W^{\pm}(T)$, $\mathcal G(T)$, $\mathcal {\tilde G}(T)$.
 The following three results that will make it easier to evaluate these expressions.
 
 \begin{lemma}\label{lem:warm}  {\rm (See \cite[Lemma~4.5]{conj}.)}
Let $\lbrace a_n\rbrace_{n \in \mathbb N}$ denote a sequence of elements in $\mathcal A_q$, and consider the generating function $a(t)=\sum_{n \in \mathbb N} a_n t^n$.
 Then
 \begin{align*}
a  \biggl(\frac{q+q^{-1}}{t+t^{-1}}\biggr) = \sum_{n \in \mathbb N} a^\downarrow_n \lbrack 2 \rbrack^n_q t^n,
\end{align*} 
  where $a^\downarrow_0=a_0$ and
\begin{align*}
 a^\downarrow_n = \sum_{\ell=0}^{\lfloor (n-1) /2 \rfloor} (-1)^\ell \binom{n-1-\ell}{\ell} \lbrack 2 \rbrack^{-2\ell}_q a_{n-2\ell}, \qquad \quad n\geq 1.
\end{align*}
\end{lemma}
\begin{proof} This is routinely checked.
\end{proof}

\begin{lemma} \label{lem:warm2}
With reference to Lemma \ref{lem:warm} we have
 \begin{align*}
\frac{q+q^{-1}}{t+t^{-1}} a  \biggl(\frac{q+q^{-1}}{t+t^{-1}}\biggr) = \sum_{n \in \mathbb N} a^\Downarrow_n \lbrack 2 \rbrack^{n+1}_q t^{n+1},
\end{align*} 
  where 
\begin{align*}
 a^\Downarrow_n = \sum_{\ell=0}^{\lfloor n /2 \rfloor} (-1)^\ell \binom{n-\ell}{\ell} \lbrack 2 \rbrack^{-2\ell}_q a_{n-2\ell}, \qquad \quad n \in \mathbb N.
\end{align*}
 \end{lemma}
 \begin{proof} Apply Lemma \ref{lem:warm}, with $a(t)$ replaced by $t a(t)$.
\end{proof}

\begin{lemma} \label{lem:ramp} 
Let $\lbrace a_n\rbrace_{n \in \mathbb N}$ denote a sequence of elements in $\mathcal A_q$, and consider the generating function $a(t)=\sum_{n \in \mathbb N} a_n t^n$.
Then
\begin{align*}
a(S) &= \sum_{n \in \mathbb N} a^\downarrow_n q^{-n} \lbrack 2 \rbrack^n_q t^n,
\qquad \quad 
Sa(S) = \sum_{n \in \mathbb N} a^\Downarrow_n q^{-n-1} \lbrack 2 \rbrack^{n+1}_q t^{n+1},
\\
a(T) &= \sum_{n \in \mathbb N} a^\downarrow_n q^{n} \lbrack 2 \rbrack^n_q t^n,
\qquad \quad
Ta(T) = \sum_{n \in \mathbb N} a^\Downarrow_n q^{n+1} \lbrack 2 \rbrack^{n+1}_q t^{n+1}.
\end{align*}
\end{lemma}
\begin{proof} Apply Lemmas \ref{lem:warm}, \ref{lem:warm2} with $t$ replaced by $q^{-1}t$ or $qt$.
\end{proof}

\noindent The generating function $\mathcal Z(t)$ below is a variation on the generating function $\Delta(u)$ that appears in
\cite[Proposition~2.1]{basBel}.
\begin{definition} \label{lem:zV1} \rm For the algebra $\mathcal A_q$ define
\begin{align*}
\mathcal Z(t) &= 
t^{-1} ST\mathcal W^-(S) \mathcal W^+(T) +
t ST\mathcal W^+(S) \mathcal W^-(T) -
q^2 ST \mathcal W^-(S) \mathcal W^-(T)
\\
& \quad \qquad -
q^{-2}ST \mathcal W^+(S) \mathcal W^+(T) +
(q^2-q^{-2})^{-2} \mathcal G(S) \mathcal {\tilde G}(T).
\end{align*}
\end{definition}

\begin{remark}\rm In Appendix B we describe
$\lbrace \mathcal W^\Downarrow_{-n}\rbrace_{n\in\mathbb N}$,
$\lbrace \mathcal W^\Downarrow_{n+1}\rbrace_{n\in \mathbb N}$,
$\lbrace \mathcal G^\downarrow_{n}\rbrace_{n\in \mathbb N}$,
$\lbrace \mathcal {\tilde G}^\downarrow_{n}\rbrace_{n \in \mathbb N}$.
\end{remark}

 \begin{definition}\label{def:Zn} For $n \in \mathbb N$ define $\mathcal Z_n \in \mathcal A_q$ such that
 \begin{align*}
 \mathcal Z(t) = \sum_{n\in \mathbb N} \mathcal Z_n \lbrack 2 \rbrack_q^n t^n.
 \end{align*}
 \end{definition}

\noindent Later in this section we give an explicit formula for each $\mathcal Z_n$.
\medskip

\noindent Shortly we will show that each $\mathcal Z_n$ is central in $\mathcal A_q$. First we describe how $\mathcal Z(t)$ looks in the PBW basis for $\mathcal A_q$
from Theorem \ref{thm:rr}.

 \begin{lemma}\label{def:Zdef}
  The function $\mathcal Z(t)$ is equal to $ST$ times
 \begin{align*}
 &t^{-1} \mathcal W^-(S)\mathcal W^+(T)
 +
  t\mathcal W^-(T)\mathcal W^+(S)
 -
  q^2 \mathcal W^-(S)\mathcal W^-(T)
 \\
 &\qquad \qquad \qquad  -
   q^{-2} \mathcal W^+(S)\mathcal W^+(T)
 + \frac{t \mathcal G(T)\mathcal {\tilde G}(S)-t^{-1} \mathcal G(S) \mathcal {\tilde G}(T)}{(q^2-q^{-2})(q+q^{-1})^2 (S-T)}.
 \end{align*}
 \end{lemma}
 \begin{proof} Use the GF reduction rules to express $\mathcal Z(t)$ in the PBW basis for $\mathcal A_q$ from Theorem \ref{thm:rr}.
 \end{proof}
 
 \noindent We mention some alternative ways to express $\mathcal Z(t)$.
  \begin{lemma} \label{lem:zV2} We have
\begin{align*}
\mathcal Z(t) &=                  
t^{-1} ST\mathcal W^+(S) \mathcal W^-(T) +
t ST \mathcal W^-(S) \mathcal W^+(T) -
q^2 ST \mathcal W^+(S) \mathcal W^+(T) 
\\
&\quad \qquad 
-q^{-2} ST \mathcal W^-(S) \mathcal W^-(T)
 +
(q^2-q^{-2})^{-2} \mathcal {\tilde G}(S) \mathcal G(T),
\\
\mathcal Z(t) &= 
t^{-1} ST\mathcal W^+(T) \mathcal W^-(S) +
t ST\mathcal W^-(T) \mathcal W^+(S) 
-q^2 ST \mathcal W^-(T) \mathcal W^-(S)
\\
& \quad \qquad -
q^{-2}ST \mathcal W^+(T) \mathcal W^+(S) +
(q^2-q^{-2})^{-2} \mathcal G(T) \mathcal {\tilde G}(S),
\\
\mathcal Z(t) &= 
t^{-1} ST\mathcal W^-(T) \mathcal W^+(S) +
t ST\mathcal W^+(T) \mathcal W^-(S) -
q^{2}ST \mathcal W^+(T) \mathcal W^+(S) 
\\
& \quad \qquad -
q^{-2} ST \mathcal W^-(T) \mathcal W^-(S)
 +
(q^2-q^{-2})^{-2} \mathcal {\tilde G}(T) \mathcal G(S).
\end{align*}
\end{lemma}
\begin{proof} For each of the above three equations,  use the GF reduction rules to show that the right-hand side is equal to
 the right-hand side of the equation in Lemma  \ref{def:Zdef}.
\end{proof}


  \begin{lemma} \label{lem:OmS} The function $\mathcal Z(t)$  is fixed by the automorphism $\sigma$ from Lemma  \ref{lem:aut}
 and the antiautomorphism $\dagger$ from Lemma \ref{lem:antiaut}. 
 \end{lemma}
 \begin{proof} This is routinely checked using the formulas for $\mathcal Z(t)$ from Lemmas \ref{lem:zV1}, \ref{lem:zV2} along with
 Lemma \ref{lem:aut} and Lemma \ref{lem:antiaut}.
 \end{proof}
 
\begin{lemma} For $n \in \mathbb N$ the element $\mathcal Z_n$ is fixed by the automorphism $\sigma$ from Lemma  \ref{lem:aut}
 and the antiautomorphism $\dagger$ from Lemma \ref{lem:antiaut}. 
 \end{lemma}
 \begin{proof} By Lemma \ref{lem:OmS}.
 \end{proof}
 
\noindent Let $r$ denote an indeterminate the commutes with $t$.
 \begin{lemma} \label{lem:ZtCentral}
 For the algebra $\mathcal A_q$,
 \begin{align*}
& \lbrack \mathcal W^-(r), \mathcal Z(t) \rbrack = 0, \qquad \qquad \lbrack \mathcal W^+(r) , \mathcal Z(t) \rbrack=0,
 \\
 & \lbrack \mathcal G(r), \mathcal Z(t) \rbrack=0, \qquad \quad \qquad \lbrack \mathcal {\tilde G}(r), \mathcal Z(t) \rbrack=0.
 \end{align*}
 \end{lemma}
 \begin{proof} Working with the formula for $\mathcal Z(t)$ given in Lemma \ref{def:Zdef},
 use the GF reduction rules to express each of the above commutators in the PBW basis from Theorem \ref{thm:rr}. In each case the commutator is found to be zero.
 \end{proof}
 
 \noindent The following result is a variation on \cite[Proposition~2.1]{basBel}.
 \begin{proposition} \label{prop:ZC}
 The elements $\lbrace \mathcal Z_n \rbrace_{n \in \mathbb N}$ are central in $\mathcal A_q$.
 \end{proposition}
 \begin{proof} By Lemma \ref{lem:ZtCentral}.
 \end{proof}

 \noindent Our next goal is to interpret $\mathcal Z(t)$ using matrices. Strictly speaking this interpretation is not needed, but it helps to illuminate the nature of $\mathcal Z(t)$.

\begin{lemma} \label{lem:prep1} For the algebra $\mathcal A_q$,
\begin{align*}
  0&= qT
    \mathcal G(S) \mathcal W^+(T)
 - q^{-1} t^{-1}T
 \mathcal G(S) \mathcal W^-(T) 
-q^{-1}S \mathcal W^+ (S) \mathcal G(T) +
  qt^{-1}S \mathcal W^-(S) \mathcal  G(T),
  \\
    0 &= 
   qT
    \mathcal {\tilde G}(S) \mathcal W^-(T)
 - q^{-1} t^{-1}T
 \mathcal {\tilde G}(S) \mathcal W^+(T) 
-q^{-1}S \mathcal W^- (S) \mathcal {\tilde G}(T) +
  qt^{-1}S \mathcal W^+(S) \mathcal  {\tilde G}(T),
  \\
   0 &= 
   qT
    \mathcal W^+(T) \mathcal {\tilde G}(S)
 - q^{-1} t^{-1}T
 \mathcal W^-(T) \mathcal {\tilde G}(S) 
-q^{-1}S \mathcal {\tilde G} (T) \mathcal W^+(S) +
  qt^{-1}S \mathcal {\tilde G}(T) \mathcal  W^-(S),
\\
 0&= qT
    \mathcal W^-(T) \mathcal G(S)
 - q^{-1} t^{-1}T
 \mathcal W^+(T) \mathcal G(S) 
-q^{-1}S \mathcal G (T) \mathcal W^-(S) +
  qt^{-1}S \mathcal G(T) \mathcal  W^+(S).
 \end{align*}
 \end{lemma}
 \begin{proof} For each of the above four equations, use the GF reduction rules to express the right-hand side in the PBW basis for $\mathcal A_q$ from Theorem \ref{thm:rr}.
 In each case the right-hand side is found to be zero.
 \end{proof}

\begin{lemma} The matrix $\mathcal Z(t) I$ is equal to the product of the following two matrices, in either order:
\begin{align*}
    &    \left(
         \begin{array}{cc}
              q^{-1}t S\mathcal W^+(S) -q S\mathcal W^-(S) & (q^2-q^{-2})^{-1} \mathcal G(S)
	        \\
               (q^2-q^{-2})^{-1}  \mathcal {\tilde G}(S) & qt^{-1} S\mathcal W^+(S)-q^{-1} S \mathcal W^-(S)
                  \end{array}
              \right),
\\
 &    \left(
         \begin{array}{cc}
              q T\mathcal W^-(T) -q^{-1}t^{-1} T \mathcal W^+(T) & (q^2-q^{-2})^{-1} \mathcal G(T)
	        \\
               (q^2-q^{-2})^{-1} \mathcal {\tilde G}(T) & q^{-1} T\mathcal W^-(T)-q t T \mathcal W^+(T)
                  \end{array}
              \right).
\end{align*}
\end{lemma}
\begin{proof} For the product of the above two matrices, each diagonal entry is equal to $\mathcal Z(t)$ by Definition \ref{lem:zV1} and Lemma \ref{lem:zV2}, and each off-diagonal entry is equal to zero
by Lemma \ref{lem:prep1}.
\end{proof}

\begin{definition}\label{def:calZ}
\rm Let $\mathcal Z$ denote the subalgebra of $\mathcal A_q$ generated by $\lbrace \mathcal Z_n \rbrace_{n=1}^\infty$.
\end{definition}
\noindent By Proposition \ref{prop:ZC}
the algebra $\mathcal Z$ is contained in the center of $\mathcal A_q$. Later in the paper we will show that $\mathcal Z$ is equal to the center of $\mathcal A_q$.
\medskip

\noindent Our next goal is to show that the algebra $\mathcal A_q$ is generated by  $\mathcal W_0, \mathcal W_1, \mathcal Z$.
The following result is a variation on \cite[Lemma~2.1]{basBel}.
\begin{lemma} \label{lem:Znform} For $n\in\mathbb N$,
\begin{align*}
\mathcal Z_n &= \lbrack 2 \rbrack_q \sum_{k=0}^{n-1} \mathcal W^\Downarrow_{-k} \mathcal W^\Downarrow_{n-k} q^{n-1-2k}
+\lbrack 2 \rbrack_q^{-1} \sum_{k=0}^{n-3} \mathcal W^\Downarrow_{n-2-k} \mathcal W^\Downarrow_{-k} q^{2k-n+3}
-  \sum_{k=0}^{n-2} \mathcal W^\Downarrow_{-k} \mathcal W^\Downarrow_{k-n+2} q^{n-2k}
\\
&\quad - \sum_{k=0}^{n-2} \mathcal W^\Downarrow_{k+1} \mathcal W^\Downarrow_{n-k-1}q^{n-2k-4}
 + (q^2-q^{-2})^{-2} \sum_{k=0}^n \mathcal G^\downarrow_k \mathcal {\tilde G}^\downarrow_{n-k} q^{n-2k}
\end{align*}
\end{lemma}
\begin{proof} Evaluate the $\mathcal Z(t)$ formula in Definition \ref{lem:zV1} using Lemma \ref{lem:ramp}.
\end{proof}
\begin{remark} \label{rem:charF} \rm At the beginning of this section we mentioned some central elements $\lbrace \Delta_n \rbrace_{n=1}^\infty$ in $\mathcal A_q$. These central elements
are computed in \cite[Lemma~2.1]{basBel}. 
Comparing Lemma \ref{lem:Znform} with \cite[Lemma~2.1]{basBel} we obtain
\begin{align*}
\Delta_n = - 2 \frac{q-q^{-1}}{q^n + q^{-n}} \mathcal Z_n \qquad \qquad n\geq 1.
\end{align*}
\noindent  If ${\rm Char}(\mathbb F)=2$ then $\Delta_n=0$ for $n\geq 1$.
\end{remark}

\begin{lemma} \label{lem:Z0} We have
\begin{align}
\label{eq:Z0}
\mathcal Z_0 = (q^2-q^{-2})^{-2} \mathcal G_0 \mathcal {\tilde G}_0 = \lbrack 2 \rbrack^2_q.
\end{align}
\end{lemma}
\begin{proof} Set $n=0$ in  Lemma \ref{lem:Znform}, and use \eqref{eq:GG0}.
\end{proof}
\noindent 
Now consider $\mathcal Z_n$ for $n\geq 1$. In the formula for $\mathcal Z_n$ from Lemma \ref{lem:Znform}, the alternating generators that show up are
$\lbrace \mathcal W_{-k} \rbrace_{k=0}^{n-1}$,
$\lbrace \mathcal W_{k+1} \rbrace_{k=0}^{n-1}$,
$\lbrace \mathcal G_{k+1} \rbrace_{k=0}^{n-1}$,
$\lbrace \mathcal {\tilde G}_{k+1} \rbrace_{k=0}^{n-1}$.
We now adjust $\mathcal Z_n$ by subtracting off the terms that involve $\mathcal G_n$ or $\mathcal  {\tilde G}_n$.
\begin{definition}
\rm
For $n\geq 1$ define
\begin{align}
\label{eq:Zadj}
\overline {\mathcal Z}_n = \mathcal Z_n - (q^2-q^{-2})^{-2} (\mathcal G_0 \mathcal {\tilde G}_n q^n + \mathcal G_n \mathcal {\tilde G}_0 q^{-n}).
\end{align}
\end{definition}
\begin{lemma} \label{lem:olZ} For $n\geq 1$,
\begin{align}
\overline {\mathcal Z}_n = \mathcal Z_n + \frac{\mathcal G_n q^{-n} + \mathcal {\tilde G}_n q^n}{q-q^{-1}}.
\label{eq:olZ}
\end{align}
\end{lemma}
\begin{proof} Evaluate \eqref{eq:Zadj} using \eqref{eq:GG0}.
\end{proof}

\noindent For $n\geq 1$ we now express $\overline {\mathcal Z}_n$  in terms of
$\lbrace \mathcal W_{-k} \rbrace_{k=0}^{n-1}$,
$\lbrace \mathcal W_{k+1} \rbrace_{k=0}^{n-1}$,
$\lbrace \mathcal G_{k+1} \rbrace_{k=0}^{n-2}$,
$\lbrace \mathcal {\tilde G}_{k+1} \rbrace_{k=0}^{n-2}$.

\begin{lemma} \label{lem:Zadj2} For $n\geq 1$,
\begin{align*}
\overline {\mathcal Z}_n &= \lbrack 2 \rbrack_q \sum_{k=0}^{n-1} \mathcal W^\Downarrow_{-k} \mathcal W^\Downarrow_{n-k} q^{n-1-2k}
+\lbrack 2 \rbrack_q^{-1} \sum_{k=0}^{n-3} \mathcal W^\Downarrow_{n-2-k} \mathcal W^\Downarrow_{-k} q^{2k-n+3}
-  \sum_{k=0}^{n-2} \mathcal W^\Downarrow_{-k} \mathcal W^\Downarrow_{k-n+2} q^{n-2k}
\\
&\quad - \sum_{k=0}^{n-2} \mathcal W^\Downarrow_{k+1} \mathcal W^\Downarrow_{n-k-1}q^{n-2k-4}
 + (q^2-q^{-2})^{-2} \sum_{k=1}^{n-1} \mathcal G^\downarrow_k \mathcal {\tilde G}^\downarrow_{n-k} q^{n-2k}
 \\
 &\quad - \frac{q^{-n}}{q-q^{-1}} \sum_{\ell=1}^{\lfloor (n-1) /2 \rfloor} (-1)^\ell \binom{n-1-\ell}{\ell} \lbrack 2 \rbrack^{-2\ell}_q \mathcal  G_{n-2\ell} \\
&\quad - \frac{q^n}{q-q^{-1}} \sum_{\ell=1}^{\lfloor (n-1) /2 \rfloor} (-1)^\ell \binom{n-1-\ell}{\ell} \lbrack 2 \rbrack^{-2\ell}_q \mathcal {\tilde G}_{n-2\ell}.
\end{align*}
\end{lemma}
\begin{proof} Evaluate \eqref{eq:Zadj} using Lemma \ref{lem:Znform} and \eqref{eq:GG0}, \eqref{eq:GDown}, \eqref{eq:tGDown}.
\end{proof}

\noindent The following result is a variation on \cite[Proposition~3.1]{basBel}.
\begin{lemma}
\label{lem:nrecgenCal}
Using the equations below, the alternating
generators of
$\mathcal A_q$
are recursively obtained from $\mathcal W_0, \mathcal W_1,
\mathcal Z_1, \mathcal Z_2, \ldots$ in the following order:
\begin{align*}
\mathcal W_0, \quad \mathcal W_1, \quad \mathcal G_1, 
\quad \mathcal {\tilde G}_1, \quad \mathcal  W_{-1}, 
\quad \mathcal W_2, \quad
\mathcal G_2, \quad \mathcal {\tilde G}_2, \quad 
\mathcal W_{-2}, \quad \mathcal W_3,  \quad \ldots
\end{align*}
For $n\geq 1$,
\begin{align}
\label{eq:Gsolve}
\mathcal G_n &= \frac{(\overline {\mathcal Z}_n-\mathcal Z_n)(q-q^{-1})- \lbrack \mathcal W_0, \mathcal W_n\rbrack q^n (q+q^{-1})}{q^n+q^{-n}},
\\
\mathcal {\tilde G}_n &= \mathcal G_n +(q+q^{-1}) \lbrack \mathcal W_0, \mathcal W_n\rbrack,
\label{eq:tGsolve}
\\
\mathcal W_{-n} &= \mathcal W_n -\frac{\lbrack \mathcal W_0, \mathcal G_n \rbrack_q}{(q^2-q^{-2})^2},
\label{eq:Wmsolve}
\\
\mathcal W_{n+1} &= \mathcal W_{1-n} -\frac{\lbrack \mathcal G_n, \mathcal W_1 \rbrack_q}{(q^2-q^{-2})^2}.
\label{eq:Wpsolve}
\end{align}
\end{lemma}
\begin{proof} Equation \eqref{eq:tGsolve} is from
\eqref{eq:3p1}. To get 
\eqref{eq:Gsolve}, eliminate $\mathcal {\tilde G}_n$ in \eqref{eq:olZ} using
\eqref{eq:tGsolve} and solve the resulting equation for $\mathcal G_n$.
The equations \eqref{eq:Wmsolve}, \eqref{eq:Wpsolve} are from \eqref{eq:3p2}, \eqref{eq:3p3}.
\end{proof}

\noindent The following result is a variation on \cite[Corollary~3.1]{basBel}.
\begin{corollary} \label{cor:ZZ}
The algebra $\mathcal A_q$ is generated by $\mathcal W_0, \mathcal W_1, \mathcal Z$.
\end{corollary}
\begin{proof} By Definition \ref{def:calZ} and
 Lemma
\ref{lem:nrecgenCal}.
\end{proof}

\noindent We now bring  in the filtration $\lbrace B_d\rbrace_{d\in \mathbb N}$ from Section 7.

\begin{lemma} \label{lem:Bn} For $n \in \mathbb N$,
\begin{enumerate}
\item[\rm (i)]  $\mathcal W^\Downarrow_{-n}, \mathcal W^\Downarrow_{n+1}\in B_{2n+1}$;
\item[\rm (ii)] $\mathcal G^\downarrow_n, \mathcal {\tilde G}^\downarrow_n\in B_{2n}$.
\end{enumerate}
\end{lemma}
\begin{proof} (i) By Definition \ref{def:degree}
and
\eqref{eq:WmDown},
 \eqref{eq:WpDown}.
 \\
 \noindent (ii) By Definition \ref{def:degree}
and
\eqref{eq:GDown},
 \eqref{eq:tGDown}.
 \end{proof}
\noindent For a subset $X$ of any vector space let ${\rm Span}(X)$ denote the subspace spanned by $X$.
\begin{lemma} \label{lem:Zbar} For $n \geq 1$,
\begin{align*}
\overline {\mathcal Z}_n \in B_{2n-1}+ \sum_{k=1}^{2n-1} {\rm Span} (B_k B_{2n-k}).
\end{align*}
\end{lemma}
\begin{proof}  Use Lemmas \ref{lem:Zadj2}, \ref{lem:Bn}.
\end{proof}

\begin{lemma}\label{lem:2n}
 For $n \in \mathbb N$ we have $\mathcal Z_n \in B_{2n}$. 
\end{lemma}
\begin{proof} For $n=0$ the result holds by Lemma  \ref{lem:Z0} and since $B_0=\mathbb F I$. For $n\geq 1$ we use
Lemma \ref{lem:olZ}. 
 We have $\overline {\mathcal Z}_n \in B_{2n}$ by
Lemma \ref{lem:Zbar} along with Lemma \ref{lem:123}(ii) and Proposition \ref{prop:filt}. By construction $\mathcal G_n \in B_{2n}$ and $ \mathcal {\tilde G}_n \in B_{2n}$.
 The result holds by these comments and Lemma
 \ref{lem:olZ}.
 \end{proof}
 
\begin{lemma} \label{lem:factor}
The following {\rm (i), (ii)} hold:
\begin{enumerate}
\item[\rm (i)] 
the subspace 
 $B_1$ has basis $1, \mathcal W_0, \mathcal W_1$;
\item[\rm (ii)] for $d\geq 2$ we have
\begin{align}
\label{eq:odd}
B_{d} =E_d+ B_{d-1} + \sum_{k=1}^{d-1} {\rm Span}(B_k B_{d-k}),
\end{align}
\noindent where $E_d=0$ if $d$ is odd and $E_d=\mathbb F \mathcal Z_n$ if $d=2n$ is even.
\end{enumerate}
\end{lemma}
\begin{proof}(i) The elements  $1, \mathcal W_0, \mathcal W_1$ are the irreducible $\mathcal A_q$-words with degree at most 1.
\\
\noindent (ii) Let $B'_d$ denote the vector space on the right in \eqref{eq:odd}. We show $B_d = B'_d$. We have $B_d \supseteq B'_d$ by Proposition \ref{prop:filt} and Lemma \ref{lem:2n}.
Next we show $B_d \subseteq B'_d$.
The vector space $B_{d}$ is spanned by the irreducible $\mathcal A_q$-words that have degree at most $d$.
Let $w=a_1a_2\cdots a_r$ denote an irreducible $\mathcal A_q$-word with degree at most $d$. We show $w \in B'_d$.
We may assume that $w$ has degree $d$;
otherwise $w\in B_{d-1}\subseteq B'_d$. Assume for the moment that $r\geq 2$. Write $w=w_1 w_2$ with $w_1=a_1$ and $w_2 = a_2\cdots a_r$. Let $k$ denote the degree of $w_1$.
By construction $w_2$ has degree $d-k$. Also by construction $k\geq 1$ and $d-k\geq 1$, so $1 \leq k \leq d-1$. We have $w \in B_k B_{d-k} \subseteq B'_d$.
Next assume that $r=1$, so $w=a_1$. Suppose that $d=2n+1$ is odd. Then $w=\mathcal W_{-n}$ or $w=\mathcal W_{n+1}$. In either case $w \in B_{d-1}+ {\rm Span}(B_1 B_{d-1}) + {\rm Span}(B_{d-1} B_1)\subseteq B'_d$ in 
view of \eqref{eq:Wmsolve}, \eqref{eq:Wpsolve}. Next suppose that $d=2n$ is even. Then $w=\mathcal G_n$ or $w=\mathcal {\tilde G}_n$.
By  \eqref{eq:Gsolve}, \eqref{eq:tGsolve} we have $w \in \mathbb F \mathcal Z_n + B_{d-1} + {\rm Span}(B_1 B_{d-1}) + {\rm Span}(B_{d-1}B_1) \subseteq B'_d$. We have shown $B_d \subseteq B'_d$. By the above comments $B_d=B'_d$.
\end{proof}

\noindent We saw earlier that $\mathcal A_q$ is generated by $\mathcal W_0, \mathcal W_1, \mathcal Z$ with $\mathcal Z$ central. Next we consider how $\mathcal W_0, \mathcal W_1$ are related.
\begin{lemma} \label{eq:DG} {\rm (See \cite[line (3.7)]{basBel}.)}
For the algebra $\mathcal A_q$,
\begin{align*}
&\lbrack \mathcal W_0, \lbrack \mathcal W_0, \lbrack \mathcal W_0, \mathcal W_1\rbrack_q \rbrack_{q^{-1}} \rbrack =(q^2 - q^{-2})^2 \lbrack \mathcal W_1, \mathcal W_0 \rbrack,
\\
&\lbrack \mathcal W_1, \lbrack \mathcal W_1, \lbrack \mathcal W_1, \mathcal W_0\rbrack_q \rbrack_{q^{-1}}\rbrack = (q^2-q^{-2})^2 \lbrack \mathcal W_0, \mathcal W_1 \rbrack.
\end{align*}
\end{lemma}

\noindent The above equations are the defining relations for the $q$-Onsager algebra. The relationship between $\mathcal A_q$ and the $q$-Onsager algebra is investigated in the next section.

\section{How $\mathcal A_q$ is related to the $q$-Onsager algebra}

\noindent In this section we describe how the algebra $\mathcal A_q$ is related to the $q$-Onsager algebra.
\begin{definition} \label{def:U} \rm
(See \cite[Section~2]{bas1}, \cite[Definition~3.9]{qSerre}.)
Define the algebra $O_q$ by generators $W_0$, $W_1$ and relations
\begin{align}
\label{eq:qOns1}
&\lbrack W_0, \lbrack W_0, \lbrack W_0, W_1\rbrack_q \rbrack_{q^{-1}} \rbrack =(q^2 - q^{-2})^2 \lbrack W_1, W_0 \rbrack,
\\
\label{eq:qOns2}
&\lbrack W_1, \lbrack W_1, \lbrack W_1, W_0\rbrack_q \rbrack_{q^{-1}}\rbrack = (q^2-q^{-2})^2 \lbrack W_0, W_1 \rbrack.
\end{align}
We call $O_q$ the {\it $q$-Onsager algebra}.
The relations \eqref{eq:qOns1}, \eqref{eq:qOns2}  are called the {\it $q$-Dolan/Grady relations}.
\end{definition}

\begin{definition}\label{def:poly} \rm
Let $\lbrace z_n \rbrace_{n=1}^\infty$ denote mutually commuting indeterminates. Let $\mathbb F \lbrack z_1, z_2, \ldots \rbrack$ denote
the algebra consisting of the polynomials in $z_1, z_2, \ldots $ that have all coefficients in $\mathbb F$.
For notational convenience define $z_0=1$.
\end{definition}
\noindent We define the algebra
\begin{align}
\mathcal O_q = O_q \otimes \mathbb F \lbrack z_1, z_2, \ldots \rbrack.
\label{eq:mainalg}
\end{align}
\noindent The algebra $\mathcal O_q$ is generated by
\begin{align*}
W_0 \otimes 1, \qquad \quad
W_1 \otimes 1,\qquad \quad 
\lbrace 1 \otimes z_n\rbrace_{n=1}^\infty.
\end{align*}

\begin{lemma} \label{lem:phi} 
There exists a unique algebra homomorphism $\phi: \mathcal O_q \to \mathcal A_q$ that sends
\begin{align*}
W_0 \otimes 1 \mapsto \mathcal W_0, \qquad \quad
W_1 \otimes 1 \mapsto \mathcal W_1, \qquad \quad 
1 \otimes z_n \mapsto \mathcal Z_n, \qquad n\geq 1.
\end{align*}
\noindent Moreover $\phi$ is surjective.
\end{lemma}
\begin{proof} The algebra homomorphism $\phi$ exists by Lemma \ref{eq:DG} and since $\lbrace \mathcal Z_n \rbrace_{n=1}^\infty$ are central in $\mathcal A_q$. The map $\phi$ is unique since
$\mathcal O_q $ is generated by $W_0\otimes 1$, $W_1\otimes 1$, $\lbrace 1\otimes z_n \rbrace_{n=1}^\infty$.
The map $\phi$ is surjective by Lemma \ref{def:calZ} and Corollary \ref{cor:ZZ}.
\end{proof}

\noindent We are going to show that $\phi$ is an algebra isomorphism. Our strategy is to employ a filtration of $\mathcal O_q$ that is mapped by $\phi$
to the filtration $\lbrace B_d \rbrace_{d\in \mathbb N}$ of $\mathcal A_q$. We will construct our filtration of $\mathcal O_q$ using
a filtration of $O_q$ and a grading of $\mathbb F \lbrack z_1, z_2, \ldots \rbrack$. 
\medskip

\noindent We recall from \cite[Section~4]{pospart}
a filtration of $O_q$. For notational convenience abbreviate $O=O_q$.
For $d \in \mathbb N$ let  $O_d$  denote the subspace of $O$
spanned by the products
$g_1g_2\cdots g_r$ such that $0 \leq r \leq d$ and
$g_i$ is among $W_0, W_1$ for $1 \leq i \leq r$.
\begin{lemma}\label{lem:Ofil} {\rm (See \cite[Section~4]{pospart}.)}
With the above notation, the following {\rm (i)--(iv)} hold:
\begin{enumerate}
\item[\rm (i)] $O_0 = \mathbb F 1$; 
\item[\rm (ii)] 
 $O_{d-1} \subseteq O_d$ for $d\geq 1$;
 \item[\rm (iii)] $O = \cup_{d \in \mathbb N} O_d$;
 \item[ \rm (iv)]  $O_r O_s \subseteq  O_{r+s}$
for $r,s \in \mathbb N$. 
\end{enumerate}
\end{lemma}
\noindent By Definition \ref{def:filtration} and Lemma \ref{lem:Ofil} the sequence
$\lbrace O_d \rbrace_{d \in \mathbb N}$ 
is a filtration of $O_q$.
\medskip

\noindent We now consider the dimensions of the subspaces $\lbrace O_d\rbrace_{d \in \mathbb N}$. For notational convenience define $O_{-1}=0$.
\begin{definition}\rm
 We define a generating function in the indeterminate $x$:
\begin{align*}
H(x) = \sum_{d \in \mathbb N} {\rm dim}(O_d/O_{d-1}) x^d.
\end{align*}
\end{definition}

\begin{lemma}\label{lem:genf2} {\rm (See \cite[line (39)]{shape} and \cite[Theorem~4.4]{pospart})} We have
\begin{align}
\label{eq:prod3}
H(x) = \prod_{i=1}^\infty \frac{1}{(1-x^{2i-1})^2} \frac{1}{1-x^{2i}}.
\end{align}
\end{lemma}
\noindent We turn our attention to the algebra $\mathbb F \lbrack z_1, z_2, \ldots \rbrack$. 
First we describe a basis for the vector space $\mathbb F \lbrack z_1, z_2, \ldots \rbrack$.
For $n \in \mathbb N$, a {\it partition of $n$} is a sequence $\lambda = \lbrace \lambda_i \rbrace_{i=1}^\infty$
of natural numbers such that $\lambda_i \geq \lambda_{i+1}$ for $i\geq 1$ and $n=\sum_{i=1}^\infty \lambda_i$.
The set $\Lambda_n$ consists of the partitions of $n$. Define $\Lambda = \cup_{n \in \mathbb N} \Lambda_n$. 
For $\lambda \in \Lambda$ define $z_\lambda = \prod_{i=1}^\infty z_{\lambda_i}$. The elements 
$\lbrace z_\lambda\rbrace_{\lambda \in \Lambda}$ form a basis for the vector space $\mathbb F \lbrack z_1, z_2, \ldots \rbrack$.
\medskip

\noindent
Next we construct a grading for the algebra $\mathbb F \lbrack z_1, z_2, \ldots \rbrack$.
For notational convenience abbreviate  $Z=\mathbb F \lbrack z_1, z_2, \ldots \rbrack$.
  For $n \in \mathbb N$ let $Z_n$ denote the subspace of
$Z$ with basis $\lbrace Z_\lambda \rbrace_{\lambda \in \Lambda_n}$. For example $ Z_0 = \mathbb F 1$.
The sum $Z = \sum_{n\in \mathbb N} Z_n$ is direct. Moreover $Z_r Z_s\subseteq Z_{r+s}$
for $r,s\in \mathbb N$. By these comments and Definition \ref{def:gr}
the subspaces $\lbrace Z_n \rbrace_{n\in \mathbb N}$ form
a grading of the algebra $Z$. 
Note that $z_n \in Z_n$ for $n \in \mathbb N$. 
\medskip

\noindent Next we consider the dimensions of the subspaces $\lbrace Z_n\rbrace_{n \in \mathbb N}$.
For $n \in \mathbb N$ let $p_n$ denote the number of partitions of $n$. By construction $p_n$ is the
dimension of the vector space $Z_n$. Define the generating function $P(x) = \sum_{n \in \mathbb N} p_n x^n$. The following result is well known; see for example \cite[Theorem~8.3.4]{bruIntro}.
\begin{align*}
P(x) = \prod_{i=1}^\infty \frac{1}{1-x^i}.
\end{align*}
\noindent Shortly we will consider $P(x^2)$. Note that
\begin{align}
P(x^2) = \prod_{i=1}^\infty \frac{1}{1-x^{2i}}.
\label{eq:Zgf}
\end{align}

\noindent We just constructed a filtration of $O_q$ and a grading of $\mathbb F\lbrack z_1, z_2, \ldots\rbrack$. We now combine these constructions to get a filtration of $O_q \otimes \mathbb F \lbrack z_1, z_2,\ldots \rbrack$.

\begin{definition}\label{def:cOn} \rm
 For notational convenience abbreviate $\mathcal O=\mathcal O_q$.
For $d \in \mathbb N$ define 
\begin{align}
\label{eq:skip}
\mathcal O_d = \sum_{k=0}^{\lfloor d/2 \rfloor}  O_{d-2k} \otimes Z_{k}.
\end{align}

\end{definition}
\begin{lemma}\label{lem:cOgrading} With reference to Definition \ref{def:cOn}, 
The following {\rm(i)--(iv)} hold:
\begin{enumerate}
\item[\rm (i)] $\mathcal O_0$ has basis $1\otimes 1$;
\item[\rm (ii)] $\mathcal O_{d-1} \subseteq \mathcal O_d$ for $d\geq 1$;
\item[\rm (iii)] $\mathcal O= \cup_{d \in \mathbb N} \mathcal O_d$;
\item[\rm (iv)] $\mathcal O_r \mathcal O_s \subseteq \mathcal O_{r+s} $ for $r,s\in \mathbb N$.
\end{enumerate}
\end{lemma}
\begin{proof} (i) Since $O_0=\mathbb F 1$ and $Z_0=\mathbb F 1$.
\\
\noindent (ii) By Lemma \ref{lem:Ofil}(ii) 
and
\eqref{eq:skip}.
\\
\noindent (iii) By Lemma \ref{lem:Ofil}(iii) and $Z=\sum_{n\in \mathbb N} Z_n$.
\\
\noindent (iv) By \eqref{eq:skip},
\begin{align*}
\mathcal O_r = \sum_{k=0}^{\lfloor r/2\rfloor} O_{r-2k} \otimes Z_k, \qquad \qquad
\mathcal O_s = \sum_{\ell=0}^{\lfloor s/2\rfloor} O_{s-2\ell} \otimes Z_\ell.
\end{align*}
\noindent Observe that
\begin{align*} 
      \mathcal O_r \mathcal O_s &\subseteq  \sum_{k=0}^{\lfloor r/2\rfloor} \sum_{\ell=0}^{\lfloor s/2 \rfloor} {\rm Span}\bigl((O_{r-2k} \otimes Z_k )(O_{s-2\ell} \otimes Z_\ell)\bigr)
      \\
     &= \sum_{k=0}^{\lfloor r/2\rfloor} \sum_{\ell=0}^{\lfloor s/2 \rfloor} {\rm Span}\bigl((O_{r-2k} O_{s-2\ell}) \otimes (Z_k Z_\ell)\bigr)
  \\
     &\subseteq \sum_{k=0}^{\lfloor r/2\rfloor} \sum_{\ell=0}^{\lfloor s/2 \rfloor} O_{r+s-2k-2\ell} \otimes Z_{k+\ell}
     \\
     &\subseteq \sum_{m=0}^{\lfloor (r+s)/2\rfloor} O_{r+s-2m} \otimes Z_m
     \\
     &= \mathcal O_{r+s}.
     \end{align*}
\end{proof}

\noindent By   Definition \ref{def:filtration} and Lemma \ref{lem:cOgrading}, the sequence $\lbrace \mathcal O_d \rbrace_{d \in \mathbb N}$ is a filtration of $\mathcal O_q$.
\medskip

\noindent We have a comment.
\begin{lemma}\label{lem:cm} For $n \in \mathbb N$,
 \begin{align} 
 \label{eq:cm}
  1 \otimes z_n \in 1 \otimes Z_n \subseteq \mathcal O_{2n}.
  \end{align}
  \end{lemma}
  \begin{proof} Set $d=2n$ in \eqref{eq:skip} and recall that $z_n \in Z_n$.
  \end{proof}

\noindent The next result is reminiscent of Lemma \ref{lem:factor}.

\begin{lemma} \label{lem:factor2}
With reference to Definition \ref{def:cOn},    the following {\rm (i), (ii)} hold:
\begin{enumerate}
\item[\rm (i)] 
the subspace 
 $\mathcal O_1$ has basis $1\otimes 1,  W_0\otimes 1, W_1\otimes 1$;
\item[\rm (ii)] for $d\geq 2$ we have
\begin{align}
\label{eq:odd2}
\mathcal O_{d} =F_d+ \mathcal O_{d-1} + \sum_{k=1}^{d-1} {\rm Span}(\mathcal O_k \mathcal O_{d-k}),
\end{align}
\noindent where $F_d=0$ if $d$ is odd and $F_d$ is spanned by $1\otimes z_n$ if $d=2n$ is even.
\end{enumerate}
\end{lemma}
\begin{proof} (i) Setting $d=1$ in \eqref{eq:skip} we obtain $\mathcal O_1 = O_1 \otimes Z_0$. The subspace $O_1$ has basis $1, W_0, W_1$ and the subspace $Z_0$ has basis $1$. The result follows.
\\
\noindent (ii) Let $\mathcal O'_d $ denote the vector space on the right in \eqref{eq:odd2}. We show that $\mathcal O_d = \mathcal O'_d$. We have
$\mathcal O_d \supseteq \mathcal O'_d$ by Lemma \ref{lem:cOgrading}(ii),(iv) and \eqref{eq:cm}.
Next we show that $\mathcal O_d \subseteq \mathcal O'_d$. In equation \eqref{eq:skip}, $\mathcal O_d$
is expressed as a sum. Each summand $O_{d-2k}\otimes Z_k$ is spanned by elements of the form $g_1g_2\cdots g_r\otimes z_\lambda$ 
such that $0 \leq r \leq d-2k$ and $g_i$ is among $W_0, W_1$ for $1 \leq i \leq r$ and $\lambda \in \Lambda_k$. Let $w=g_1g_2\cdots g_r\otimes z_\lambda$ denote one of these elements.
We show that $w \in \mathcal O'_d$. We may assume that $r=d-2k$; otherwise $w \in O_{d-1-2k}\otimes Z_k \subseteq \mathcal O_{d-1} \subseteq \mathcal O'_d$.
Assume for the moment that $r\geq 1$. Then $w=w_1w_2$ with $w_1 = g_1 \otimes 1$ and $w_2 = g_2 \cdots g_r \otimes z_\lambda$. In this case $w_1 \in O_1 \otimes Z_0 \subseteq \mathcal O_1$ and $w_2 \in O_{d-1-2k}\otimes Z_k\subseteq \mathcal O_{d-1}$, so
$w =w_1w_2\in \mathcal O_1 \mathcal O_{d-1} \subseteq \mathcal O'_d$.
For the rest of the proof, assume that $r=0$. The integer $d=2k$ is even and $k=n$. Note that $n\geq 1$ since $d\geq 2$. By construction $w = 1\otimes z_\lambda$ with $\lambda \in \Lambda_n$. Write $\lambda = \lbrace \lambda_i \rbrace_{i=1}^\infty$ and note that
$1\leq \lambda_1 \leq n$.
Assume for the moment that $\lambda_1\leq n-1$, and consider the partition $\mu=\lbrace \lambda_{i+1} \rbrace_{i=1}^\infty \in \Lambda_{n-\lambda_1}$.
We have $w=w_1 w_2$ with $w_1=1\otimes z_{\lambda_1}$ and $w_2 = 1 \otimes z_\mu$.
By this and  \eqref{eq:cm} we obtain $w_1 \in 1\otimes Z_{\lambda_1}\subseteq \mathcal O_{2\lambda_1}$ and $w_2 \in 1 \otimes Z_{n-\lambda_1}\subseteq  \mathcal O_{2(n-\lambda_1)}=\mathcal O_{d-2\lambda_1}$. Therefore $w=w_1w_2 \in \mathcal O_{2\lambda_1} \mathcal O_{d-2\lambda_1} \subseteq \mathcal O'_d$.
Next assume that $\lambda_1 = n$. Then $z_\lambda = z_n$, so $w=1 \otimes z_n \in F_d \subseteq \mathcal O'_d$.
We have shown that $w \in \mathcal O'_d$ in all cases, so $\mathcal O_d \subseteq \mathcal O'_d$. By the above comments $\mathcal O_d = \mathcal O'_d$.
\end{proof}

\noindent Next we consider the dimensions of the subspaces $\lbrace \mathcal O_d \rbrace_{d \in \mathbb N}$. For notational convenience define $\mathcal O_{-1} = 0$.

\begin{lemma} \label{lem:Odim}
We have
\begin{align*}
H(x) P(x^2)=
\sum_{d \in \mathbb N} {\rm dim}(\mathcal O_d/\mathcal O_{d-1}) x^d.
\end{align*}
\end{lemma}
\begin{proof} We have
\begin{align*}
H(x) = \sum_{k\in \mathbb N} {\rm dim}(O_k/O_{k-1}) x^k, \qquad \qquad P(x^2) = \sum_{\ell \in \mathbb N} {\rm dim}(Z_\ell) x^{2\ell}.
\end{align*}
For $d \in \mathbb N$ the coefficient of $x^d$ in $H(x) P(x^2)$ is equal to
\begin{align*}
\sum_{\stackrel{\scriptstyle k+2\ell=d,}{\scriptstyle k,\ell\geq 0}}
                     &{\rm dim}(O_k/O_{k-1}) {\rm dim} (Z_\ell)
\\
& = \sum_{\ell=0}^{\lfloor d/2 \rfloor} {\rm dim}(O_{d-2\ell}/O_{d-1-2\ell}) {\rm dim}(Z_\ell)
\\
&= \sum_{\ell=0}^{\lfloor d/2 \rfloor} \bigl( {\rm dim} (O_{d-2\ell}) - {\rm dim}(O_{d-1-2\ell}) \bigr)  {\rm dim}(Z_\ell)
\\
& ={\rm dim} \sum_{\ell=0}^{\lfloor d/2 \rfloor} O_{d-2\ell} \otimes Z_\ell- {\rm dim} \sum_{\ell=0}^{\lfloor d/2\rfloor} O_{d-1-2\ell}\otimes Z_\ell
\\
& ={\rm dim} \sum_{\ell=0}^{\lfloor d/2 \rfloor} O_{d-2\ell} \otimes Z_\ell- {\rm dim} \sum_{\ell=0}^{\lfloor (d-1)/2\rfloor} O_{d-1-2\ell}\otimes Z_\ell
\\
& ={\rm dim} (\mathcal O_d) - {\rm dim}(\mathcal O_{d-1})
\\
&= {\rm dim}(\mathcal O_d/ \mathcal O_{d-1}).
\end{align*}
\noindent The result follows.
\end{proof}

\begin{lemma} \label{lem:obs} We have 
\begin{align*}
 \mathcal H(x) = H(x) P(x^2).
 \end{align*}
\end{lemma}
\begin{proof} Compute $H(x)P(x^2)$ using 
\eqref{eq:prod3},
\eqref{eq:Zgf} and compare the result with \eqref{eq:prod2}.
\end{proof}

\begin{corollary} \label{cor:samed}
For $d \in \mathbb N$ the vector spaces $B_d$ and $\mathcal O_d$ have the same dimension.
\end{corollary}
\begin{proof} We use induction on $d$. The result holds for $d=0$ since $B_0$ and $\mathcal O_0$ have dimension 1.
Next assume that $d\geq 1$. By induction $B_{d-1}$ and $\mathcal O_{d-1}$ have the same dimension, so it suffices to
show that the quotient vector spaces $B_d/B_{d-1}$ and $\mathcal O_d/\mathcal O_{d-1}$ have the same dimension.
This is the case by \eqref{eq:grcharB} and Lemmas \ref{lem:Odim}, \ref{lem:obs}.
\end{proof}

\noindent 
Recall the algebra homomorphism $\phi : \mathcal O_q  \to \mathcal A_q$ from Lemma \ref{lem:phi}.
\medskip

\noindent  We now come to  our second main result.

\begin{theorem} \label{thm:2}
The map $\phi$ is an algebra isomorphism.
\end{theorem}
\begin{proof}  It suffices to show that $\phi$ is bijective. By Lemma \ref{lem:phi},  $\phi$ is surjective. In a moment we will show that $\phi$ is injective.
We  claim that $\phi(\mathcal O_d)=B_d$ for $d \in \mathbb N$. We now prove the claim by induction on $d$.
The claim holds for $d=0$ by Lemmas  \ref{lem:123}(i), \ref{lem:cOgrading}(i).
The claim holds for $d=1$ by Lemmas \ref{lem:factor}(i), \ref{lem:phi}, \ref{lem:factor2}(i).
For $d\geq 2$ we use induction along with \eqref{eq:odd}, \eqref{eq:odd2} and Lemma \ref{lem:phi} to obtain
\begin{align*} 
\phi(\mathcal O_d) &= 
\phi \biggl(F_d+ \mathcal O_{d-1} + \sum_{k=1}^{d-1} {\rm Span}(\mathcal O_k \mathcal O_{d-k})\biggr) \\
&=\phi(F_d)+ \phi(\mathcal O_{d-1}) + \sum_{k=1}^{d-1} {\rm Span}\bigl(\phi (\mathcal O_k)\phi( \mathcal O_{d-k})\bigr) \\
 &=E_d+ B_{d-1} + \sum_{k=1}^{d-1} {\rm Span}(B_k B_{d-k}) 
 \\
&=B_d.
\end{align*}
\noindent The claim is proved. We can now easily show that $\phi$ is injective. Let $K$ denote the kernel of $\phi$.
For $d \in \mathbb N$ we have $\phi(\mathcal O_d)= B_d$, so by Corollary \ref{cor:samed} the restriction of $\phi$ to $\mathcal O_d$ gives a bijection $\mathcal O_d \to B_d$, so
$\mathcal O_d \cap K=0$. By this and Lemma \ref{lem:cOgrading}(iii)  we obtain $K=0$. Therefore $\phi$ is injective.
We have shown that $\phi$ is bijective, and hence an algebra isomorphism.
\end{proof}

\section{Two subalgebras of $\mathcal A_q$}
 In this section we use Theorem \ref{thm:2} to investigate the following two subalgebras of $\mathcal A_q$: (i)  the center of $\mathcal A_q$; (ii) the subalgebra of $\mathcal A_q$ generated by $\mathcal W_0, \mathcal W_1$.

\medskip

\noindent Concerning (i), we will use the following fact.

\begin{lemma} \label{lem:center} {\rm (See \cite[Theorem~8.3]{kolb}.)}
The center of $O_q$ is equal to $\mathbb F 1$.
\end{lemma}

\noindent Recall the subalgebra $\mathcal Z$ of $\mathcal A_q$ from Definition \ref{def:calZ}.
\begin{theorem} \label{lem:c}
The following {\rm (i)--(iii)} hold:
\begin{enumerate}
\item[\rm (i)] the center of $\mathcal A_q$ is equal to $\mathcal Z$;
\item[\rm (ii)] there exists an algebra isomorphism $\mathbb F\lbrack z_1, z_2,\ldots \rbrack \to \mathcal Z$ that sends $z_n \mapsto \mathcal Z_n$ for $n\geq 1$;
\item[\rm (iii)] the elements $\lbrace \mathcal Z_n \rbrace_{n=1}^\infty$ are algebraically independent.
\end{enumerate}
\end{theorem}
\begin{proof} (i) By Lemma \ref{lem:center} and the construction, the center of $O_q \otimes \mathbb F\lbrack z_1, z_2,\ldots \rbrack$ is equal to
$1 \otimes \mathbb F\lbrack z_1, z_2,\ldots \rbrack$. Applying the algebra isomorphism $\phi$ we find that the center of $\mathcal A_q$ is equal to $\mathcal Z$.
\\
\noindent (ii) It suffices to display an injective algebra homomorphism $\mathbb F\lbrack z_1, z_2,\ldots \rbrack \to \mathcal A_q$ that sends $z_n \mapsto \mathcal Z_n$ for $n\geq 1$.
The following composition meets this description:

\begin{equation*}
{\begin{CD}
\mathbb F \lbrack z_1,z_2,\ldots \rbrack @>>z_n \mapsto 1\otimes z_n > O_q \otimes \mathbb F \lbrack z_1, z_2,\ldots \rbrack @>>\phi > \mathcal A_q.
         \end{CD}}
\end{equation*}
\noindent (iii) Immediate from (ii) above.
\end{proof}

\noindent Let $\langle \mathcal W_0, \mathcal W_1 \rangle$ 
 denote the subalgebra of $\mathcal A_q$ generated by $\mathcal W_0$, $\mathcal W_1$.
\medskip

\begin{theorem} \label{lem:ww}
There exists an algebra isomorphism $O_q \to  \langle \mathcal W_0, \mathcal W_1 \rangle$ that sends $W_0 \mapsto \mathcal W_0$ and $W_1 \mapsto \mathcal W_1$.
\end{theorem}
\begin{proof} It suffices to display an injective algebra homomorphism $O_q \to \mathcal A_q$ that sends  $W_0 \mapsto \mathcal W_0$ and $W_1 \mapsto \mathcal W_1$.
There exists an algebra homomorphism $\gamma: O_q \to O_q \otimes \mathbb F\lbrack z_1, z_2,\ldots \rbrack $ that sends $W_0 \mapsto W_0\otimes 1$ and $W_1 \mapsto W_1 \otimes 1$.
The map $\gamma$ is injective.
The following composition is an injective algebra homomorphism $O_q \to \mathcal A_q$ that sends  $W_0 \mapsto \mathcal W_0$ and $W_1 \mapsto \mathcal W_1$:
\begin{equation*}
{\begin{CD}
O_q @>>\gamma > O_q \otimes \mathbb F \lbrack z_1, z_2,\ldots \rbrack @>>\phi > \mathcal A_q.
         \end{CD}}
\end{equation*}
\end{proof}

\noindent Next we describe how the subalgebras $\langle \mathcal W_0, \mathcal W_1 \rangle$ and $\mathcal Z$ are related. 

\begin{theorem} 
\label{lem:UZ}
The multiplication map
\begin{align*}
\langle \mathcal W_0,\mathcal W_1\rangle \otimes \mathcal Z &\to
	       \mathcal A_q
	       \\
w \otimes z &\mapsto      wz            
\end{align*}
is an algebra isomorphism.
\end{theorem}
\begin{proof} Let $m$ denote the above multiplication map. The map $m$ is an algebra homomorphism since $\mathcal Z$ is central in $\mathcal A_q$.
Let
$a : O_q \to \langle \mathcal W_0,\mathcal W_1\rangle$ denote the algebra isomorphism from Theorem \ref{lem:ww}. 
Let $b:\mathbb F \lbrack z_1,z_2,\ldots \rbrack \to \mathcal Z$ denote the algebra isomorphism from Theorem \ref{lem:c}(ii). 
The following composition is an algebra isomorphism:

\begin{equation*}
{\begin{CD}
\langle \mathcal W_0,\mathcal W_1\rangle \otimes \mathcal Z  @>>a^{-1} \otimes b^{-1} > O_q \otimes \mathbb F \lbrack z_1, z_2,\ldots \rbrack @>>\phi > \mathcal A_q.
         \end{CD}}
\end{equation*}
\noindent The above map is equal to $m$ since it agrees with $m$ on the generators $\mathcal W_0 \otimes 1$, $\mathcal W_1 \otimes 1$, $\lbrace 1 \otimes \mathcal Z_n\rbrace_{n=1}^\infty$ of
$\langle \mathcal W_0,\mathcal W_1\rangle \otimes \mathcal Z$.
It follows that $m$ is an algebra isomorphism.
\end{proof}


\section{Acknowledgment} 
The author thanks Pascal Baseilhac for many discussions about the algebras $\mathcal A_q$ and $O_q$.
The author thanks Stefan Kolb for helpful comments about the center of $O_q$. The author thanks Travis Scrimshaw
for checking the relations in  Proposition \ref{lem:rra}
using SageMath \cite{sage}.

\newpage
\section{Appendix A}

\noindent In this appendix we list some relations that hold in $\mathcal A_q$. We will define an algebra $\mathcal A^\vee_q$ that is a
homomorphic preimage of $\mathcal A_q$. All the results in this appendix are about $\mathcal A^\vee_q$.
\medskip

\noindent 
Define the algebra $\mathcal A^\vee_q$ by generators 
\begin{align*}
\lbrace \mathcal W_{-k}\rbrace_{n\in \mathbb N}, \qquad  \lbrace \mathcal  W_{k+1}\rbrace_{n\in \mathbb N},\qquad  
 \lbrace \mathcal G_{k+1}\rbrace_{n\in \mathbb N},
\qquad
\lbrace \mathcal {\tilde G}_{k+1}\rbrace_{n\in \mathbb N}
\end{align*}
 and the following relations. For $k \in \mathbb N$,
\begin{align}
&
 \lbrack \mathcal W_0, \mathcal W_{k+1}\rbrack= 
\lbrack \mathcal W_{-k}, \mathcal W_{1}\rbrack=
({\mathcal{\tilde G}}_{k+1} - \mathcal G_{k+1})/(q+q^{-1}),
\label{eq:3p1App}
\\
&
\lbrack \mathcal W_0, \mathcal G_{k+1}\rbrack_q= 
\lbrack {\mathcal{\tilde G}}_{k+1}, \mathcal W_{0}\rbrack_q= 
\rho  \mathcal W_{-k-1}-\rho 
 \mathcal W_{k+1},
\label{eq:3p2App}
\\
&
\lbrack \mathcal G_{k+1}, \mathcal W_{1}\rbrack_q= 
\lbrack \mathcal W_{1}, {\mathcal {\tilde G}}_{k+1}\rbrack_q= 
\rho  \mathcal W_{k+2}-\rho 
 \mathcal W_{-k},
\label{eq:3p3App}
\\
&
\lbrack \mathcal W_{0}, \mathcal W_{-k}\rbrack=0,  \qquad 
\lbrack \mathcal W_{1}, \mathcal W_{k+1}\rbrack= 0.
\label{eq:3p4App}
\end{align}
\noindent Recall that $\rho =-(q^2-q^{-2})^2$, and define $\mathcal G_0$, $\mathcal {\tilde G}_0$ as in \eqref{eq:GG0}.
\medskip

\noindent The algebra $\mathcal A^\vee_q$ has an automorphism $\sigma$ and an antiautomorphism $\dagger$ that satisfy
Lemmas 
\ref{lem:aut}--\ref{lem:sdcom}. For $\mathcal A^\vee_q$ we define the generating functions $\mathcal W^+(t)$, $\mathcal W^-(t)$, $\mathcal G(t)$, $\mathcal {\tilde G}(t)$ as in
Definition \ref{def:gf4}.
In terms of these generating functions the relations \eqref{eq:3p1App}--\eqref{eq:3p4App} look as follows:
\begin{align}
& \label{eq:3pp1App}
\lbrack \mathcal W_0, \mathcal W^+(t) \rbrack = \lbrack \mathcal W^-(t), \mathcal W_1 \rbrack = t^{-1}(\mathcal {\tilde G}(t)-\mathcal G(t))/(q+q^{-1}),
\\
& \label{eq:3pp2App}
\lbrack \mathcal W_0, \mathcal G(t) \rbrack_q = \lbrack \mathcal {\tilde G}(t), \mathcal W_0 \rbrack_q = \rho \mathcal W^-(t)-\rho t \mathcal W^+(t),
\\
&\label{eq:3pp3App}
\lbrack \mathcal G(t), \mathcal W_1 \rbrack_q = \lbrack \mathcal W_1, \mathcal {\tilde G}(t) \rbrack_q = \rho \mathcal W^+(t) -\rho t \mathcal W^-(t),
\\
&\label{eq:3pp4App}
\lbrack  \mathcal W_0, \mathcal W^-(t) \rbrack = 0, 
\qquad 
\lbrack \mathcal W_1,  \mathcal W^+(t) \rbrack = 0.
\end{align}

\noindent Let $s$ denote an indeterminate that commutes with $t$. Define the generating functions
\begin{align*}
A(s,t)&=\lbrack  \mathcal W^-(s), \mathcal W^-(t) \rbrack ,
\\
B(s,t)&=\lbrack \mathcal W^+(s),  \mathcal W^+(t) \rbrack,
\\
C(s,t)&=\lbrack  \mathcal W^-(s), \mathcal W^+(t) \rbrack
+
\lbrack \mathcal W^+(s), \mathcal W^-(t) \rbrack,
\\
D(s,t)&=s \lbrack \mathcal  W^-(s), \mathcal G(t) \rbrack 
+
t \lbrack  \mathcal G(s),  \mathcal W^-(t) \rbrack ,
\\
E(s,t)&=s \lbrack  \mathcal W^-(s), \mathcal {\tilde G}(t) \rbrack 
+
t \lbrack  \mathcal {\tilde G}(s),  \mathcal W^-(t) \rbrack,
\\
F(s,t)&=s \lbrack  \mathcal  W^+(s),  \mathcal G(t) \rbrack
+
t \lbrack   \mathcal G(s), \mathcal W^+(t) \rbrack,
\\
G(s,t)&=s \lbrack   \mathcal W^+(s), \mathcal {\tilde G}(t) \rbrack
+
t \lbrack \mathcal {\tilde G}(s), \mathcal W^+(t) \rbrack,
\\
H(s,t)&=\lbrack   \mathcal G(s), \mathcal G(t) \rbrack, 
\\
I(s,t)&=\lbrack  \mathcal {\tilde G}(s), \mathcal  {\tilde G}(t) \rbrack,
\\
J(s,t)&=\lbrack   \mathcal {\tilde G}(s), \mathcal G(t) \rbrack +
\lbrack   \mathcal G(s), \mathcal {\tilde G}(t) \rbrack
\end{align*}
\noindent and also
\begin{align*}
K(s,t)&=\lbrack \mathcal W^-(s), \mathcal G(t) \rbrack_q -
\lbrack \mathcal W^-(t), \mathcal G(s) \rbrack_q -
s\lbrack \mathcal W^+(s), \mathcal G(t) \rbrack_q +
t\lbrack \mathcal W^+(t), \mathcal G(s) \rbrack_q,
\\
L(s,t)&=\lbrack \mathcal G(s), \mathcal W^+(t) \rbrack_q -
\lbrack \mathcal G(t), \mathcal W^+(s) \rbrack_q -
t\lbrack \mathcal G(s), \mathcal W^-(t) \rbrack_q +
s\lbrack \mathcal G(t), \mathcal W^-(s) \rbrack_q,
\\
M(s,t)&=\lbrack \mathcal {\tilde G}(s), \mathcal W^-(t) \rbrack_q -
\lbrack \mathcal {\tilde G}(t), \mathcal W^-(s) \rbrack_q -
t\lbrack \mathcal {\tilde G}(s), \mathcal W^+(t) \rbrack_q +
s\lbrack \mathcal {\tilde G}(t), \mathcal W^+(s) \rbrack_q,
\\
N(s,t)&=\lbrack \mathcal W^+(s), \mathcal {\tilde G}(t) \rbrack_q -
\lbrack \mathcal W^+(t), \mathcal {\tilde G}(s) \rbrack_q -
s\lbrack \mathcal W^-(s), \mathcal {\tilde G}(t) \rbrack_q +
t\lbrack \mathcal W^-(t), \mathcal {\tilde G}(s) \rbrack_q,
\\
P(s,t)&=\frac{t^{-1} \lbrack \mathcal G(s), \mathcal {\tilde G}(t) \rbrack - s^{-1} \lbrack \mathcal G(t),\mathcal {\tilde G}(s) \rbrack}{\rho (q+q^{-1})}
- \lbrack \mathcal W^-(t), \mathcal W^+(s)\rbrack_q +
 \lbrack \mathcal W^-(s), \mathcal W^+(t)\rbrack_q \nonumber
 \\
&\quad -st \lbrack \mathcal W^+(t), \mathcal W^-(s)\rbrack_q
+st \lbrack \mathcal W^+(s), \mathcal W^-(t)\rbrack_q 
-t\lbrack \mathcal W^-(s), \mathcal W^-(t) \rbrack_q 
\\
& \quad 
+s \lbrack \mathcal W^-(t), \mathcal W^-(s) \rbrack_q 
-s\lbrack \mathcal W^+(s), \mathcal W^+(t) \rbrack_q 
+t \lbrack \mathcal W^+(t), \mathcal W^+(s) \rbrack_q,
\\
Q(s,t)&=\frac{t^{-1} \lbrack \mathcal {\tilde G}(s), \mathcal  G(t) \rbrack - s^{-1} \lbrack \mathcal {\tilde G}(t),\mathcal G(s) \rbrack}{\rho (q+q^{-1})}
- \lbrack \mathcal W^+(t), \mathcal W^-(s)\rbrack_q +
 \lbrack \mathcal W^+(s), \mathcal W^-(t)\rbrack_q \nonumber
 \\
&\quad -st \lbrack \mathcal W^-(t), \mathcal W^+(s)\rbrack_q
+st \lbrack \mathcal W^-(s), \mathcal W^+(t)\rbrack_q
-t\lbrack \mathcal W^+(s), \mathcal W^+(t) \rbrack_q 
\\
& \quad 
+s \lbrack \mathcal W^+(t), \mathcal W^+(s) \rbrack_q 
-s\lbrack \mathcal W^-(s), \mathcal W^-(t) \rbrack_q 
+t \lbrack \mathcal W^-(t), \mathcal W^-(s) \rbrack_q,
\\
R(s,t)&=\lbrack \mathcal G(s), \mathcal {\tilde G}(t)\rbrack_q -
\lbrack \mathcal G(t), \mathcal {\tilde G}(s)\rbrack_q -(q+q^{-1})\rho t \lbrack \mathcal W^-(t), \mathcal W^+(s)\rbrack \nonumber
\\
& \qquad \qquad \qquad \qquad \qquad \qquad \qquad \qquad \quad 
+(q+q^{-1}) \rho s \lbrack \mathcal W^-(s), \mathcal W^+(t)\rbrack,
\\
S(s,t)&=
\lbrack \mathcal {\tilde G}(s), \mathcal  G(t)\rbrack_q -
\lbrack \mathcal {\tilde G}(t), \mathcal G(s)\rbrack_q - (q+q^{-1})\rho t \lbrack \mathcal W^+(t), \mathcal W^-(s)\rbrack \nonumber
\\
& \qquad \qquad \qquad \qquad \qquad \qquad \qquad \qquad \quad 
+(q+q^{-1}) \rho s \lbrack \mathcal W^+(s), \mathcal W^-(t)\rbrack.
\end{align*}
\noindent One checks that 
\begin{align*}
P(s,t)+ Q(s,t)&= (q+q^{-1})(1+st)C(s,t)+\frac{s^{-1}+t^{-1}}{(q+q^{-1})\rho} J(s,t) \nonumber
\\
& \qquad  \qquad -(q+q^{-1})(s+t)A(s,t) -(q+q^{-1})(s+t)B(s,t), 
\\
R(s,t)+S(s,t) &= \rho (q+q^{-1})(s+t) C(s,t) + (q+q^{-1}) J(s,t). 
\end{align*}
For  $\mathcal A^\vee_q$ the maps $\sigma$, $\dagger $ act on  $A(s,t), B(s,t), \ldots, S(s,t)$ as follow:
\bigskip

\centerline{
\begin{tabular}[t]{c|cccccc}
  $u$ & 
  $A(s,t)$ &   $B(s,t)$ &   $C(s,t)$ &   $D(s,t)$ &   $E(s,t)$ &   $F(s,t)$ 
   \\
\hline
$\sigma(u)$ &
 $B(s,t)$ &   $A(s,t)$ &   $C(s,t)$ &   $G(s,t)$ &   $F(s,t)$ &   $E(s,t)$ 
 \\
$\dagger(u)$ &
 $-A(s,t)$ &   $-B(s,t)$ &   $-C(s,t)$ &   $-E(s,t)$ &   $-D(s,t)$ &   $-G(s,t)$ 
	       \end{tabular}}
	       \bigskip
	       
	  \centerline{
\begin{tabular}[t]{c|cccccc}
  $u$ & 
  $G(s,t)$ &   $H(s,t)$ &   $I(s,t)$ &   $J(s,t)$ &   $K(s,t)$ &   $L(s,t)$ 
   \\
\hline
$\sigma(u)$ & $D(s,t)$ &   $I(s,t)$ &   $H(s,t)$ &   $J(s,t)$ &   $N(s,t)$ &   $M(s,t)$ 
\\
$\dagger(u)$ & $-F(s,t)$ &   $-I(s,t)$ &   $-H(s,t)$ &   $-J(s,t)$ &   $-M(s,t)$ &   $-N(s,t)$ 
	       \end{tabular}}
          \bigskip
	       
	  \centerline{
\begin{tabular}[t]{c|cccccc}
  $u$ & 
  $M(s,t)$ &   $N(s,t)$ &   $P(s,t)$ &   $Q(s,t)$ &   $R(s,t)$ &   $S(s,t)$ 
   \\
\hline
$\sigma(u)$ & $L(s,t)$ &   $K(s,t)$ &   $Q(s,t)$ &   $P(s,t)$ &   $S(s,t)$ &   $R(s,t)$ 
 \\
$\dagger(u)$ & $-K(s,t)$ &   $-L(s,t)$ &   $-Q(s,t)$ &   $-P(s,t)$ &   $-R(s,t)$ &   $-S(s,t)$ 
	       \end{tabular}}

\noindent  By \eqref{eq:3pp1App}--\eqref{eq:3pp4App} the following relations  hold in $\mathcal A^\vee_q$:
\begin{align*}
\lbrack \mathcal W_0, A(s,t)\rbrack&=0,
\\
\lbrack \mathcal W_0, B(s,t)\rbrack &= \frac{G(s,t)-F(s,t)}{st(q+q^{-1})},
\\
\lbrack \mathcal W_0, C(s,t)\rbrack &= \frac{E(s,t)-D(s,t)}{st(q+q^{-1})},
\\
\lbrack \mathcal W_0, D(s,t)\rbrack_q &= \rho (s+t)A(s,t)- \rho st C(s,t),
\\
\lbrack  E(s,t), \mathcal W_0 \rbrack_q &= \rho (s+t)A(s,t)- \rho st C(s,t),
\\
\lbrack \mathcal W_0, F(s,t)\rbrack_q &= \frac{S(s,t)}{q+q^{-1}}-H(s,t)-2\rho s t B(s,t),
\\
\lbrack G(s,t), \mathcal W_0 \rbrack_q &= \frac{S(s,t)}{q+q^{-1}}-I(s,t)-2\rho s t B(s,t),
\\
\lbrack \mathcal W_0, H(s,t)\rbrack_{q^2} &= \rho K(s,t),
\\
\lbrack I(s,t), \mathcal W_0 \rbrack_{q^2} &= \rho M(s,t),
\\
\lbrack \mathcal W_0, J(s,t)\rbrack &= \rho M(s,t)-\rho K(s,t)
\end{align*}
\noindent and also
\begin{align*}
\lbrack \mathcal W_0, K(s,t) \rbrack_q &= \frac{q^2+q^{-2}}{q+q^{-1}}H(s,t)+(q+q^{-1}) \rho A(s,t) + (q+q^{-1}) \rho s t B(s,t)
\\
&\qquad \qquad \qquad \qquad \qquad \qquad \qquad \qquad \qquad  +\frac{J(s,t)}{q+q^{-1}}-S(s,t),
\\
\lbrack \mathcal W_0, L(s,t)\rbrack_q &= \rho P(s,t)-\frac{s^{-1}+t^{-1}}{q+q^{-1}}H(s,t),
\\
\lbrack M(s,t), \mathcal W_0 \rbrack_q &= \frac{q^2+q^{-2}}{q+q^{-1}}I(s,t)+(q+q^{-1}) \rho A(s,t) + (q+q^{-1}) \rho s t B(s,t)
\\
&\qquad \qquad \qquad \qquad \qquad \qquad \qquad \qquad \qquad  +\frac{J(s,t)}{q+q^{-1}}-S(s,t),
\\
\lbrack N(s,t), \mathcal W_0 \rbrack_q &= \rho Q(s,t)-\frac{s^{-1}+t^{-1}}{q+q^{-1}}I(s,t),
\\
\lbrack P(s,t),\mathcal W_0\rbrack &= (s^{-1}+t^{-1})G(s,t)-(1+s^{-1}t^{-1})E(s,t)\\
&\qquad \qquad \qquad \qquad \qquad  +\frac{(s^{-1}+t^{-1})K(s,t)-L(s,t)-N(s,t)}{q+q^{-1}},
\\
\lbrack \mathcal W_0, Q(s,t) \rbrack &= (s^{-1}+t^{-1})F(s,t)-(1+s^{-1}t^{-1})D(s,t)
\\
&\qquad \qquad \qquad \qquad \qquad  +\frac{(s^{-1}+t^{-1})M(s,t)-L(s,t)-N(s,t)}{q+q^{-1}},
\\
\lbrack \mathcal W_0, R(s,t) \rbrack &= \rho(s^{-1}+ t^{-1}) E(s,t) -\rho (s^{-1}+t^{-1}) D(s,t)+ \rho F(s,t)-\rho G(s,t),
\\
\lbrack \mathcal W_0, S(s,t)\rbrack &= \rho G(s,t)-\rho F(s,t)-(q+q^{-1})\rho K(s,t)+(q+q^{-1})\rho M(s,t).
\end{align*}
\newpage
\noindent For the previous 18 equations we apply $\sigma$ to each side, and obtain the following relations that hold in $\mathcal A^\vee_q$:
\begin{align*}
\lbrack \mathcal W_1, A(s,t)\rbrack &= \frac{D(s,t)-E(s,t)}{st(q+q^{-1})},
\\
\lbrack \mathcal W_1, B(s,t)\rbrack&=0,
\\
\lbrack \mathcal W_1, C(s,t)\rbrack &= \frac{F(s,t)-G(s,t)}{st(q+q^{-1})},
\\
\lbrack D(s,t), \mathcal W_1 \rbrack_q &= \frac{R(s,t)}{q+q^{-1}}-H(s,t)-2\rho s t A(s,t),
\\
\lbrack \mathcal W_1, E(s,t)\rbrack_q &= \frac{R(s,t)}{q+q^{-1}}-I(s,t)-2\rho s t A(s,t),
\\
\lbrack  F(s,t), \mathcal W_1 \rbrack_q &= \rho (s+t)B(s,t)- \rho st C(s,t),
\\
\lbrack \mathcal W_1, G(s,t)\rbrack_q &= \rho (s+t)B(s,t)- \rho st C(s,t),
\\
\lbrack H(s,t), \mathcal W_1 \rbrack_{q^2} &= \rho L(s,t),
\\
\lbrack \mathcal W_1, I(s,t)\rbrack_{q^2} &= \rho N(s,t),
\\
\lbrack \mathcal W_1, J(s,t)\rbrack &= \rho L(s,t)-\rho N(s,t)
\end{align*}
\noindent and also
\begin{align*}
\lbrack K(s,t), \mathcal W_1 \rbrack_q &= \rho P(s,t)-\frac{s^{-1}+t^{-1}}{q+q^{-1}}H(s,t),
\\
\lbrack L(s,t), \mathcal W_1 \rbrack_q &= \frac{q^2+q^{-2}}{q+q^{-1}}H(s,t)+(q+q^{-1}) \rho B(s,t) + (q+q^{-1}) \rho s t A(s,t)
\\
&\qquad \qquad \qquad \qquad \qquad \qquad \qquad \qquad \qquad  +\frac{J(s,t)}{q+q^{-1}}-R(s,t),
\\
\lbrack \mathcal W_1, M(s,t)\rbrack_q &= \rho Q(s,t)-\frac{s^{-1}+t^{-1}}{q+q^{-1}}I(s,t),
\\
\lbrack \mathcal W_1, N(s,t) \rbrack_q &= \frac{q^2+q^{-2}}{q+q^{-1}}I(s,t)+(q+q^{-1}) \rho B(s,t) + (q+q^{-1}) \rho s t A(s,t)
\\
&\qquad \qquad \qquad \qquad \qquad \qquad \qquad \qquad \qquad  +\frac{J(s,t)}{q+q^{-1}}-R(s,t),
\\
\lbrack \mathcal W_1, P(s,t) \rbrack &= (s^{-1}+t^{-1})E(s,t)-(1+s^{-1}t^{-1})G(s,t)\\
&\qquad \qquad \qquad \qquad \qquad  +\frac{(s^{-1}+t^{-1})L(s,t)-M(s,t)-K(s,t)}{q+q^{-1}},
\\
\lbrack Q(s,t),\mathcal W_1\rbrack &= (s^{-1}+t^{-1})D(s,t)-(1+s^{-1}t^{-1})F(s,t)\\
&\qquad \qquad \qquad \qquad \qquad  +\frac{(s^{-1}+t^{-1})N(s,t)-M(s,t)-K(s,t)}{q+q^{-1}},
\\
\lbrack \mathcal W_1, R(s,t)\rbrack &= \rho D(s,t)-\rho E(s,t)-(q+q^{-1})\rho N(s,t)+(q+q^{-1})\rho L(s,t),
\\
\lbrack \mathcal W_1, S(s,t) \rbrack &= \rho(s^{-1}+ t^{-1}) F(s,t) -\rho (s^{-1}+t^{-1}) G(s,t)+ \rho E(s,t)-\rho D(s,t).
\end{align*}

\newpage
\section{Appendix B}
\noindent In this appendix  we describe the elements 
\begin{align*}
\lbrace \mathcal W^\Downarrow_{-n}\rbrace_{n\in\mathbb N}, \quad
\lbrace \mathcal W^\Downarrow_{n+1}\rbrace_{n\in \mathbb N}, \quad 
\lbrace \mathcal G^\downarrow_{n}\rbrace_{n\in \mathbb N}, \quad
\lbrace \mathcal {\tilde G}^\downarrow_{n}\rbrace_{n \in \mathbb N}
\end{align*}
\noindent  that appeared in Section 8. For $n \in \mathbb N$,
\begin{align}
\label{eq:WmDown}
 \mathcal W^\Downarrow_{-n} &= \sum_{\ell=0}^{\lfloor n /2 \rfloor} (-1)^\ell \binom{n-\ell}{\ell} \lbrack 2 \rbrack^{-2\ell}_q \mathcal W_{2\ell-n},
 \\
 \label{eq:WpDown}
 \mathcal W^\Downarrow_{n+1} &= \sum_{\ell=0}^{\lfloor n /2 \rfloor} (-1)^\ell \binom{n-\ell}{\ell} \lbrack 2 \rbrack^{-2\ell}_q \mathcal W_{n-2\ell+1}.
\end{align}
\noindent In the tables below, we display $\mathcal W^\Downarrow_{-n}$ and
$\mathcal W^\Downarrow_{n+1}$ for $0 \leq n \leq 8$.
\bigskip
 
\centerline{
\begin{tabular}[t]{c|c}
  $n$ & $\mathcal W^\Downarrow_{-n}$
   \\
\hline
$0$ & $\mathcal W_{0} $ \\
$1$ & $\mathcal W_{-1} $ \\
$2$ & $\mathcal W_{-2} - \mathcal W_{0} \lbrack 2 \rbrack^{-2}_q $ \\
$3$ & $\mathcal W_{-3} -2 \mathcal W_{-1} \lbrack 2 \rbrack^{-2}_q$ \\
$4$ & $\mathcal W_{-4} -3 \mathcal W_{-2} \lbrack 2 \rbrack^{-2}_q + \mathcal W_{0} \lbrack 2 \rbrack^{-4}_q $ \\
$5$ & $\mathcal W_{-5} -4 \mathcal W_{-3} \lbrack 2 \rbrack^{-2}_q + 3\mathcal W_{-1} \lbrack 2 \rbrack^{-4}_q $ \\
$6$ & $\mathcal W_{-6} -5 \mathcal W_{-4} \lbrack 2 \rbrack^{-2}_q + 6\mathcal W_{-2} \lbrack 2 \rbrack^{-4}_q - \mathcal W_{0} \lbrack 2 \rbrack^{-6}_q $ \\
$7$ & $\mathcal W_{-7} -6 \mathcal W_{-5} \lbrack 2 \rbrack^{-2}_q + 10\mathcal W_{-3} \lbrack 2 \rbrack^{-4}_q - 4\mathcal W_{-1} \lbrack 2 \rbrack^{-6}_q $ \\
$8$ & $\mathcal W_{-8} -7 \mathcal W_{-6} \lbrack 2 \rbrack^{-2}_q + 15\mathcal W_{-4} \lbrack 2 \rbrack^{-4}_q - 10\mathcal W_{-2} \lbrack 2 \rbrack^{-6}_q +\mathcal W_{0} \lbrack 2 \rbrack^{-8}_q$ 
\end{tabular}}
\bigskip
 
\centerline{
\begin{tabular}[t]{c|c}
  $n$ & $\mathcal W^\Downarrow_{n+1}$
   \\
\hline
$0$ & $\mathcal W_{1} $ \\
$1$ & $\mathcal W_{2} $ \\
$2$ & $\mathcal W_{3} - \mathcal W_{1} \lbrack 2 \rbrack^{-2}_q $ \\
$3$ & $\mathcal W_{4} -2 \mathcal W_{2} \lbrack 2 \rbrack^{-2}_q$ \\
$4$ & $\mathcal W_{5} -3 \mathcal W_{3} \lbrack 2 \rbrack^{-2}_q + \mathcal W_{1} \lbrack 2 \rbrack^{-4}_q $ \\
$5$ & $\mathcal W_{6} -4 \mathcal W_{4} \lbrack 2 \rbrack^{-2}_q + 3\mathcal W_{2} \lbrack 2 \rbrack^{-4}_q $ \\
$6$ & $\mathcal W_{7} -5 \mathcal W_{5} \lbrack 2 \rbrack^{-2}_q + 6\mathcal W_{3} \lbrack 2 \rbrack^{-4}_q - \mathcal W_{1} \lbrack 2 \rbrack^{-6}_q $ \\
$7$ & $\mathcal W_{8} -6 \mathcal W_{6} \lbrack 2 \rbrack^{-2}_q + 10\mathcal W_{4} \lbrack 2 \rbrack^{-4}_q - 4\mathcal W_{2} \lbrack 2 \rbrack^{-6}_q $ \\
$8$ & $\mathcal W_{9} -7 \mathcal W_{7} \lbrack 2 \rbrack^{-2}_q + 15\mathcal W_{5} \lbrack 2 \rbrack^{-4}_q - 10\mathcal W_{3} \lbrack 2 \rbrack^{-6}_q +\mathcal W_{1} \lbrack 2 \rbrack^{-8}_q$ 
\end{tabular}}
\bigskip

\noindent Recall that
\begin{align*} 
\mathcal G^\downarrow_0=\mathcal G_0= -(q-q^{-1}) \lbrack 2 \rbrack^2_q, \qquad \qquad
\mathcal {\tilde G}^\downarrow_0=\mathcal {\tilde G}_0= -(q-q^{-1}) \lbrack 2 \rbrack^2_q.
\end{align*}

\noindent For $n\geq 1$,
\begin{align}\label{eq:GDown}
 \mathcal G^\downarrow_n = \sum_{\ell=0}^{\lfloor (n-1) /2 \rfloor} (-1)^\ell \binom{n-1-\ell}{\ell} \lbrack 2 \rbrack^{-2\ell}_q \mathcal G_{n-2\ell},\\
 \label{eq:tGDown}
 \mathcal {\tilde G}^\downarrow_n = \sum_{\ell=0}^{\lfloor (n-1) /2 \rfloor} (-1)^\ell \binom{n-1-\ell}{\ell} \lbrack 2 \rbrack^{-2\ell}_q \mathcal {\tilde G}_{n-2\ell}.
\end{align}
\noindent In the tables below, we display $\mathcal G^\downarrow_n$ and $\mathcal {\tilde G}^\downarrow_n$ for $1 \leq n \leq 9$.
\bigskip

\centerline{
\begin{tabular}[t]{c|c}
  $n$ & $\mathcal G^\downarrow_{n}$
   \\
\hline
$1$ & $\mathcal G_{1} $ \\
$2$ & $\mathcal G_{2} $ \\
$3$ & $\mathcal G_{3} - \mathcal G_{1} \lbrack 2 \rbrack^{-2}_q $ \\
$4$ & $\mathcal G_{4} -2 \mathcal G_{2} \lbrack 2 \rbrack^{-2}_q$ \\
$5$ & $\mathcal G_{5} -3 \mathcal G_{3} \lbrack 2 \rbrack^{-2}_q + \mathcal G_{1} \lbrack 2 \rbrack^{-4}_q $ \\
$6$ & $\mathcal G_{6} -4 \mathcal G_{4} \lbrack 2 \rbrack^{-2}_q + 3\mathcal G_{2} \lbrack 2 \rbrack^{-4}_q $ \\
$7$ & $\mathcal G_{7} -5 \mathcal G_{5} \lbrack 2 \rbrack^{-2}_q + 6\mathcal G_{3} \lbrack 2 \rbrack^{-4}_q - \mathcal G_{1} \lbrack 2 \rbrack^{-6}_q $ \\
$8$ & $\mathcal G_{8} -6 \mathcal G_{6} \lbrack 2 \rbrack^{-2}_q + 10\mathcal G_{4} \lbrack 2 \rbrack^{-4}_q - 4\mathcal G_{2} \lbrack 2 \rbrack^{-6}_q $ \\
$9$ & $\mathcal G_{9} -7 \mathcal G_{7} \lbrack 2 \rbrack^{-2}_q + 15\mathcal G_{5} \lbrack 2 \rbrack^{-4}_q - 10\mathcal G_{3} \lbrack 2 \rbrack^{-6}_q +\mathcal G_{1} \lbrack 2 \rbrack^{-8}_q$ 
\end{tabular}}
\bigskip

\centerline{
\begin{tabular}[t]{c|c}
  $n$ & $\mathcal {\tilde G}^\downarrow_{n}$
   \\
\hline
$1$ & $\mathcal {\tilde G}_{1} $ \\
$2$ & $\mathcal {\tilde G}_{2} $ \\
$3$ & $\mathcal {\tilde G}_{3} - \mathcal {\tilde G}_{1} \lbrack 2 \rbrack^{-2}_q $ \\
$4$ & $\mathcal {\tilde G}_{4} -2 \mathcal {\tilde G}_{2} \lbrack 2 \rbrack^{-2}_q$ \\
$5$ & $\mathcal {\tilde G}_{5} -3 \mathcal {\tilde G}_{3} \lbrack 2 \rbrack^{-2}_q + \mathcal {\tilde G}_{1} \lbrack 2 \rbrack^{-4}_q $ \\
$6$ & $\mathcal {\tilde G}_{6} -4 \mathcal {\tilde G}_{4} \lbrack 2 \rbrack^{-2}_q + 3\mathcal {\tilde G}_{2} \lbrack 2 \rbrack^{-4}_q $ \\
$7$ & $\mathcal {\tilde G}_{7} -5 \mathcal {\tilde G}_{5} \lbrack 2 \rbrack^{-2}_q + 6\mathcal {\tilde G}_{3} \lbrack 2 \rbrack^{-4}_q - \mathcal {\tilde G}_{1} \lbrack 2 \rbrack^{-6}_q $ \\
$8$ & $\mathcal {\tilde G}_{8} -6 \mathcal {\tilde G}_{6} \lbrack 2 \rbrack^{-2}_q + 10\mathcal {\tilde G}_{4} \lbrack 2 \rbrack^{-4}_q - 4\mathcal {\tilde G}_{2} \lbrack 2 \rbrack^{-6}_q $ \\
$9$ & $\mathcal {\tilde G}_{9} -7 \mathcal {\tilde G}_{7} \lbrack 2 \rbrack^{-2}_q + 15\mathcal {\tilde G}_{5} \lbrack 2 \rbrack^{-4}_q - 10\mathcal {\tilde G}_{3} \lbrack 2 \rbrack^{-6}_q +\mathcal {\tilde G}_{1} \lbrack 2 \rbrack^{-8}_q$ 
\end{tabular}}
\bigskip

\newpage
\section{Appendix C}

In this appendix we give some details from the proof of Theorem \ref{thm:rr}. In that proof we invoke the
Bergman diamond lemma. In our discussion of that lemma we  list the overlap ambiguities; there are four types (i)--(iv).  Our goal in this appendix is to show that all the overlap ambiguities are resolvable. 
 Our strategy is to express the overlap ambiguities in terms of generating functions that involve mutually commuting indeterminates $r,s,t$. There are four overlap ambiguities of type (i). Here is the first one.
Using the GF reduction rules, let us evaluate the overlap ambiguity  
\begin{align}
\label{eq:One}
\mathcal W^+(t) \mathcal W^-(s) \mathcal G(r).
\end{align}
We can proceed in two ways. If we evaluate  $\mathcal W^+(t) \mathcal W^-(s)$ first, then we find that \eqref{eq:One} is equal to 
a weighted sum with the following terms and coefficients:
\bigskip
 
\centerline{
\begin{tabular}[t]{c|c}
  {\rm term} & {\rm coefficient} 
   \\
\hline
$\mathcal G(s) \mathcal G(t)  \mathcal {\tilde G}(r)$ & $-a'_{r, t}b'_{t, s}e_{t, r} - a'_{r, t}B'_{t, s}e_{s, r} + A'_{r, t}b'_{r, s}e_{r, t}$\\
$\mathcal G(r) \mathcal G(t)  \mathcal {\tilde G}(s)$ &$-e_{t, s} + a'_{r, t}B'_{t, s}e_{s, r} + A'_{r, t}B'_{r, s}e_{s, t}$\\
$\mathcal G(r) \mathcal G(s)  \mathcal {\tilde G}(t)$ & $e_{t, s} + a'_{r, t}b'_{t, s}e_{t, r} - A'_{r, t}b'_{r, s}e_{r, t} - A'_{r, t}B'_{r, s}e_{s, t}$\\
\hline
$\mathcal G(r) \mathcal W^-(s) \mathcal W^+(t) $ & $b_{r, s}a_{r, t}$\\
$\mathcal G(r) \mathcal W^-(t) \mathcal W^+(s) $ & $A'_{r, t}B'_{r, s}$\\
$\mathcal G(s) \mathcal W^-(t) \mathcal W^+(r) $ & $e_{t, s}f_{t, r} + A_{r, t}B_{t, s} + A'_{r, t}b'_{r, s}$\\
$\mathcal G(s) \mathcal W^-(r) \mathcal W^+(t) $ & $-e_{t, s}f_{t, r} + a_{r, t}B_{r, s} + a'_{r, t}b'_{t, s}$ \\
$\mathcal G(t) \mathcal W^-(r) \mathcal W^+(s) $ &$e_{t, s}f_{s, r} + a'_{r, t}B'_{t, s}$ \\
$\mathcal G(t) \mathcal W^-(s) \mathcal W^+(r) $ & $-e_{t, s}f_{s, r} + A_{r, t}b_{t, s}$\\
\hline
$\mathcal G(r) \mathcal W^-(s) \mathcal W^-(t) $ &$A'_{r, t}b_{r, s}$\\
$\mathcal G(s) \mathcal W^-(r) \mathcal W^-(t) $ &$-e_{t, s}F_{t, r} + a'_{r, t}B_{t, s} + A'_{r, t}B_{r, s}$ \\
$\mathcal G(t) \mathcal W^-(r) \mathcal W^-(s) $ &$e_{t, s}F_{s, r} + a'_{r, t}b_{t, s}$\\
\hline
$\mathcal G(r) \mathcal W^+(s) \mathcal W^+(t) $ &$a_{r, t}B'_{r, s}$ \\
$\mathcal G(s) \mathcal W^+(r) \mathcal W^+(t) $ & $e_{t, s}F_{t, r} + A_{r, t}b'_{t, s} + a_{r, t}b'_{r, s}$\\
$\mathcal G(t) \mathcal W^+(r) \mathcal W^+(s) $&$-e_{t, s}F_{s, r} + A_{r, t}B'_{t, s}$ 
\end{tabular}}
\bigskip
\newpage
\noindent If we  evaluate  $\mathcal W^-(s) \mathcal G(r)$ first, then we find that \eqref{eq:One} is equal to a weighted sum with the following terms and coefficients:
\bigskip
 
\centerline{
\begin{tabular}[t]{c|c}
  {\rm term} & {\rm coefficient} 
   \\
\hline
$\mathcal G(s) \mathcal G(t)  \mathcal {\tilde G}(r)$ &$-B_{r, s}A_{s, t}e_{s, r} - B_{r, s}a_{s, t}e_{t, r} + b_{r, s}A_{r, t}e_{r, s}$\\
$\mathcal G(r) \mathcal G(t)  \mathcal {\tilde G}(s)$ &$B_{r, s}A_{s, t}e_{s, r} - b_{r, s}A_{r, t}e_{r, s} - b_{r, s}a_{r, t}e_{t, s}$\\
$\mathcal G(r) \mathcal G(s)  \mathcal {\tilde G}(t)$ & $B_{r, s}a_{s, t}e_{t, r} + b_{r, s}a_{r, t}e_{t, s}$\\
\hline
$\mathcal G(r) \mathcal W^-(s) \mathcal W^+(t) $ & $b_{r, s}a_{r, t}$\\
$\mathcal G(r) \mathcal W^-(t) \mathcal W^+(s) $ & $B'_{r, s}A'_{r, t}$\\
$\mathcal G(s) \mathcal W^-(t) \mathcal W^+(r) $ & $b'_{r, s}A'_{s, t}$\\
$\mathcal G(s) \mathcal W^-(r) \mathcal W^+(t) $ & $B_{r, s}a_{s, t}$ \\
$\mathcal G(t) \mathcal W^-(r) \mathcal W^+(s) $ &$B'_{r, s}a'_{r, t} + B_{r, s}A_{s, t}$ \\
$\mathcal G(t) \mathcal W^-(s) \mathcal W^+(r) $ & $b'_{r, s}a'_{s, t} + b_{r, s}A_{r, t}$\\
\hline
$\mathcal G(r) \mathcal W^-(s) \mathcal W^-(t) $ &$b_{r, s}A'_{r, t}$ \\
$\mathcal G(s) \mathcal W^-(r) \mathcal W^-(t) $ &$B_{r, s}A'_{s, t}$\\
$\mathcal G(t) \mathcal W^-(r) \mathcal W^-(s) $ &$B_{r, s}a'_{s, t} + b_{r, s}a'_{r, t}$\\
\hline
$\mathcal G(r) \mathcal W^+(s) \mathcal W^+(t) $ &$B'_{r, s}a_{r, t}$ \\
$\mathcal G(s) \mathcal W^+(r) \mathcal W^+(t) $ & $b'_{r, s}a_{s, t}$\\
$\mathcal G(t) \mathcal W^+(r) \mathcal W^+(s) $&$b'_{r, s}A_{s, t} + B'_{r, s}A_{r, t}$ 
\end{tabular}}
\bigskip

\noindent Referring to the previous two tables, for each row the given coefficients are equal and their common value is displayed below:
 \bigskip
 
\centerline{
\begin{tabular}[t]{c|c}
  {\rm term} & {\rm coefficient} 
   \\
\hline
$\mathcal G(s) \mathcal G(t)  \mathcal {\tilde G}(r)$ &$0$\\
$\mathcal G(r) \mathcal G(t)  \mathcal {\tilde G}(s)$ &$\frac{q^4}{(q - 1)(q + 1)(q^2 + 1)^3(-t + s)}$\\
$\mathcal G(r) \mathcal G(s)  \mathcal {\tilde G}(t)$ & $\frac{-q^4}{(q - 1)(q + 1)(q^2 + 1)^3(-t + s)}$\\
\hline
$\mathcal G(r) \mathcal W^-(s) \mathcal W^+(t) $ & $\frac{-(q^2r - t)(q^2s - r)}{(r - t)(r - s)q^2}$\\
$\mathcal G(r) \mathcal W^-(t) \mathcal W^+(s) $ & $\frac{-r^2ts(q - 1)^2(q + 1)^2}{(r - t)(r - s)q^2}$\\
$\mathcal G(s) \mathcal W^-(t) \mathcal W^+(r) $ & $\frac{r^2ts(q - 1)^2(q + 1)^2}{q^2(r - s)(-t + s)}$\\
$\mathcal G(s) \mathcal W^-(r) \mathcal W^+(t) $ & $\frac{r(q - 1)(q + 1)(q^2s - t)}{q^2(r - s)(-t + s)}$ \\
$\mathcal G(t) \mathcal W^-(r) \mathcal W^+(s) $ &$\frac{(q - 1)^2(q + 1)^2(r^2s - r^2t - r + t)rs}{q^2(r - s)(r - t)(-t + s)}$ \\
$\mathcal G(t) \mathcal W^-(s) \mathcal W^+(r) $ & $\frac{-r(q - 1)(q + 1)(q^2r^2s^2 - q^2rs^2t - q^2s^2 + q^2st - r^2s^2 + rs^2t + rs - rt)}{q^2(r - s)(r - t)(-t + s)}$\\
\hline
$\mathcal G(r) \mathcal W^-(s) \mathcal W^-(t) $ &$\frac{rt(q - 1)(q + 1)(q^2s - r)}{(r - t)(r - s)q^2}$ \\
$\mathcal G(s) \mathcal W^-(r) \mathcal W^-(t) $ &$\frac{-str(q - 1)^2(q + 1)^2}{q^2(r - s)(-t + s)}$\\
$\mathcal G(t) \mathcal W^-(r) \mathcal W^-(s) $ &$\frac{r(q - 1)(q + 1)(q^2st + rs - rt - st)}{(r - t)(-t + s)q^2}$\\
\hline
$\mathcal G(r) \mathcal W^+(s) \mathcal W^+(t) $ &$\frac{rs(q - 1)(q + 1)(q^2r - t)}{(r - t)(r - s)q^2}$ \\
$\mathcal G(s) \mathcal W^+(r) \mathcal W^+(t) $ & $\frac{-r^2(q - 1)(q + 1)(q^2s - t)}{q^2(r - s)(-t + s)}$\\
$\mathcal G(t) \mathcal W^+(r) \mathcal W^+(s) $&$\frac{(q - 1)^2(q + 1)^2r^2s}{(r - t)(-t + s)q^2}$ 
\end{tabular}}

\bigskip

\noindent The overlap ambiguity \eqref{eq:One} is resolvable.

\newpage
\noindent Next we evaluate the overlap ambiguity  
\begin{align}
\label{eq:Two}
\mathcal {\tilde{G}}(t) \mathcal W^-(s) \mathcal G(r).
\end{align}
 We can proceed in two ways. If we evaluate  $\mathcal {\tilde{G}}(t) \mathcal W^-(s)$ first, then we find that \eqref{eq:Two} is equal to a weighted sum with the following terms and coefficients:
 \bigskip
 
\centerline{
\begin{tabular}[t]{c|c}
  {\rm term} & {\rm coefficient} 
   \\
\hline
$\mathcal G(r) \mathcal W^-(s) \mathcal {\tilde G}(t)$ & $b_{t, s} b_{r, s} + B'_{t, s} A'_{r, s} - f_{s, r} b'_{t, s} e_{t, r} A'_{t, s} - B'_{t, s} f_{t, r} e_{s, r} a'_{s, t} - F_{s, r} b'_{t, s} e_{t, r} b_{r, s}$ \\
$\mathcal G(r) \mathcal W^-(t) \mathcal {\tilde G}(s)$ & $B_{t, s} b_{r, t} + b'_{t, s} A'_{r, t} - f_{s, r} b'_{t, s} e_{t, r} a'_{t, s} - B'_{t, s} f_{t, r} e_{s, r} A'_{s, t} - B'_{t, s} F_{t, r} e_{s, r} b_{r, t} $\\
$\mathcal G(s) \mathcal W^-(t) \mathcal {\tilde G}(r)$ & $f_{s, r} b'_{t, s} e_{t, s} a'_{t, r} - F_{s, r} b'_{t, s} e_{t, s} B_{t, r} - B'_{t, s} f_{t, r} e_{s, t} a'_{t, r} + B'_{t, s} f_{t, r} e_{s, r} A'_{r, t}$\\ & $ + B'_{t, s} F_{t, r} e_{s, t} B_{t, r} + F_{s, r} b'_{t, s} e_{t, r} B_{t, s} + B'_{t, s} F_{t, r} e_{s, r}b_{s, t}$\\
$\mathcal G(s) \mathcal W^-(r) \mathcal {\tilde G}(t)$ &$b_{t, s}B_{r, s} + B'_{t, s}a'_{r, s} + f_{s, r}b'_{t, s}e_{t, s}A'_{t, r} - F_{s, r}b'_{t, s}e_{t, s}b_{t, r} - B'_{t, s}f_{t, r}e_{s, t}A'_{t, r}$\\&$+ B'_{t, s}f_{t, r}e_{s, r}a'_{r, t} + B'_{t, s}F_{t, r}e_{s, t}b_{t, r} - F_{s, r}b'_{t, s}e_{t, r}B_{r, s}$\\
$\mathcal G(t) \mathcal W^-(r) \mathcal {\tilde G}(s)$ &$B_{t, s}B_{r, t} + b'_{t, s}a'_{r, t} - f_{s, r}b'_{t, s}e_{t, s}A'_{s, r} + f_{s, r}b'_{t, s}e_{t, r}a'_{r, s} + F_{s, r}b'_{t, s}e_{t, s}b_{s, r}$\\
\end{tabular}}
\newpage
\noindent
 If we evaluate $\mathcal W^-(s) \mathcal G(r)$ first, then we find that \eqref{eq:Two} is equal to a weighted sum with the following terms and coefficients:
 \bigskip
 
\centerline{
\begin{tabular}[t]{c|c}
  {\rm term} & {\rm coefficient} 
   \\
\hline
$\mathcal G(r) \mathcal W^-(s) \mathcal {\tilde G}(t)$ &$b_{r, s}b_{t, s} + B'_{r, s}A'_{t, s} - B_{r, s}f_{t, s}e_{t, r}b_{r, s} + B_{r, s}F_{t, s}e_{t, r}A'_{t, s} - b_{r, s}f_{t, r}e_{t, s}B_{s, r}$\\
$\mathcal G(r) \mathcal W^-(t) \mathcal {\tilde G}(s)$ & $b_{r, s}B_{t, s} + B'_{r, s}a'_{t, s} + B_{r, s}f_{t, s}e_{s, r}b_{r, t} + B_{r, s}F_{t, s}e_{s, r}A'_{r, t} - b_{r, s}f_{t, r}e_{r, s}b_{r, t} $\\&$+ B_{r, s}F_{t, s}e_{t, r}a'_{t, s} + b_{r, s}f_{t, r}e_{t, s}B_{t, r} - b_{r, s}F_{t, r}e_{r, s}A'_{r, t}$\\
$\mathcal G(s) \mathcal W^-(t) \mathcal {\tilde G}(r)$ & $B_{r, s}B_{t, r} + b'_{r, s}a'_{t, r} - B_{r, s}f_{t, s}e_{s, r}b_{s, t} + B_{r, s}f_{t, s}e_{t, r}B_{t, s} - B_{r, s}F_{t, s}e_{s, r}A'_{s, t}$\\&$ + b_{r, s}f_{t, r}e_{r, s}b_{s, t} + b_{r, s}F_{t, r}e_{r, s}A'_{s, t} + b_{r, s}F_{t, r}e_{t, s}a'_{t, r}$\\
$\mathcal G(s) \mathcal W^-(r) \mathcal {\tilde G}(t)$ &$B_{r, s}b_{t, r} + b'_{r, s}A'_{t, r} - B_{r, s}f_{t, s}e_{t, r}B_{r, s} + b_{r, s}F_{t, r}e_{t, s}A'_{t, r} - b_{r, s}f_{t, r}e_{t, s}b_{s, r}$\\
$\mathcal G(t) \mathcal W^-(r) \mathcal {\tilde G}(s)$ &$B_{r, s}f_{t, s}e_{s, r}B_{r, t} + B_{r, s}F_{t, s}e_{s, r}a'_{r, t} - b_{r, s}f_{t, r}e_{r, s}B_{r, t} - B_{r, s}F_{t, s}e_{t, r}a'_{r, s}$\\
\end{tabular}}

\newpage

\noindent Referring to the previous two tables, for each row the given coefficients are equal and their common value is displayed below:
 \bigskip
 
\centerline{
\begin{tabular}[t]{c|c}
  {\rm term} & {\rm coefficient} 
   \\
\hline
$\mathcal G(r) \mathcal W^-(s) \mathcal {\tilde G}(t)$ &$\frac{q^{10}rs^2t^2 - q^8r^2s^2t - 2q^8rs^2t^2 + 2q^6r^2s^2t - q^6r^2st^2 + q^6rs^2t^2 + q^6rs^2 - q^6s^2t - q^4r^2s^2t}{q^6(t-r)(s-t)(r-s)}$\\
&$+\frac{ 3q^4r^2st^2 - q^4r^2s + q^4st^2 - 3q^2r^2st^2 + q^2r^2t - q^2rt^2 + r^2st^2}{q^6(t - r)(s - t)(r - s)}$\\
$\mathcal G(r) \mathcal W^-(t) \mathcal {\tilde G}(s)$ &$\frac{ t(q^2 - 1)(q^8rst^2 - q^6r^2t^2 - q^6rt^2 - q^4r^2st + q^4r^2t^2 + q^4rt - q^4st + 2q^2r^2st - q^2r^2 + q^2rs - r^2st)}{q^6(-t + r)(s - t)(r - s)}$\\
$\mathcal G(s) \mathcal W^-(t) \mathcal {\tilde G}(r)$ &$\frac{-rst^2(q^2 - 1)^3(q^4t - r)}{q^6(-t + r)(s - t)(r - s)}$\\
$\mathcal G(s) \mathcal W^-(r) \mathcal {\tilde G}(t)$ &$\frac{r(q^2 - 1)(q^8rst^2 - q^6r^2st - q^6rst^2 + q^4r^2st - q^4rst^2 + q^4rs - q^4st + 2q^2rst^2 - q^2rt + q^2t^2 - rst^2)}{q^6(-t + r)(s - t)(r - s)}$\\
$\mathcal G(t) \mathcal W^-(r) \mathcal {\tilde G}(s)$ &$\frac{-(q^2 - 1)^2(q^6rst - q^4r^2t - q^2rst + q^2r - q^2s + rst)rt}{q^6(-t + r)(s - t)(r - s)}$\\
$\mathcal G(t) \mathcal W^-(s) \mathcal {\tilde G}(r)$ &$\frac{rst^2(q^2 - 1)^3(q^4s - r)}{q^6(-t + r)(s - t)(r - s)}$\\
\end{tabular}}
\bigskip

\noindent The overlap ambiguity \eqref{eq:Two} is resolvable.

\newpage
\noindent Next we evaluate the overlap ambiguity  
\begin{align}
\label{eq:Three}
\mathcal {\tilde{G}}(t) \mathcal W^+(s) \mathcal G(r).
\end{align}
 We can proceed in two ways. If we evaluate  $\mathcal {\tilde{G}}(t) \mathcal W^+(s)$ first, then we find that \eqref{eq:Three} is equal to a weighted sum with the following terms and coefficients:
 \bigskip
 
\centerline{
\begin{tabular}[t]{c|c}
  {\rm term} & {\rm coefficient} 
   \\
\hline
$\mathcal G(r) \mathcal W^-(s) \mathcal {\tilde G}(t)$ & $a_{t, s}A'_{r, s} - A_{t, s}f_{s, r}e_{t, r}A'_{t, s} - a_{t, s}f_{t, r}e_{s, r}a'_{s, t} - A_{t, s}F_{s, r}e_{t, r}b_{r, s} + A'_{t, s}b_{r, s}$ \\
$\mathcal G(r) \mathcal W^-(t) \mathcal {\tilde G}(s)$ & $ a'_{t, s}b_{r, t} + A_{t, s}A'_{r, t} - A_{t, s}f_{s, r}e_{t, r}a'_{t, s} - a_{t, s}f_{t, r}e_{s, r}A'_{s, t} - a_{t, s}F_{t, r}e_{s, r}b_{r, t}$\\
$\mathcal G(s) \mathcal W^-(t) \mathcal {\tilde G}(r)$ & $A_{t, s}f_{s, r}e_{t, s}a'_{t, r} - A_{t, s}F_{s, r}e_{t, s}B_{t, r} - a_{t, s}f_{t, r}e_{s, t}a'_{t, r} + a_{t, s}f_{t, r}e_{s, r}A'_{r, t} $\\&$+ a_{t, s}F_{t, r}e_{s, t}B_{t, r} + A_{t, s}F_{s, r}e_{t, r}B_{t, s} + a_{t, s}F_{t, r}e_{s, r}b_{s, t}$\\
$\mathcal G(s) \mathcal W^-(r) \mathcal {\tilde G}(t)$ &$A'_{t, s}B_{r, s} + a_{t, s}a'_{r, s} + A_{t, s}f_{s, r}e_{t, s}A'_{t, r} - A_{t, s}F_{s, r}e_{t, s}b_{t, r} - a_{t, s}f_{t, r}e_{s, t}A'_{t, r} $\\&$+ a_{t, s}f_{t, r}e_{s, r}a'_{r, t} + a_{t, s}F_{t, r}e_{s, t}b_{t, r} - A_{t, s}F_{s, r}e_{t, r}B_{r, s}$\\
$\mathcal G(t) \mathcal W^-(r) \mathcal {\tilde G}(s)$ &$a'_{t, s}B_{r, t} + A_{t, s}a'_{r, t} - A_{t, s}f_{s, r}e_{t, s}A'_{s, r} + A_{t, s}f_{s, r}e_{t, r}a'_{r, s} + A_{t, s}F_{s, r}e_{t, s}b_{s, r} $\\
\end{tabular}}
\newpage

 \noindent If we evaluate  $\mathcal W^+(s)\mathcal G(r)$ first, then we find that \eqref{eq:Three} is equal to a weighted sum with the following terms and coefficients:
 \bigskip
 

\centerline{
\begin{tabular}[t]{c|c}
  {\rm term} & {\rm coefficient} 
   \\
\hline
$\mathcal G(r) \mathcal W^-(s) \mathcal {\tilde G}(t)$ & $A'_{r, s}b_{t, s} + a_{r, s}A'_{t, s}       -a'_{r, s}f_{t, s}e_{t, r}b_{r,s}    -A'_{r, s}f_{t, r}e_{t, s}B_{s,r} +  a'_{r, s}F_{t, s}e_{t, r}A'_{t,s}   $ \\
$\mathcal G(r) \mathcal W^-(t) \mathcal {\tilde G}(s)$ & $A'_{r, s}B_{t, s} + a_{r, s}a'_{t, s} + a'_{r, s}f_{t, s}e_{s, r}b_{r, t}  +  a'_{r, s}F_{t, s}e_{s, r}A'_{r,t}   -A'_{r, s}f_{t, r}e_{r, s}b_{r,t} $\\&$+  A'_{r, s}f_{t, r}e_{t, s}B_{t,r}    -A'_{r, s}F_{t, r}e_{r, s}A'_{r,t} + a'_{r, s}F_{t, s}e_{t, r}a'_{t,s}$\\
$\mathcal G(s) \mathcal W^-(t) \mathcal {\tilde G}(r)$ & $a'_{r, s}B_{t, r} + A_{r, s}a'_{t, r}       -a'_{r, s}f_{t, s}e_{s, r}b_{s,t}  + a'_{r, s}f_{t, s}e_{t, r}B_{t,s}     -a'_{r, s}F_{t, s}e_{s, r}A'_{s,t}$\\&$+ A'_{r, s}f_{t, r}e_{r, s}b_{s,t}+ A'_{r, s}F_{t, r}e_{r, s}A'_{s,t} +  A'_{r, s}F_{t, r}e_{t, s}a'_{t,r}$\\
$\mathcal G(s) \mathcal W^-(r) \mathcal {\tilde G}(t)$ &$a'_{r, s}b_{t, r} + A_{r, s}A'_{t, r}          -a'_{r, s}f_{t, s}e_{t, r}B_{r,s}    -A'_{r, s}f_{t, r}e_{t, s}b_{s,r} +  A'_{r, s}F_{t, r}e_{t, s}A'_{t,r}  $\\
$\mathcal G(t) \mathcal W^-(r) \mathcal {\tilde G}(s)$ &$a'_{r, s}f_{t, s}e_{s, r}B_{r, t}     +   a'_{r, s}F_{t, s}e_{s, r}a'_{r,t}      -A'_{r, s}f_{t, r}e_{r, s}B_{r,t} +   A'_{r, s}f_{t, r}e_{t, s}b_{t,r}   $\\&$  -A'_{r, s}F_{t, r}e_{r, s}a'_{r,t}  -a'_{r, s}F_{t, s}e_{t, r}a'_{r,s}  -A'_{r, s}F_{t, r}e_{t, s}A'_{s,r}$\\
$\mathcal G(t) \mathcal W^-(s) \mathcal {\tilde G}(r)$ &$           -a'_{r, s}f_{t, s}e_{s, r} B_{s,t}      +              a'_{r, s}f_{t, s}e_{t, r}b_{t,s}    -a'_{r, s}F_{t, s}e_{s, r}a'_{s,t}  +  A'_{r, s}f_{t, r}e_{r, s}B_{s,t}  $\\&$+    A'_{r, s}F_{t, r}e_{r, s}a'_{s,t}   -a'_{r, s}F_{t, s}e_{t, r}A'_{r,s}   -A'_{r, s}F_{t, r}e_{t, s}a'_{s,r} $\\
\hline
$\mathcal G(r) \mathcal W^+(s) \mathcal {\tilde G}(t)$ &$ A'_{r, s}B'_{t, s} + a_{r, s}a_{t, s}         -a'_{r, s}f_{t, s}e_{t, r}B'_{r,s}    -A'_{r, s}f_{t, r}e_{t, s}b'_{s,r}   +  a'_{r, s}F_{t, s}e_{t, r}a_{t,s}      $\\
$\mathcal G(r) \mathcal W^+(t) \mathcal {\tilde G}(s)$ &$A'_{r, s}b'_{t, s} + a_{r, s}A_{t, s} + a'_{r, s}f_{t, s}e_{s, r}B'_{r, t}  +     a'_{r, s}F_{t, s}e_{s, r}a_{r,t}     -A'_{r, s}f_{t, r}e_{r, s}B'_{r,t}  $\\&$+A'_{r, s}f_{t, r}e_{t, s}b'_{t,r}   -A'_{r, s}F_{t, r}e_{r, s}a_{r,t}   + a'_{r, s}F_{t, s}e_{t, r}A_{t,s} $\\
$\mathcal G(s) \mathcal W^+(t) \mathcal {\tilde G}(r)$ &$ a'_{r, s}b'_{t, r} + A_{r, s}A_{t, r}        -a'_{r, s}f_{t, s}e_{s, r} B'_{s,t}   +  a'_{r, s}f_{t, s}e_{t, r}b'_{t,s}     -a'_{r, s}F_{t, s}e_{s, r}a_{s,t} $\\&$+ A'_{r, s}f_{t, r}e_{r, s}B'_{s,t}  +  A'_{r, s}F_{t, r}e_{r, s}a_{s,t}  +A'_{r, s}F_{t, r}e_{t, s}A_{t,r}  $\\
$\mathcal G(s) \mathcal W^+(r) \mathcal {\tilde G}(t)$ &$a'_{r, s}B'_{t, r} + A_{r, s}a_{t, r}             -a'_{r, s}f_{t, s}e_{t, r}b'_{r,s}    -A'_{r, s}f_{t, r}e_{t, s}B'_{s,r} +   A'_{r, s}F_{t, r}e_{t, s}a_{t,r}        $\\
$\mathcal G(t) \mathcal W^+(r) \mathcal {\tilde G}(s)$ &$ a'_{r, s}f_{t, s}e_{s, r}b'_{r, t}       +      a'_{r, s}F_{t, s}e_{s, r}A_{r,t}      -A'_{r, s}f_{t, r}e_{r, s}b'_{r,t}  +    A'_{r, s}f_{t, r}e_{t, s}B'_{t,r}  $\\&$  -A'_{r, s}F_{t, r}e_{r, s}A_{r,t}   -a'_{r, s}F_{t, s}e_{t, r}A_{r,s}  -A'_{r, s}F_{t, r}e_{t, s}a_{s,r}$\\
$\mathcal G(t) \mathcal W^+(s) \mathcal {\tilde G}(r)$ &$      -a'_{r, s}f_{t, s}e_{s, r} b'_{s,t}   +      a'_{r, s}f_{t, s}e_{t, r}B'_{t,s}     -a'_{r, s}F_{t, s}e_{s, r}A_{s,t}  +  A'_{r, s}f_{t, r}e_{r, s}b'_{s,t}  $\\&$ +   A'_{r, s}F_{t, r}e_{r, s}A_{s,t}     -a'_{r, s}F_{t, s}e_{t, r}a_{r,s}    -A'_{r, s}F_{t, r}e_{t, s}A_{s,r}    $\\
\hline
$\mathcal W^-(r) \mathcal W^-(s) \mathcal W^-(t)$ &$-a'_{r, s}F_{t, s} - A'_{r, s}F_{t, r}                             $\\
$\mathcal W^-(s) \mathcal W^-(t) \mathcal W^+(r)$ &$-A_{r, s}F_{t, s} + A'_{r, s}f_{t, r}                      $\\
$\mathcal W^-(r) \mathcal W^-(t) \mathcal W^+(s)$ &$-a_{r, s}F_{t, r} + a'_{r, s}f_{t, s}                     $\\
$\mathcal W^-(r) \mathcal W^-(s) \mathcal W^+(t)$ &$-a'_{r, s}f_{t, s} - A'_{r, s}f_{t, r}                  $\\
$\mathcal W^-(r) \mathcal W^+(s) \mathcal W^+(t)$ &$-a_{r, s}f_{t, r} + a'_{r, s}F_{t, s}                  $\\
$\mathcal W^-(s) \mathcal W^+(r) \mathcal W^+(t)$ &$-A_{r, s}f_{t, s} + A'_{r, s}F_{t, r}                $\\
$\mathcal W^-(t) \mathcal W^+(r) \mathcal W^+(s)$ &$A_{r, s}f_{t, s} + a_{r, s}f_{t, r}                     $\\
$\mathcal W^+(r) \mathcal W^+(s) \mathcal W^+(t)$ & $A_{r, s}F_{t, s} + a_{r, s}F_{t, r}                     $
\end{tabular}}
\newpage

\noindent 
Referring to the previous two tables, for each row the given coefficients are equal and their common value is displayed below:
 \bigskip
 
\centerline{
\begin{tabular}[t]{c|c}
  {\rm term} & {\rm coefficient} 
   \\
\hline
$\mathcal G(r) \mathcal W^-(s) \mathcal {\tilde G}(t)$ & $\frac{s(q^2 - 1)(q^8r^2st^2 - 2q^6r^2st^2 - q^6r^2t + q^6rt^2 + q^4r^2st^2 + q^4r^2s - q^4r^2t - q^4st^2 + q^2r^2t + q^2rt^2 - r^2t}{q^4(t - r)(s - t)(r - s)}$ \\
$\mathcal G(r) \mathcal W^-(t) \mathcal {\tilde G}(s)$ & $\frac{t(q^2 - 1)(q^8r^2st^2 - 2q^6r^2st^2 - q^6r^2t + q^6rt^2 + q^4r^2st^2 - q^4st^2 + q^2r^2s + q^2rst - r^2s)}{q^4(-t + r)(s - t)(r - s)}$\\
$\mathcal G(s) \mathcal W^-(t) \mathcal {\tilde G}(r)$ & $\frac{(q^2 - 1)^2(q^6rst^2 - q^4rst^2 - q^4st + q^4t^2 - q^2rt + rs)tr}{q^4(t - r)(s - t)(r - s)}$\\
$\mathcal G(s) \mathcal W^-(r) \mathcal {\tilde G}(t)$ &$\frac{(q^2 - 1)(q^8rst^2 - 2q^6rst^2 - q^6st + q^6t^2 + q^4rst^2 + q^4rs - q^4rt - q^4t^2 + q^2st + q^2t^2 - st)r^2}{q^4(-t + r)(s - t)(r - s)}$\\
$\mathcal G(t) \mathcal W^-(r) \mathcal {\tilde G}(s)$ &$\frac{(q^2 - 1)^3(q^4rt - 1)tr^2s}{q^4(t - r)(s - t)(r - s)}$\\
$\mathcal G(t) \mathcal W^-(s) \mathcal {\tilde G}(r)$ &$\frac{sr^2t(q^2 - 1)^3(q^4st - 1)}{q^4(-t + r)(s - t)(r - s)}$\\
\hline
$\mathcal G(r) \mathcal W^+(s) \mathcal {\tilde G}(t)$ &$\frac{q^{10}r^2st^2 - 3q^8r^2st^2 - q^8r^2t + q^8rt^2 + 3q^6r^2st^2 - q^6rs^2t^2 + q^6r^2s - q^6st^2}{q^4(-t + r)(r - s)(s - t)}
$\\&$+ \frac{q^4r^2s^2t - q^4r^2st^2 + 2q^4rs^2t^2 - q^4rs^2 + q^4s^2t - 2q^2r^2s^2t - q^2rs^2t^2 + r^2s^2t}{q^4(-t + r)(r - s)(s - t)}$\\
$\mathcal G(r) \mathcal W^+(t) \mathcal {\tilde G}(s)$ &$\frac{t(q^2 - 1)(q^8r^2st - 2q^6r^2st - q^6r^2 + q^6rt + q^4r^2st - q^4rst^2 + q^4rs - q^4st + q^2r^2st + q^2rst^2 - r^2st)}{q^4(t - r)(s - t)(r - s)}$\\
$\mathcal G(s) \mathcal W^+(t) \mathcal {\tilde G}(r)$ &$\frac{(q^2- 1)^2(q^6rst - q^4rst - q^4s + q^4t - q^2rt^2 + rst)tr}{q^4(-t + r)(s - t)(r - s)}$\\
$\mathcal G(s) \mathcal W^+(r) \mathcal {\tilde G}(t)$ &$\frac{(q^2 - 1)(q^8rst^2 - 2q^6rst^2 - q^6st + q^6t^2 - q^4r^2t^2 + q^4rst^2 + q^4rs - q^4rt + q^2r^2st + q^2r^2t^2 - r^2st)r}{q^4(t - r)(s - t)(r - s)}$\\
$\mathcal G(t) \mathcal W^+(r) \mathcal {\tilde G}(s)$ &$\frac{(q^2 - 1)^3(q^4t - r)sr^2t}{q^4(-t + r)(s - t)(r - s)}$\\
$\mathcal G(t) \mathcal W^+(s) \mathcal {\tilde G}(r)$ &$\frac{r^2st(q^2 - 1)^3(q^4t - s)}{q^4(t - r)(s - t)(r - s)}$\\
\hline
$\mathcal W^-(r) \mathcal W^-(s) \mathcal W^-(t)$ &$0$\\
$\mathcal W^-(s) \mathcal W^-(t) \mathcal W^+(r)$ &$\frac{(r - 1)(r + 1)(q^2 + 1)^3(q^2 - 1)^4t^2sr}{(r - s)q^6(-t + r)}$\\
$\mathcal W^-(r) \mathcal W^-(t) \mathcal W^+(s)$ &$\frac{rt(q^2 - 1)^3(q^2 + 1)^3(q^2rs^2t - q^2rt - rs^2t + rs - s^2 + st)}{q^6(t - s)(r - s)}$\\
$\mathcal W^-(r) \mathcal W^-(s) \mathcal W^+(t)$ &$\frac{tsr^2(t - 1)(t + 1)(q^2 + 1)^3(q^2 - 1)^4}{q^6(s - t)(-t + r)}$\\
$\mathcal W^-(r) \mathcal W^+(s) \mathcal W^+(t)$ &$\frac{(q^2 - 1)^3(q^2 + 1)^3(q^2r^2s + q^2r^2t - q^2rst - q^2r - r^2s + s)rt}{(r - s)q^6(-t + r)}$\\
$\mathcal W^-(s) \mathcal W^+(r) \mathcal W^+(t)$ &$\frac{tsr(q^2 + 1)^3(q^2 - 1)^4(rs - rt + st - 1)}{q^6(t - s)(r - s)}$\\
$\mathcal W^-(t) \mathcal W^+(r) \mathcal W^+(s)$ &$\frac{rt(q^2 - 1)^3(q^2 + 1)^3(q^2rst - q^2rt^2 - q^2st^2 + q^2t + st^2 - s)}{q^6(s - t)(t- r)}$\\
$\mathcal W^+(r) \mathcal W^+(s) \mathcal W^+(t)$ & $\frac{(q^2 + 1)^3(q^2 - 1)^3rt}{q^4}$
\end{tabular}}
\bigskip

\noindent The overlap ambiguity \eqref{eq:Three} is resolvable.

\newpage

\noindent Next we evaluate the overlap ambiguity  
\begin{align}
\label{eq:Four}
\mathcal {\tilde{G}}(t) \mathcal W^+(s) \mathcal W^-(r).
\end{align}
 We can proceed in two ways. If we evaluate  $\mathcal {\tilde{G}}(t) \mathcal W^+(s)$ first, then we find that \eqref{eq:Four} is equal to a weighted sum with the following terms and coefficients:

 \bigskip

\centerline{
\begin{tabular}[t]{c|c}
  {\rm term} & {\rm coefficient} 
   \\
\hline
$\mathcal G(r) \mathcal {\tilde G}(s)  \mathcal {\tilde G}(t)$ & $A_{t, s}b_{s, r}e_{t, r} + a_{t, s}b_{t, r}e_{s, r}$\\
$\mathcal G(s) \mathcal {\tilde G}(r)  \mathcal {\tilde G}(t)$ &$A_{t, s}B_{s, r}e_{t, s} - a_{t, s}B_{t, r}e_{s, t} - a_{t, s}b_{t, r}e_{s, r}$\\
$\mathcal G(t) \mathcal {\tilde G}(r)  \mathcal {\tilde G}(s)$ & $-A_{t, s}B_{s, r}e_{t, s} - A_{t, s}b_{s, r}e_{t, r} + a_{t, s}B_{t, r}e_{s,t}$\\
\hline
$ \mathcal W^-(r) \mathcal W^+(s)\mathcal {\tilde G}(t)$ & $a_{t, s}b_{t, r}$\\
$ \mathcal W^-(r) \mathcal W^+(t)\mathcal {\tilde G}(s)$ & $A_{t, s}b_{s, r}$\\
$ \mathcal W^-(s) \mathcal W^+(t)\mathcal {\tilde G}(r) $ & $A'_{t, s}b'_{t, r} + A_{t, s}B_{s, r}$\\
$ \mathcal W^-(s) \mathcal W^+(r)\mathcal {\tilde G}(t) $ & $A'_{t, s}B'_{t, r}$ \\
$ \mathcal W^-(t) \mathcal W^+(r)\mathcal {\tilde G}(s) $ &$a'_{t, s}B'_{s, r}$ \\
$ \mathcal W^-(t) \mathcal W^+(s) \mathcal {\tilde G}(r)$ & $a'_{t, s}b'_{s, r} + a_{t, s}B_{t, r}$\\
\hline
$ \mathcal W^-(s) \mathcal W^-(t) \mathcal {\tilde G}(r)$ &$a'_{t, s}B_{s, r} + A'_{t, s}B_{t, r}$\\
$ \mathcal W^-(r) \mathcal W^-(t)\mathcal {\tilde G}(s) $ &$a'_{t, s}b_{s, r}$ \\
$\mathcal W^-(r) \mathcal W^-(s) \mathcal {\tilde G}(t)$ &$A'_{t, s}b_{t, r}$\\
\hline
$ \mathcal W^+(s) \mathcal W^+(t)\mathcal {\tilde G}(r) $ &$A_{t, s}b'_{s, r} + a_{t, s}b'_{t, r}$ \\
$ \mathcal W^+(r) \mathcal W^+(t)\mathcal {\tilde G}(s)$ & $A_{t, s}B'_{s, r}$\\
$ \mathcal W^+(r) \mathcal W^+(s)\mathcal {\tilde G}(t) $&$a_{t, s}B'_{t, r}$ 
\end{tabular}}
\bigskip

\noindent If we evaluate  $ \mathcal W^+(s) \mathcal W^-(r)$ first, then we find that \eqref{eq:Four} is equal to a weighted sum with the following terms and coefficients:

\centerline{
\begin{tabular}[t]{c|c}
  {\rm term} & {\rm coefficient} 
   \\
\hline
$\mathcal G(r) \mathcal {\tilde G}(s)  \mathcal {\tilde G}(t)$ & $e_{s, r} + b'_{t, r}a'_{r, s}e_{t, r} - B'_{t, r}a'_{t, s}e_{r, t} - B'_{t, r}A'_{t, s}e_{r, s}$\\
$\mathcal G(s) \mathcal {\tilde G}(r)  \mathcal {\tilde G}(t)$ &$-e_{s, r} + b'_{t, r}A'_{r, s}e_{t, s} + B'_{t, r}A'_{t, s}e_{r, s}$\\
$\mathcal G(t) \mathcal {\tilde G}(r)  \mathcal {\tilde G}(s)$ & $-b'_{t, r}a'_{r, s}e_{t, r} - b'_{t, r}A'_{r, s}e_{t, s} + B'_{t, r}a'_{t, s}e_{r, t}$\\
\hline
$ \mathcal W^-(r) \mathcal W^+(s)\mathcal {\tilde G}(t)$ & $b_{t, r}a_{t, s}$\\
$ \mathcal W^-(r) \mathcal W^+(t)\mathcal {\tilde G}(s)$ & $-e_{s, r}f_{t, r} + b_{t, r}A_{t, s} + b'_{t, r}a'_{r, s}$\\
$ \mathcal W^-(s) \mathcal W^+(t)\mathcal {\tilde G}(r) $ & $e_{s, r}f_{t, s} + b'_{t, r}A'_{r, s}$\\
$ \mathcal W^-(s) \mathcal W^+(r)\mathcal {\tilde G}(t) $ & $B'_{t, r}A'_{t, s}$ \\
$ \mathcal W^-(t) \mathcal W^+(r)\mathcal {\tilde G}(s) $ &$e_{s, r}f_{t, r} + B_{t, r}A_{r, s} + B'_{t, r}a'_{t, s}$ \\
$ \mathcal W^-(t) \mathcal W^+(s) \mathcal {\tilde G}(r)$ & $-e_{s, r}f_{t, s} + B_{t, r}a_{r, s}$\\
\hline
$ \mathcal W^-(s) \mathcal W^-(t) \mathcal {\tilde G}(r)$ &$e_{s, r}F_{t, s} + B_{t, r}A'_{r, s}$\\
$ \mathcal W^-(r) \mathcal W^-(t)\mathcal {\tilde G}(s) $ &$-e_{s, r}F_{t, r} + B_{t, r}a'_{r, s} + b_{t, r}a'_{t, s}$ \\
$\mathcal W^-(r) \mathcal W^-(s) \mathcal {\tilde G}(t)$ &$b_{t, r}A'_{t, s}$\\
\hline
$ \mathcal W^+(s) \mathcal W^+(t)\mathcal {\tilde G}(r) $ &$-e_{s, r}F_{t, s} + b'_{t, r}a_{r, s}$ \\
$ \mathcal W^+(r) \mathcal W^+(t)\mathcal {\tilde G}(s)$ & $e_{s, r}F_{t, r} + b'_{t, r}A_{r, s} + B'_{t, r}A_{t, s}$\\
$ \mathcal W^+(r) \mathcal W^+(s)\mathcal {\tilde G}(t) $&$a_{t, s}B'_{t, r}$ 
\end{tabular}}
\bigskip

\newpage

\noindent Referring to the previous two tables, for each row the given coefficients are equal and their common value is displayed below:
 \bigskip
 
\centerline{
\begin{tabular}[t]{c|c}
  {\rm term} & {\rm coefficient} 
   \\
\hline
$\mathcal G(r) \mathcal {\tilde G}(s)  \mathcal {\tilde G}(t)$ & $\frac{-q^4}{(q^2 + 1)^3(q - 1)(q + 1)(-s + r)}$\\
$\mathcal G(s) \mathcal {\tilde G}(r)  \mathcal {\tilde G}(t)$ &$ \frac{q^4}{(q^2 + 1)^3(q - 1)(q + 1)(-s + r)}$\\
$\mathcal G(t) \mathcal {\tilde G}(r)  \mathcal {\tilde G}(s)$ & $0$\\
\hline
$ \mathcal W^-(r) \mathcal W^+(s)\mathcal {\tilde G}(t)$ & $ \frac{-(q^2t - s)(q^2r - t)}{(s - t)(-t + r)q^2}$\\
$ \mathcal W^-(r) \mathcal W^+(t)\mathcal {\tilde G}(s)$ & $ \frac{t(q - 1)(q + 1)(q^2r - s)}{(s - t)(-s + r)q^2}$\\
$ \mathcal W^-(s) \mathcal W^+(t)\mathcal {\tilde G}(r) $ & $ \frac{ts(q - 1)^2(q + 1)^2(rt^2 - st^2 - r + t)}{(s - t)(-t + r)q^2(-s + r)}$\\
$ \mathcal W^-(s) \mathcal W^+(r)\mathcal {\tilde G}(t) $ & $ \frac{-t^2sr(q - 1)^2(q + 1)^2}{(s - t)(-t + r)q^2}$ \\
$ \mathcal W^-(t) \mathcal W^+(r)\mathcal {\tilde G}(s) $ &$ \frac{t^2sr(q - 1)^2(q + 1)^2}{(s - t)(-s + r)q^2}$ \\
$ \mathcal W^-(t) \mathcal W^+(s) \mathcal {\tilde G}(r)$ & $ \frac{-t(q - 1)(q + 1)(q^2rs^2t - q^2s^2t^2 - q^2rt + q^2st - rs^2t + s^2t^2 + rs - s^2)}{(s - t)(-t + r)q^2(-s + r)}$\\
\hline
$ \mathcal W^-(s) \mathcal W^-(t) \mathcal {\tilde G}(r)$ &$ \frac{(q - 1)^2(q + 1)^2t^2s}{(-s + r)(-t + r)q^2}$\\
$ \mathcal W^-(r) \mathcal W^-(t)\mathcal {\tilde G}(s) $ &$ \frac{-t^2(q - 1)(q + 1)(q^2r - s)}{(s - t)(-s + r)q^2}$ \\
$\mathcal W^-(r) \mathcal W^-(s) \mathcal {\tilde G}(t)$ &$ \frac{ts(q - 1)(q + 1)(q^2r - t)}{(s - t)(-t + r)q^2}$\\
\hline
$ \mathcal W^+(s) \mathcal W^+(t)\mathcal {\tilde G}(r) $ &$ \frac{(q - 1)(q + 1)(q^2rs + q^2rt - q^2st - rs)t}{(-s + r)(-t + r)q^2}$ \\
$ \mathcal W^+(r) \mathcal W^+(t)\mathcal {\tilde G}(s)$ & $ \frac{-tsr(q - 1)^2(q + 1)^2}{(s - t)(-s + r)q^2}$\\
$ \mathcal W^+(r) \mathcal W^+(s)\mathcal {\tilde G}(t) $&$ \frac{rt(q - 1)(q + 1)(q^2t - s)}{(-t + r)(s - t)q^2}$ 
\end{tabular}}
\bigskip

\noindent The overlap ambiguity \eqref{eq:Four} is resolvable.
\medskip

\noindent We have shown that the overlap ambiguities of type (i) are resolvable.

\newpage

\noindent Next we evaluate the six overlap ambiguities of type (ii). Here is the first one. Let evaluate the overlap ambiguity  
\begin{align}
\label{eq:Five}
\mathcal W^-(t) \mathcal W^-(s) \mathcal G(r).
\end{align}
We can proceed in two ways. If we evaluate  $\mathcal W^-(t) \mathcal W^-(s)$ first, then we find that \eqref{eq:Five} is equal to a weighted sum with the following terms and coefficients:
\bigskip
 
\centerline{
\begin{tabular}[t]{c|c}
  {\rm term} & {\rm coefficient} 
   \\
\hline
$\mathcal G(s) \mathcal G(t)  \mathcal {\tilde G}(r)$ & $b_{r, t}b'_{r, s}e_{r, t} - B_{r, t}b'_{t, s}e_{t, r} - B_{r, t}B'_{t, s}e_{s, r}$\\
$\mathcal G(r) \mathcal G(t)  \mathcal {\tilde G}(s)$ &$b_{r, t}B'_{r, s}e_{s, t} + B_{r, t}B'_{t, s}e_{s, r}$\\
$\mathcal G(r) \mathcal G(s)  \mathcal {\tilde G}(t)$ & $-b_{r, t}b'_{r, s}e_{r, t} - b_{r, t}B'_{r, s}e_{s, t} + B_{r, t}b'_{t, s}e_{t, r}$\\
\hline
$\mathcal G(r) \mathcal W^-(s) \mathcal W^+(t) $ & $B'_{r, t}b_{r, s}$\\
$\mathcal G(r) \mathcal W^-(t) \mathcal W^+(s) $ & $b_{r, t}B'_{r, s}$\\
$\mathcal G(s) \mathcal W^-(t) \mathcal W^+(r) $ & $b'_{r, t}B_{t, s} + b_{r, t}b'_{r, s}$\\
$\mathcal G(s) \mathcal W^-(r) \mathcal W^+(t) $ & $B'_{r, t}B_{r, s} + B_{r, t}b'_{t, s}$ \\
$\mathcal G(t) \mathcal W^-(r) \mathcal W^+(s) $ &$B_{r, t}B'_{t, s}$ \\
$\mathcal G(t) \mathcal W^-(s) \mathcal W^+(r) $ & $b'_{r, t}b_{t, s}$\\
\hline
$\mathcal G(r) \mathcal W^-(s) \mathcal W^-(t) $ &$b_{r, t}b_{r, s}$\\
$\mathcal G(s) \mathcal W^-(r) \mathcal W^-(t) $ &$B_{r, t}B_{t, s} + b_{r, t}B_{r, s}$ \\
$\mathcal G(t) \mathcal W^-(r) \mathcal W^-(s) $ &$B_{r, t}b_{t, s}$\\
\hline
$\mathcal G(r) \mathcal W^+(s) \mathcal W^+(t) $ &$B'_{r, t}B'_{r, s}$ \\
$\mathcal G(s) \mathcal W^+(r) \mathcal W^+(t) $ & $b'_{r, t}b'_{t, s} + B'_{r, t}b'_{r, s}$\\
$\mathcal G(t) \mathcal W^+(r) \mathcal W^+(s) $&$b'_{r, t}B'_{t, s}$ 
\end{tabular}}
\bigskip

\noindent If we evaluate  $\mathcal W^-(s) \mathcal G(r)$ first, then we find that \eqref{eq:Five} is equal to a weighted sum with the following terms and coefficients:
\bigskip
 
\centerline{
\begin{tabular}[t]{c|c}
  {\rm term} & {\rm coefficient} 
   \\
\hline
$\mathcal G(s) \mathcal G(t)  \mathcal {\tilde G}(r)$ & $b_{r, s}b'_{r, t}e_{r, s} - B_{r, s}b'_{s, t}e_{s, r} - B_{r, s}B'_{s, t}e_{t, r}$\\
$\mathcal G(r) \mathcal G(t)  \mathcal {\tilde G}(s)$ &$-b_{r, s}b'_{r, t}e_{r, s} - b_{r, s}B'_{r, t}e_{t, s} + B_{r, s}b'_{s, t}e_{s, r}$\\
$\mathcal G(r) \mathcal G(s)  \mathcal {\tilde G}(t)$ & $b_{r, s}B'_{r, t}e_{t, s} + B_{r, s}B'_{s, t}e_{t, r}$\\
\hline
$\mathcal G(r) \mathcal W^-(s) \mathcal W^+(t) $ & $B'_{r, t}b_{r, s}$\\
$\mathcal G(r) \mathcal W^-(t) \mathcal W^+(s) $ & $b_{r, t}B'_{r, s}$\\
$\mathcal G(s) \mathcal W^-(t) \mathcal W^+(r) $ & $b'_{r, s}b_{s, t}$\\
$\mathcal G(s) \mathcal W^-(r) \mathcal W^+(t) $ & $B_{r, s}B'_{s, t}$ \\
$\mathcal G(t) \mathcal W^-(r) \mathcal W^+(s) $ &$B'_{r, s}B_{r, t} + B_{r, s}b'_{s, t}$ \\
$\mathcal G(t) \mathcal W^-(s) \mathcal W^+(r) $ & $b'_{r, s}B_{s, t} + b_{r, s}b'_{r, t}$\\
\hline
$\mathcal G(r) \mathcal W^-(s) \mathcal W^-(t) $ &$b_{r, t}b_{r, s}$\\
$\mathcal G(s) \mathcal W^-(r) \mathcal W^-(t) $ &$B_{r, s}b_{s, t}$ \\
$\mathcal G(t) \mathcal W^-(r) \mathcal W^-(s) $ &$B_{r, s}B_{s, t} + b_{r, s}B_{r, t}$\\
\hline
$\mathcal G(r) \mathcal W^+(s) \mathcal W^+(t) $ &$B'_{r, t}B'_{r, s}$ \\
$\mathcal G(s) \mathcal W^+(r) \mathcal W^+(t) $ & $b'_{r, s}B'_{s, t}$\\
$\mathcal G(t) \mathcal W^+(r) \mathcal W^+(s) $&$b'_{r, s}b'_{s, t} + B'_{r, s}b'_{r, t}$ 
\end{tabular}}

\newpage
\noindent Referring to the previous two tables, for each row the given coefficients are equal and their common value is displayed below:

\bigskip
 
\centerline{
\begin{tabular}[t]{c|c}
  {\rm term} & {\rm coefficient} 
   \\
\hline
$\mathcal G(s) \mathcal G(t)  \mathcal {\tilde G}(r)$ & $\frac{-r^2}{(q^2 + 1)^3(r - t)(r - s)}$\\
$\mathcal G(r) \mathcal G(t)  \mathcal {\tilde G}(s)$ &$\frac{rs}{(-t + s)(q^2 + 1)^3(r - s)}$\\
$\mathcal G(r) \mathcal G(s)  \mathcal {\tilde G}(t)$ & $\frac{-rt}{(-t + s)(q^2 + 1)^3(r - t)}$\\
\hline
$\mathcal G(r) \mathcal W^-(s) \mathcal W^+(t) $ & $\frac{-rt(q - 1)(q + 1)(q^2s - r)}{q^4(r - t)(r - s)}$\\
$\mathcal G(r) \mathcal W^-(t) \mathcal W^+(s) $ & $\frac{-rs(q - 1)(q + 1)(q^2t - r)}{q^4(r - t)(r - s)}$\\
$\mathcal G(s) \mathcal W^-(t) \mathcal W^+(r) $ & $\frac{(q - 1)(q + 1)(q^2t - s)r^2}{(r - s)q^4(-t + s)}$\\
$\mathcal G(s) \mathcal W^-(r) \mathcal W^+(t) $ & $\frac{rts(q - 1)^2(q + 1)^2}{(r - s)q^4(-t + s)}$ \\
$\mathcal G(t) \mathcal W^-(r) \mathcal W^+(s) $ &$\frac{-rts(q - 1)^2(q + 1)^2}{q^4(r - t)(-t + s)}$ \\
$\mathcal G(t) \mathcal W^-(s) \mathcal W^+(r) $ & $\frac{-r^2(q - 1)(q + 1)(q^2s - t)}{q^4(r - t)(-t + s)}$\\
\hline
$\mathcal G(r) \mathcal W^-(s) \mathcal W^-(t) $ &$\frac{(q^2t - r)(q^2s - r)}{q^4(r - t)(r - s)}$\\
$\mathcal G(s) \mathcal W^-(r) \mathcal W^-(t) $ &$\frac{-(q - 1)(q + 1)(q^2t - s)r}{(r - s)q^4(-t + s)}$ \\
$\mathcal G(t) \mathcal W^-(r) \mathcal W^-(s) $ &$\frac{r(q - 1)(q + 1)(q^2s - t)}{q^4(r - t)(-t + s)}$\\
\hline
$\mathcal G(r) \mathcal W^+(s) \mathcal W^+(t) $ &$\frac{r^2ts(q - 1)^2(q + 1)^2}{q^4(r - t)(r - s)}$ \\
$\mathcal G(s) \mathcal W^+(r) \mathcal W^+(t) $ & $\frac{-r^2ts(q - 1)^2(q + 1)^2}{(r - s)q^4(-t + s)}$\\
$\mathcal G(t) \mathcal W^+(r) \mathcal W^+(s) $&$\frac{r^2ts(q - 1)^2(q + 1)^2}{q^4(r - t)(-t + s)}$ 
\end{tabular}}
\bigskip

\noindent The overlap ambiguity \eqref{eq:Five} is resolvable.

\newpage

\noindent Next we evaluate the overlap ambiguity  
\begin{align}
\label{eq:Six}
\mathcal W^+(t) \mathcal W^+(s) \mathcal G(r).
\end{align}
We can proceed in two ways. If we evaluate  $\mathcal W^+(t) \mathcal W^+(s)$ first, then we find that \eqref{eq:Six} is equal to a weighted sum with the following terms and coefficients:
 
\centerline{
\begin{tabular}[t]{c|c}
  {\rm term} & {\rm coefficient} 
   \\
\hline
$\mathcal G(s) \mathcal G(t)  \mathcal {\tilde G}(r)$ & $-a'_{r, t}A_{t, s}e_{t, r} - a'_{r, t}a_{t, s}e_{s, r} + A'_{r, t}A_{r, s}e_{r, t}$\\
$\mathcal G(r) \mathcal G(t)  \mathcal {\tilde G}(s)$ &$a'_{r, t}a_{t, s}e_{s, r} + A'_{r, t}a_{r, s}e_{s, t}$\\
$\mathcal G(r) \mathcal G(s)  \mathcal {\tilde G}(t)$ & $a'_{r, t}A_{t, s}e_{t, r} - A'_{r, t}A_{r, s}e_{r, t} - A'_{r, t}a_{r, s}e_{s, t}$\\
\hline
$\mathcal G(r) \mathcal W^-(s) \mathcal W^+(t) $ & $a_{r, t}A'_{r, s}$\\
$\mathcal G(r) \mathcal W^-(t) \mathcal W^+(s) $ & $A'_{r, t}a_{r, s}$\\
$\mathcal G(s) \mathcal W^-(t) \mathcal W^+(r) $ & $A_{r, t}a'_{t, s} + A'_{r, t}A_{r, s}$\\
$\mathcal G(s) \mathcal W^-(r) \mathcal W^+(t) $ & $a_{r, t}a'_{r, s} + a'_{r, t}A_{t, s}$ \\
$\mathcal G(t) \mathcal W^-(r) \mathcal W^+(s) $ &$a'_{r, t}a_{t, s}$ \\
$\mathcal G(t) \mathcal W^-(s) \mathcal W^+(r) $ & $A_{r, t}A'_{t, s}$\\
\hline
$\mathcal G(r) \mathcal W^-(s) \mathcal W^-(t) $ &$A'_{r, t}A'_{r, s}$\\
$\mathcal G(s) \mathcal W^-(r) \mathcal W^-(t) $ &$a'_{r, t}a'_{t, s} + A'_{r, t}a'_{r, s}$ \\
$\mathcal G(t) \mathcal W^-(r) \mathcal W^-(s) $ &$a'_{r, t}A'_{t, s}$\\
\hline
$\mathcal G(r) \mathcal W^+(s) \mathcal W^+(t) $ &$a_{r, t}a_{r, s}$ \\
$\mathcal G(s) \mathcal W^+(r) \mathcal W^+(t) $ & $A_{r, t}A_{t, s} + a_{r, t}A_{r, s}$\\
$\mathcal G(t) \mathcal W^+(r) \mathcal W^+(s) $&$A_{r, t}a_{t, s}$ 
\end{tabular}}
\bigskip

\noindent If we evaluate  $\mathcal W^+(s) \mathcal G(r)$ first, then we find that \eqref{eq:Six} is equal to a weighted sum with the following terms and coefficients:
  
\centerline{
\begin{tabular}[t]{c|c}
  {\rm term} & {\rm coefficient} 
   \\
\hline
$\mathcal G(s) \mathcal G(t)  \mathcal {\tilde G}(r)$ & $-a'_{r, s}A_{s, t}e_{s, r} - a'_{r, s}a_{s, t}e_{t, r} + A'_{r, s}A_{r, t}e_{r, s}$\\
$\mathcal G(r) \mathcal G(t)  \mathcal {\tilde G}(s)$ &$a'_{r, s}A_{s, t}e_{s, r} - A'_{r, s}A_{r, t}e_{r, s} - A'_{r, s}a_{r, t}e_{t, s}$\\
$\mathcal G(r) \mathcal G(s)  \mathcal {\tilde G}(t)$ & $a'_{r, s}a_{s, t}e_{t, r} + A'_{r, s}a_{r, t}e_{t, s}$\\
\hline
$\mathcal G(r) \mathcal W^-(s) \mathcal W^+(t) $ & $a_{r, t}A'_{r, s}$\\
$\mathcal G(r) \mathcal W^-(t) \mathcal W^+(s) $ & $A'_{r, t}a_{r, s}$\\
$\mathcal G(s) \mathcal W^-(t) \mathcal W^+(r) $ & $A_{r, s}A'_{s, t}$\\
$\mathcal G(s) \mathcal W^-(r) \mathcal W^+(t) $ & $a'_{r, s}a_{s, t}$ \\
$\mathcal G(t) \mathcal W^-(r) \mathcal W^+(s) $ &$a_{r, s}a'_{r, t} + a'_{r, s}A_{s, t}$ \\
$\mathcal G(t) \mathcal W^-(s) \mathcal W^+(r) $ & $A_{r, s}a'_{s, t} + A'_{r, s}A_{r, t}$\\
\hline
$\mathcal G(r) \mathcal W^-(s) \mathcal W^-(t) $ &$A'_{r, t}A'_{r, s}$\\
$\mathcal G(s) \mathcal W^-(r) \mathcal W^-(t) $ &$a'_{r, s}A'_{s, t}$ \\
$\mathcal G(t) \mathcal W^-(r) \mathcal W^-(s) $ &$a'_{r, s}a'_{s, t} + A'_{r, s}a'_{r, t}$\\
\hline
$\mathcal G(r) \mathcal W^+(s) \mathcal W^+(t) $ &$a_{r, t}a_{r, s}$ \\
$\mathcal G(s) \mathcal W^+(r) \mathcal W^+(t) $ & $A_{r, s}a_{s, t}$\\
$\mathcal G(t) \mathcal W^+(r) \mathcal W^+(s) $&$A_{r, s}A_{s, t} + a_{r, s}A_{r, t}$ 
\end{tabular}}
\bigskip

\newpage
\noindent Referring to the previous two tables, for each row the given coefficients are equal and their common value is displayed below:

\bigskip
  
\centerline{
\begin{tabular}[t]{c|c}
  {\rm term} & {\rm coefficient} 
   \\
\hline
$\mathcal G(s) \mathcal G(t)  \mathcal {\tilde G}(r)$ & $\frac{q^4r^2}{(r - s)(q^2 + 1)^3(r - t)}$\\
$\mathcal G(r) \mathcal G(t)  \mathcal {\tilde G}(s)$ &$\frac{-q^4rs}{(r - s)(q^2 + 1)^3(-t + s)}$\\
$\mathcal G(r) \mathcal G(s)  \mathcal {\tilde G}(t)$ & $\frac{q^4tr}{(q^2 + 1)^3(-t + s)(r - t)}$\\
\hline
$\mathcal G(r) \mathcal W^-(s) \mathcal W^+(t) $ & $\frac{-sr(q - 1)(q + 1)(q^2r - t)}{(r - t)(r - s)}$\\
$\mathcal G(r) \mathcal W^-(t) \mathcal W^+(s) $ & $\frac{-tr(q - 1)(q + 1)(q^2r - s)}{(r - t)(r - s)}$\\
$\mathcal G(s) \mathcal W^-(t) \mathcal W^+(r) $ & $\frac{(q - 1)^2(q + 1)^2srt}{(r - s)(-t + s)}$\\
$\mathcal G(s) \mathcal W^-(r) \mathcal W^+(t) $ & $\frac{(q - 1)(q + 1)(q^2s - t)r^2}{(r - s)(-t + s)}$ \\
$\mathcal G(t) \mathcal W^-(r) \mathcal W^+(s) $ &$\frac{-r^2(q - 1)(q + 1)(q^2t - s)}{(r - t)(-t + s)}$ \\
$\mathcal G(t) \mathcal W^-(s) \mathcal W^+(r) $ & $\frac{-(q - 1)^2(q + 1)^2srt}{(r - t)(-t + s)}$\\
\hline
$\mathcal G(r) \mathcal W^-(s) \mathcal W^-(t) $ &$\frac{r^2ts(q - 1)^2(q + 1)^2}{(r - t)(r - s)}$\\
$\mathcal G(s) \mathcal W^-(r) \mathcal W^-(t) $ &$\frac{-r^2ts(q - 1)^2(q + 1)^2}{(r - s)(-t + s)}$ \\
$\mathcal G(t) \mathcal W^-(r) \mathcal W^-(s) $ &$\frac{r^2ts(q - 1)^2(q + 1)^2}{(r - t)(-t + s)}$\\
\hline
$\mathcal G(r) \mathcal W^+(s) \mathcal W^+(t) $ &$\frac{(q^2r - t)(q^2r - s)}{(r - t)(r - s)}$ \\
$\mathcal G(s) \mathcal W^+(r) \mathcal W^+(t) $ & $\frac{-(q - 1)(q + 1)(q^2s - t)r}{(r - s)(-t + s)}$\\
$\mathcal G(t) \mathcal W^+(r) \mathcal W^+(s) $&$\frac{r(q - 1)(q + 1)(q^2t - s)}{(r - t)(-t + s)}$ 
\end{tabular}}
\bigskip

\noindent The overlap ambiguity \eqref{eq:Six} is resolvable.

\newpage

\noindent Next we evaluate the overlap ambiguity  
\begin{align}
\label{eq:Seven}
\mathcal {\tilde{G}}(t) \mathcal {\tilde G}(s) \mathcal G(r).
\end{align}
 We can proceed in two ways. If we evaluate  $\mathcal {\tilde{G}}(t) \mathcal {\tilde G}(s)$ first, then we find that \eqref{eq:Seven} is equal to a weighted sum with the following terms and coefficients:

 \bigskip

\centerline{
\begin{tabular}[t]{c|c}
  {\rm term} & {\rm coefficient} 
   \\
\hline
$\mathcal G(r) \mathcal {\tilde G}(s)  \mathcal {\tilde G}(t)$ & $1 + f_{s, r}b'_{t, s}A'_{s, r}e_{t, r} + f_{s, r}B'_{t, s}A'_{t, r}e_{s, r} - F_{s, r}B'_{t, s}b_{t, r}e_{s, r} - F_{s, r}b'_{t, s}b_{s, r}e_{t, r}$\\&$ - f_{s, r}b'_{t, r}a'_{r, s}e_{t, r} + f_{s, r}B'_{t, r}a'_{t, s}e_{r, t} + f_{s, r}B'_{t, r}A'_{t, s}e_{r, s} + F_{s, r}A_{t, s}A'_{s, r}e_{t, r}$\\&$ + F_{s, r}a_{t, s}A'_{t, r}e_{s, r}$\\
$\mathcal G(s) \mathcal {\tilde G}(r)  \mathcal {\tilde G}(t)$ &$f_{s, r}b'_{t, s}a'_{s, r}e_{t, s} - f_{s, r}B'_{t, s}a'_{t, r}e_{s, t} - f_{s, r}B'_{t, s}A'_{t, r}e_{s, r} + F_{s, r}B'_{t, s}b_{t, r}e_{s, r}$\\&$+ F_{s, r}B'_{t, s}B_{t, r}e_{s, t} - f_{s, r}b'_{t, r}A'_{r, s}e_{t, s} - f_{s, r}B'_{t, r}A'_{t, s}e_{r, s} + F_{s, r}A_{t, s}a'_{s, r}e_{t, s}$\\&$- F_{s, r}a_{t, s}a'_{t, r}e_{s, t}- F_{s, r}a_{t, s}A'_{t, r}e_{s, r} - F_{s, r}b'_{t, s}B_{s, r}e_{t, s}$\\
$\mathcal G(t) \mathcal {\tilde G}(r)  \mathcal {\tilde G}(s)$ & $-f_{s, r}b'_{t, s}a'_{s, r}e_{t, s} - f_{s, r}b'_{t, s}A'_{s, r}e_{t, r} + f_{s, r}B'_{t, s}a'_{t, r}e_{s, t} - F_{s, r}B'_{t, s}B_{t, r}e_{s, t}$\\
\end{tabular}}
\bigskip

\newpage
\noindent If we evaluate  $ \mathcal {\tilde G}(s) \mathcal G(r)$ first, then we find that \eqref{eq:Seven} is equal to a weighted sum with the following terms and coefficients:
 \bigskip

\centerline{
\begin{tabular}[t]{c|c}
  {\rm term} & {\rm coefficient} 
   \\
\hline
$\mathcal G(r) \mathcal {\tilde G}(s)  \mathcal {\tilde G}(t)$ & $1 + f_{t, r}b'_{s, t}A'_{t, r}e_{s, r} + f_{t, r}B'_{s, t}A'_{s, r}e_{t, r} - F_{t, r}B'_{s, t}b_{s, r}e_{t, r} - F_{t, r}b'_{s, t}b_{t, r}e_{s, r}$\\&$ - f_{t, r}b'_{s, r}a'_{r, t}e_{s, r} + f_{t, r}B'_{s, r}a'_{s, t}e_{r, s} + f_{t, r}B'_{s, r}A'_{s, t}e_{r, t} + F_{t, r}A_{s, t}A'_{t, r}e_{s, r}$\\&$ + F_{t, r}a_{s, t}A'_{s, r}e_{t, r}$\\
$\mathcal G(s) \mathcal {\tilde G}(r)  \mathcal {\tilde G}(t)$ &$-f_{t, r}b'_{s, t}a'_{t, r}e_{s, t} - f_{t, r}b'_{s, t}A'_{t, r}e_{s, r} + f_{t, r}B'_{s, t}a'_{s, r}e_{t, s} - F_{t, r}B'_{s, t}B_{s, r}e_{t, s}$\\&$+ F_{t, r}b'_{s, t}B_{t, r}e_{s, t} + f_{t, r}b'_{s, r}a'_{r, t}e_{s, r}+ f_{t, r}b'_{s, r}A'_{r, t}e_{s, t} - f_{t, r}B'_{s, r}a'_{s, t}e_{r, s} $\\&$+ F_{t, r}a_{s, t}a'_{s, r}e_{t, s} - F_{t, r}A_{s, t}a'_{t, r}e_{s, t} - F_{t, r}A_{s, t}A'_{t, r}e_{s, r} + F_{t, r}b'_{s, t}b_{t, r}e_{s, r}$\\
$\mathcal G(t) \mathcal {\tilde G}(r)  \mathcal {\tilde G}(s)$ & $f_{t, r}b'_{s, t}a'_{t, r}e_{s, t} - f_{t, r}B'_{s, t}a'_{s, r}e_{t, s} - f_{t, r}B'_{s, t}A'_{s, r}e_{t, r} + F_{t, r}B'_{s, t}b_{s, r}e_{t, r}$\\
\end{tabular}}
\bigskip

\newpage

\noindent Referring to the previous two tables, for each row the given coefficients are equal and their common value is displayed below:
 \bigskip
  
\centerline{
\begin{tabular}[t]{c|c}
  {\rm term} & {\rm coefficient} 
   \\
\hline
$\mathcal G(r) \mathcal {\tilde G}(s)  \mathcal {\tilde G}(t)$ & $ \frac{q^{10}r^2st - 3q^8r^2st + 4q^6r^2st - 4q^4r^2st - q^4r^2 + q^4rs + q^4rt - q^4st + 3q^2r^2st - r^2st}{q^4(r - s)(t - r)}$\\
$\mathcal G(s) \mathcal {\tilde G}(r)  \mathcal {\tilde G}(t)$ &$ \frac{trs^2(q^4 + 1)(q - 1)^3(q + 1)^3}{q^4(-t + s)(r - s)}$\\
$\mathcal G(t) \mathcal {\tilde G}(r)  \mathcal {\tilde G}(s)$ & $ \frac{t^2rs(q^4 + 1)(q - 1)^3(q + 1)^3}{q^4(t - s)(-t + r)}$\\
\hline
$ \mathcal W^-(r) \mathcal W^+(s)\mathcal {\tilde G}(t)$ & $ \frac{rsQ(q^8r^2t^2 - q^8rst^2 - q^6r^2s^2t^2 - q^6r^2t^2 + q^6rs^2t + 2q^6rst^2 + 2q^4r^2s^2t^2 - q^6rt - q^4r^2s^2)}{q^{4}(r - s)(t - s)(-t + r)}$\\
&$
+\frac{rsQ(- 3q^4rst^2 - q^2r^2s^2t^2 + q^4rs + q^4t^2 + q^2r^2st + 2q^2rst^2 - q^2s^2t^2 - q^2st - rst^2 + s^2t^2)}{q^{4}(r - s)(t - s)(-t + r)}
$\\
$ \mathcal W^-(r) \mathcal W^+(t)\mathcal {\tilde G}(s)$ & $ \frac{Q(q^8r^2s^2 - q^8rs^2t - q^6r^2s^2t^2 - q^6r^2s^2 + 2q^6rs^2t + q^6rst^2 + 2q^4r^2s^2t^2 - q^6rs)}{q^{4}(r - s)(-t + s)(-t + r)}$\\
& $+ \frac{Q(- q^4r^2t^2 - 3q^4rs^2t - q^2r^2s^2t^2 + q^4rt + q^4s^2 + q^2r^2st + 2q^2rs^2t - q^2s^2t^2 - q^2st - rs^2t + s^2t^2)rt}{q^{4}(r - s)(-t + s)(-t + r)}$\\
\end{tabular}}
\bigskip

 \noindent In the above table we abbreviate $Q=(q^2-q^{-2})^3$.
\bigskip

\noindent The overlap ambiguity \eqref{eq:Seven} is resolvable.

\newpage

\noindent Next we evaluate the overlap ambiguity  
\begin{align}
\label{eq:Eight}
\mathcal W^+(t) \mathcal W^+(s) \mathcal W^-(r).
\end{align}
 We can proceed in two ways. If we evaluate  $\mathcal W^+(t) \mathcal W^+(s)$ first, then we find that \eqref{eq:Eight} is equal to a weighted sum with the following terms and coefficients:

 \bigskip

\centerline{
\begin{tabular}[t]{c|c}
  {\rm term} & {\rm coefficient} 
   \\
\hline
$\mathcal G(r) \mathcal W^-(s)  \mathcal {\tilde G}(t)$ & $e_{t, r}A'_{t, s}$\\
$\mathcal G(r) \mathcal W^-(t)  \mathcal {\tilde G}(s)$ & $e_{s, r}A'_{r, t} + e_{t, r}a'_{t, s}$\\
$\mathcal G(s) \mathcal W^-(t)  \mathcal {\tilde G}(r)$ & $-e_{s, r}A'_{s, t}$\\
$\mathcal G(s) \mathcal W^-(r)  \mathcal {\tilde G}(t)$ & $0$\\
$\mathcal G(t) \mathcal W^-(r)  \mathcal {\tilde G}(s)$ & $e_{s, r}a'_{r, t} - e_{t, r}a'_{r, s}$\\
$\mathcal G(t) \mathcal W^-(s)  \mathcal {\tilde G}(r)$ & $-e_{s, r}a'_{s, t} - e_{t, r}A'_{r, s}$\\
\hline
$\mathcal G(r) \mathcal W^+(s)  \mathcal {\tilde G}(t)$ & $e_{t, r}a_{t, s}$\\
$\mathcal G(r) \mathcal W^+(t)  \mathcal {\tilde G}(s)$ & $e_{s, r}a_{r, t} + e_{t, r}A_{t, s}$\\
$\mathcal G(s) \mathcal W^+(t)  \mathcal {\tilde G}(r)$ & $-e_{s, r}a_{s, t}$\\
$\mathcal G(s) \mathcal W^+(r)  \mathcal {\tilde G}(t)$ & $0$\\
$\mathcal G(t) \mathcal W^+(r)  \mathcal {\tilde G}(s)$ & $e_{s, r}A_{r, t} - e_{t, r}A_{r, s}$\\
$\mathcal G(t) \mathcal W^+(s)  \mathcal {\tilde G}(r)$ & $-e_{s, r}A_{s, t} - e_{t, r}a_{r, s}$\\
\hline
$\mathcal W^-(r) \mathcal W^+(s)  \mathcal W^+(t)$ & $1$
\end{tabular}}
\bigskip

\noindent If we evaluate  $ \mathcal W^+(s) \mathcal W^-(r)$ first, then we find that \eqref{eq:Eight} is equal to a weighted sum with the following terms and coefficients:
 \bigskip

\centerline{
\begin{tabular}[t]{c|c}
  {\rm term} & {\rm coefficient} 
   \\
\hline
$\mathcal G(r) \mathcal W^-(s)  \mathcal {\tilde G}(t)$ & $e_{t, r}A'_{r, s} + e_{s, r}a'_{s, t}$\\
$\mathcal G(r) \mathcal W^-(t)  \mathcal {\tilde G}(s)$ & $e_{s, r}A'_{s, t}$\\
$\mathcal G(s) \mathcal W^-(t)  \mathcal {\tilde G}(r)$ & $-e_{t, r}a'_{t, s} - e_{s, r}A'_{r, t}$\\
$\mathcal G(s) \mathcal W^-(r)  \mathcal {\tilde G}(t)$ & $e_{t, r}a'_{r, s} - e_{s, r}a'_{r, t}$\\
$\mathcal G(t) \mathcal W^-(r)  \mathcal {\tilde G}(s)$ & $0$\\
$\mathcal G(t) \mathcal W^-(s)  \mathcal {\tilde G}(r)$ & $-e_{t, r}A'_{t, s}$\\
\hline
$\mathcal G(r) \mathcal W^+(s)  \mathcal {\tilde G}(t)$ & $e_{t, r}a_{r, s} + e_{s, r}A_{s, t}$\\
$\mathcal G(r) \mathcal W^+(t)  \mathcal {\tilde G}(s)$ & $e_{s, r}a_{s, t}$\\
$\mathcal G(s) \mathcal W^+(t)  \mathcal {\tilde G}(r)$ & $-e_{t, r}A_{t, s} - e_{s, r}a_{r, t}$\\
$\mathcal G(s) \mathcal W^+(r)  \mathcal {\tilde G}(t)$ & $e_{t, r}A_{r, s} - e_{s, r}A_{r, t}$\\
$\mathcal G(t) \mathcal W^+(r)  \mathcal {\tilde G}(s)$ & $0$\\
$\mathcal G(t) \mathcal W^+(s)  \mathcal {\tilde G}(r)$ & $-e_{t, r}a_{t, s}$\\
\hline
$\mathcal W^-(r) \mathcal W^+(s)  \mathcal W^+(t)$ & $1$
\end{tabular}}
\bigskip

\newpage

\noindent Referring to the previous two tables, for each row the given coefficients are equal and their common value is displayed below:
 \bigskip

\centerline{
\begin{tabular}[t]{c|c}
  {\rm term} & {\rm coefficient} 
   \\
\hline
$\mathcal G(r) \mathcal W^-(s)  \mathcal {\tilde G}(t)$ & $ \frac{-q^4ts}{(-t + s)(-t + r)(q^2 + 1)^3}$\\
$\mathcal G(r) \mathcal W^-(t)  \mathcal {\tilde G}(s)$ & $ \frac{q^4ts}{(-t + s)(-s + r)(q^2 + 1)^3}$\\
$\mathcal G(s) \mathcal W^-(t)  \mathcal {\tilde G}(r)$ & $ \frac{-q^4ts}{(-t + s)(-s + r)(q^2 + 1)^3}$\\
$\mathcal G(s) \mathcal W^-(r)  \mathcal {\tilde G}(t)$ & $0$\\
$\mathcal G(t) \mathcal W^-(r)  \mathcal {\tilde G}(s)$ & $0$\\
$\mathcal G(t) \mathcal W^-(s)  \mathcal {\tilde G}(r)$ & $ \frac{q^4ts}{(-t + s)(-t + r)(q^2 + 1)^3}$\\
\hline
$\mathcal G(r) \mathcal W^+(s)  \mathcal {\tilde G}(t)$ & $ \frac{q^4(q^2t - s)}{(q - 1)(q + 1)(q^2 + 1)^3(-t + r)(-t + s)}$\\
$\mathcal G(r) \mathcal W^+(t)  \mathcal {\tilde G}(s)$ & $ \frac{-(q^2s - t)q^4}{(q - 1)(q + 1)(q^2 + 1)^3(-s + r)(-t + s)}$\\
$\mathcal G(s) \mathcal W^+(t)  \mathcal {\tilde G}(r)$ & $ \frac{(q^2s - t)q^4}{(q - 1)(q + 1)(q^2 + 1)^3(-s + r)(-t + s)}$\\
$\mathcal G(s) \mathcal W^+(r)  \mathcal {\tilde G}(t)$ & $0$\\
$\mathcal G(t) \mathcal W^+(r)  \mathcal {\tilde G}(s)$ & $0$\\
$\mathcal G(t) \mathcal W^+(s)  \mathcal {\tilde G}(r)$ & $ \frac{-q^4(q^2t - s)}{(q - 1)(q + 1)(q^2 + 1)^3(-t + r)(-t + s)}$\\
\hline
$\mathcal W^-(r) \mathcal W^+(s)  \mathcal W^+(t)$ & $1$
\end{tabular}}
\bigskip
 
\noindent The overlap ambiguity \eqref{eq:Eight} is resolvable.

\newpage

\noindent Next we evaluate the overlap ambiguity  
\begin{align}
\label{eq:Nine}
\mathcal {\tilde G}(t) \mathcal {\tilde G}(s) \mathcal W^-(r).
\end{align}
 We can proceed in two ways. If we evaluate  $\mathcal {\tilde G}(t) \mathcal {\tilde G}(s)$ first, then we find that \eqref{eq:Nine} is equal to a weighted sum with the following terms and coefficients:

 \bigskip

\centerline{
\begin{tabular}[t]{c|c}
  {\rm term} & {\rm coefficient} 
   \\
\hline
$\mathcal W^-(r) \mathcal {\tilde G}(s) \mathcal {\tilde G}(t)$ & $b_{s, r}b_{t, r} + B'_{s, r}A'_{t, r}$\\
$\mathcal W^-(s) \mathcal {\tilde G}(r) \mathcal {\tilde G}(t)$ & $B_{s, r}b_{t, s} + b'_{s, r}A'_{t, s}$\\
$\mathcal W^-(t) \mathcal {\tilde G}(r) \mathcal {\tilde G}(s)$ & $B_{s, r}B_{t, s} + b_{s, r}B_{t, r} + b'_{s, r}a'_{t, s} + B'_{s, r}a'_{t, r}$\\
\hline
$\mathcal W^+(r) \mathcal {\tilde G}(s) \mathcal {\tilde G}(t)$ & $b_{s, r}B'_{t, r} + B'_{s, r}a_{t, r}$\\
$\mathcal W^+(s) \mathcal {\tilde G}(r) \mathcal {\tilde G}(t)$ & $B_{s, r}B'_{t, s} + b'_{s, r}a_{t, s}$\\
$\mathcal W^+(t) \mathcal {\tilde G}(r) \mathcal {\tilde G}(s)$ & $B_{s, r}b'_{t, s} + b_{s, r}b'_{t, r} + b'_{s, r}A_{t, s} + B'_{s, r}A_{t, r}$
\end{tabular}}
\bigskip

\noindent If we evaluate  $ \mathcal {\tilde G}(s) \mathcal W^-(r)$ first, then we find that \eqref{eq:Nine} is equal to a weighted sum with the following terms and coefficients:
 \bigskip
  
\centerline{
\begin{tabular}[t]{c|c}
  {\rm term} & {\rm coefficient} 
   \\
\hline
$\mathcal W^-(r) \mathcal {\tilde G}(s) \mathcal {\tilde G}(t)$ & $b_{t, r}b_{s, r} + B'_{t, r}A'_{s, r}$\\
$\mathcal W^-(s) \mathcal {\tilde G}(r) \mathcal {\tilde G}(t)$ & $B_{t, r}B_{s, t} + b_{t, r}B_{s, r} + b'_{t, r}a'_{s, t} + B'_{t, r}a'_{s, r}$\\
$\mathcal W^-(t) \mathcal {\tilde G}(r) \mathcal {\tilde G}(s)$ & $B_{t, r}b_{s, t} + b'_{t, r}A'_{s, t}$\\
\hline
$\mathcal W^+(r) \mathcal {\tilde G}(s) \mathcal {\tilde G}(t)$ & $b_{t, r}B'_{s, r} + B'_{t, r}a_{s, r}$\\
$\mathcal W^+(s) \mathcal {\tilde G}(r) \mathcal {\tilde G}(t)$ & $B_{t, r}b'_{s, t} + b_{t, r}b'_{s, r} + b'_{t, r}A_{s, t} + B'_{t, r}A_{s, r}$\\
$\mathcal W^+(t) \mathcal {\tilde G}(r) \mathcal {\tilde G}(s)$ & $B_{t, r}B'_{s, t} + b'_{t, r}a_{s, t}$
\end{tabular}}
\bigskip

\noindent Referring to the previous two tables, for each row the given coefficients are equal and their common value is displayed below:
 \bigskip

\centerline{
\begin{tabular}[t]{c|c}
  {\rm term} & {\rm coefficient} 
   \\
\hline
$\mathcal W^-(r) \mathcal {\tilde G}(s) \mathcal {\tilde G}(t)$ & $\frac{q^6r^2st - 2q^4r^2st - q^4r^2 + q^2r^2st + q^2rs + q^2rt - st}{q^4(s - r)(-t + r)}$\\
$\mathcal W^-(s) \mathcal {\tilde G}(r) \mathcal {\tilde G}(t)$ & $\frac{s(q - 1)(q + 1)(q^4s^2t - q^2s^2t - q^2s + t)}{q^4(-s + r)(-t + s)}$\\
$\mathcal W^-(t) \mathcal {\tilde G}(r) \mathcal {\tilde G}(s)$ & $\frac{(q - 1)(q + 1)(q^4st^2 - q^2st^2 - q^2t + s)t}{(t - r)q^4(-t + s)}$\\
\hline
$\mathcal W^+(r) \mathcal {\tilde G}(s) \mathcal {\tilde G}(t)$ & $\frac{r(q - 1)(q + 1)(q^4st - q^2rs - q^2rt + st)}{q^4(-s + r)(-t + r)}$\\
$\mathcal W^+(s) \mathcal {\tilde G}(r) \mathcal {\tilde G}(t)$ & $\frac{s^2(q - 1)(q + 1)(q^4t - q^2s - q^2t + t)}{q^4(s - r)(-t + s)}$\\
$\mathcal W^+(t) \mathcal {\tilde G}(r) \mathcal {\tilde G}(s)$ & $\frac{(q - 1)(q + 1)(q^4s - q^2s - q^2t + s)t^2}{(-t + r)q^4(-t + s)}$
\end{tabular}}
\bigskip

 \noindent The overlap ambiguity \eqref{eq:Nine} is resolvable.

\newpage

\noindent Next we evaluate the overlap ambiguity  
\begin{align}
\label{eq:Ten}
\mathcal {\tilde G}(t) \mathcal {\tilde G}(s) \mathcal W^+(r).
\end{align}
 We can proceed in two ways. If we evaluate  $\mathcal {\tilde G}(t) \mathcal {\tilde G}(s)$ first, then we find that \eqref{eq:Ten} is equal to a weighted sum with the following terms and coefficients:

 \bigskip

\centerline{
\begin{tabular}[t]{c|c}
  {\rm term} & {\rm coefficient} 
   \\
\hline
$\mathcal W^-(r) \mathcal {\tilde G}(s) \mathcal {\tilde G}(t)$ & $A'_{s, r}b_{t, r} + a_{s, r}A'_{t, r}$\\
$\mathcal W^-(s) \mathcal {\tilde G}(r) \mathcal {\tilde G}(t)$ & $a'_{s, r}b_{t, s} + A_{s, r}A'_{t, s}$\\
$\mathcal W^-(t) \mathcal {\tilde G}(r) \mathcal {\tilde G}(s)$ & $a'_{s, r}B_{t, s} + A'_{s, r}B_{t, r} + A_{s, r}a'_{t, s} + a_{s, r}a'_{t, r}$\\
\hline
$\mathcal W^+(r) \mathcal {\tilde G}(s) \mathcal {\tilde G}(t)$ & $A'_{s, r}B'_{t, r} + a_{s, r}a_{t, r}$\\
$\mathcal W^+(s) \mathcal {\tilde G}(r) \mathcal {\tilde G}(t)$ & $a'_{s, r}B'_{t, s} + A_{s, r}a_{t, s}$\\
$\mathcal W^+(t) \mathcal {\tilde G}(r) \mathcal {\tilde G}(s)$ & $a'_{s, r}b'_{t, s} + A'_{s, r}b'_{t, r} + A_{s, r}A_{t, s} + a_{s, r}A_{t, r}$
\end{tabular}}
\bigskip

\noindent If we evaluate  $ \mathcal {\tilde G}(s) \mathcal W^+(r)$ first, then we find that \eqref{eq:Ten} is equal to a weighted sum with the following terms and coefficients:
 \bigskip

\centerline{
\begin{tabular}[t]{c|c}
  {\rm term} & {\rm coefficient} 
   \\
\hline
$\mathcal W^-(r) \mathcal {\tilde G}(s) \mathcal {\tilde G}(t)$ & $A'_{t, r}b_{s, r} + a_{t, r}A'_{s, r}$\\
$\mathcal W^-(s) \mathcal {\tilde G}(r) \mathcal {\tilde G}(t)$ & $a'_{t, r}B_{s, t} + A'_{t, r}B_{s, r} + A_{t, r}a'_{s, t} + a_{t, r}a'_{s, r}$\\
$\mathcal W^-(t) \mathcal {\tilde G}(r) \mathcal {\tilde G}(s)$ & $a'_{t, r}b_{s, t} + A_{t, r}A'_{s, t}$\\
\hline
$\mathcal W^+(r) \mathcal {\tilde G}(s) \mathcal {\tilde G}(t)$ & $A'_{t, r}B'_{s, r} + a_{t, r}a_{s, r}$\\
$\mathcal W^+(s) \mathcal {\tilde G}(r) \mathcal {\tilde G}(t)$ & $a'_{t, r}b'_{s, t} + A'_{t, r}b'_{s, r} + A_{t, r}A_{s, t} + a_{t, r}A_{s, r}$\\
$\mathcal W^+(t) \mathcal {\tilde G}(r) \mathcal {\tilde G}(s)$ & $a'_{t, r}B'_{s, t} + A_{t, r}a_{s, t}$
\end{tabular}}
\bigskip

\noindent Referring to the previous two tables, for each row the given coefficients are equal and their common value is displayed below:
 \bigskip

\centerline{
\begin{tabular}[t]{c|c}
  {\rm term} & {\rm coefficient} 
   \\
\hline
$\mathcal W^-(r) \mathcal {\tilde G}(s) \mathcal {\tilde G}(t)$ & $ \frac{r(q - 1)(q + 1)(q^4st - q^2rs - q^2rt + st)}{(s - r)(-t + r)q^2}$\\
$\mathcal W^-(s) \mathcal {\tilde G}(r) \mathcal {\tilde G}(t)$ & $ \frac{s^2(q - 1)(q + 1)(q^4t - q^2s - q^2t + t)}{(s - r)(t - s)q^2}$\\
$\mathcal W^-(t) \mathcal {\tilde G}(r) \mathcal {\tilde G}(s)$ & $ \frac{(q - 1)(q + 1)(q^4s - q^2s - q^2t + s)t^2}{(t - r)q^2(-t + s)}$\\
\hline
$\mathcal W^+(r) \mathcal {\tilde G}(s) \mathcal {\tilde G}(t)$ & $ \frac{q^6st - q^4r^2st - q^4rs - q^4rt + 2q^2r^2st + q^2r^2 - r^2st}{(s - r)(t - r)q^2}$\\
$\mathcal W^+(s) \mathcal {\tilde G}(r) \mathcal {\tilde G}(t)$ & $ \frac{s(q - 1)(q + 1)(q^4t - q^2s^2t - q^2s + s^2t)}{(s - r)(-t + s)q^2}$\\
$\mathcal W^+(t) \mathcal {\tilde G}(r) \mathcal {\tilde G}(s)$ & $ \frac{(q - 1)(q + 1)(q^4s - q^2st^2 - q^2t + st^2)t}{(t - r)q^2(t - s)}$
\end{tabular}}
\bigskip

\noindent The overlap ambiguity \eqref{eq:Ten} is resolvable.
 
\medskip

\noindent We have shown that the overlap ambiguities of type (ii) are resolvable.
\medskip

\noindent The overlap ambiguities of type (iii) are obtained from the overlap ambiguities of type (ii) by applying  the antiautomorphism $\dagger$ or $\sigma \dagger$. Consequently they are resolvable.
\medskip

\noindent It is transparent that the overlap ambiguities of type (iv) are resolvable.
\medskip

\noindent We have shown that all the overlap ambiguities are resolvable.

\newpage

\bigskip

\noindent Paul Terwilliger \hfil\break
\noindent Department of Mathematics \hfil\break
\noindent University of Wisconsin \hfil\break
\noindent 480 Lincoln Drive \hfil\break
\noindent Madison, WI 53706-1388 USA \hfil\break
\noindent email: {\tt terwilli@math.wisc.edu }\hfil\break

\end{document}